\pdfoutput=1
\documentclass[a4paper,11pt]{article}
\usepackage[english]{babel}
\usepackage{amsfonts}
\usepackage{fontenc}
\usepackage{amsthm}
\usepackage{amsmath}
\usepackage{amssymb}
\usepackage{enumerate}
\usepackage{mathtools}
\usepackage{tikz}
\usepackage{tikz-cd}
\usepackage[all]{xy}
\usepackage{graphicx}
\usepackage{amsmath}
\usepackage{adjustbox}
\usepackage{graphicx}
\usepackage{extpfeil}
\usepackage{comment}
\usepackage{mathabx}
\usepackage{float}
\usepackage[labelformat=empty]{caption}
\usepackage[toc]{appendix}
\usepackage[left=2.8cm,top=2.5cm,right=2.8cm,bottom=2.5cm]{geometry}
\usepackage{hyperref}
\hypersetup{
colorlinks=true,
linkcolor=[RGB]{0,0,204},
filecolor=magenta,      
urlcolor=cyan,
citecolor=[RGB]{0,204,0},
}
\usepackage{resizegather}
\usepackage[activate={true, nocompatibility}, final, tracking=true, kerning=true, spacing=true, factor=1100, stretch=10, shrink=10]{microtype}
\usepackage{authblk}
\usepackage{cleveref}

\usepackage{tikz}
\usepackage{dsfont}
\usetikzlibrary{matrix}

\theoremstyle{plain}
\newtheorem{thm}{Theorem}[section]
\newtheorem{coroll}[thm]{Corollary}
\newtheorem{defn}[thm]{Definition}
\newtheorem{lemma}[thm]{Lemma}

\newtheorem{example}[thm]{Example}
\newtheorem{context}[thm]{Context}
    \newtheorem{notn}[thm]{Notation}
\newtheorem{prop}[thm]{Proposition}

\newtheorem{remark}[thm]{Remark}

\newtheorem{thm*}{Theorem}

\theoremstyle{definition}
\newtheorem{construction}[thm]{Construction}

\newcommand{\bseries}[1]{ [\hspace{-0,5mm}[ {#1} ]\hspace{-0,5mm}] }

\newcommand{\dash}{{\operatorname{-}}}

\tikzset{
  symbol/.style={
    draw=none,
    every to/.append style={
      edge node={node [sloped, allow upside down, auto=false]{$#1$}}}
  }
}

\microtypecontext{spacing=nonfrench}

\DeclareMathOperator{\Sym}{Sym}

\DeclareMathOperator{\Filt}{Filt}
\DeclareMathOperator{\Ad}{Ad}
\DeclareMathOperator{\GL}{GL}

\DeclareMathOperator{\wt}{wt}

\DeclareMathOperator{\QCoh}{QCoh}

\DeclareMathOperator{\DF}{\mathbf{DF}}
\newcommand{\cY}{\mathcal{Y}}
\newcommand{\cO}{\mathcal{O}}

\newcommand{\cB}{\mathcal{B}}
\newcommand{\cD}{\mathcal{D}}
\newcommand{\cE}{\mathcal{E}}
\newcommand{\cF}{\mathcal{F}}
\newcommand{\cI}{\mathcal{I}}
\newcommand{\cM}{\mathcal{M}}

\newcommand{\cX}{\mathcal{X}}

\newcommand{\cL}{\mathcal{L}}
\newcommand{\cH}{\mathcal{H}}

\newcommand{\bA}{\mathbb{A}}
\newcommand{\bC}{\mathbb{C}}
\newcommand{\bG}{\mathbb{G}}

\newcommand{\bR}{\mathbb{R}}

\DeclareMathOperator{\Map}{Map}

\newcommand{\hodge}{\on{Hodge}^D}
\newcommand{\Higgs}{\on{Higgs}^D}

\newcommand{\on}{\operatorname}

\newcommand{\Hom}{ \on{Hom}}

\newcommand{\Spec}{\on{Spec}}

\newcommand{\uSpec}{\underline{\on{Spec}}}

\newcommand{\Bun}{ \on{Bun} } 
\newcommand{\Pic}{ \on{Pic} }
\newcommand{\Mhodge}{M_{{\mathrm{Hod}},G,d}^D}

\newcommand{\git}{\mathbin{
  \mathchoice{/\mkern-6mu/}
    {/\mkern-6mu/}
    {/\mkern-5mu/}
    {/\mkern-5mu/}}}

\newcommand{\Fc}{\mathfrak{c}}
\newcommand{\Fg}{\mathfrak{g}}
\newcommand{\Fgl}{\mathfrak{g}\mathfrak{l}}
\newcommand{\Ft}{\mathfrak{t}}
\hypersetup{pdfborder=0 0 0} 

\begin{document}
\title{\textbf{Meromorphic Hodge moduli spaces for reductive groups in arbitrary characteristic}}
\author{Andres Fernandez Herrero and Siqing Zhang}
\date{}
\maketitle

\begin{abstract}
Fix a smooth projective family of curves $C \to S$ and a split reductive group scheme $G$ over a Noetherian base scheme $S$. For any (possibly nonreduced) fixed relative Cartier divisor $D$, we provide a treatment of the moduli of $G$-bundles on the fibers of $C$ equipped with $t$-connections with pole orders bounded by $D$. Under mild assumptions on the characteristics of all the residue fields of $S$, we construct a Hodge moduli space $M_{Hod, G} \to \mathbb{A}^1_S$ for the semistable locus, construct a Harder-Narasimhan stratification, and thus obtain a semistable reduction theorem. If all the fibers of the divisor of poles $D$ are nonempty, then we show that the stack of semistable objects is smooth over $\mathbb{A}^1_{S}$. We also define a Hodge-Hitchin morphism in positive characteristic and prove that it is proper.
\end{abstract}

\tableofcontents

\begin{section}{Introduction}

Vector bundles equipped with Higgs fields or flat connections are important objects arising in Non Abelian Hodge Theory, the Geometric Langlands Program, and Mirror Symmetry \cite{simpson1990nonabelian, beilinson1991quantization, donagi2008geometric, hausel2003mirror}.
Their semistable moduli spaces were first studied in gauge theory \cite{donaldson1987twisted, hitchin1987self, corlette1993non}, and then 
in complex algebraic geometry \cite{nitsure-semistable-pairs-curve,nitsure-moduli-logarithmic, Simpson-repnI}.
It was observed by Deligne and explained by Simpson \cite{simpson-hodge-filtration} that Higgs bundles and flat connections are subsumed by the more general notion of $t$-connections, whose semistable moduli space is called the Hodge moduli space.
This notion turns out to be crucial in the study of twistor structures \cite{simpson1997mixed, hertling2003tt, wei2021cohomology}.
There are natural generalizations of $t$-connections on vector bundles in three directions, each proved to be very fruitful: (1) from $\bC$ to a more general base \cite{laszlo2001hitchin, esnault2020rigid, decaltaldo-zhang-cnah, bogdanova2022canonical, ogus2022crystalline}; 
(2) from $t$-connections to meromorphic $t$-connections \cite{mochizuki2007asymptotic, decataldo-herrero-nahpoles, sabbah2009wild};
(3) from vector bundles to $G$-bundles \cite{simpson2010iterated, donagi2012langlands, hitchin2013hyperkahler}.

This paper concerns the existence and basic properties of the semistable moduli space of meromorphic $t$-connections on $G$-bundles on a projective smooth curve over a Noetherian base, thus encompassing all three generalizations above in the curve case. 
For the general linear group $\GL_n$, the existence of the corresponding Hodge moduli space over an arbitrary base was established by Langer \cite{langer2014semistable,langer-moduli-lie-algebroids}. Thus, the generalization (3) from $\GL_n$ to an arbitrary split reductive group $G$ is an essential aspect of this paper.


We work with a smooth projective family of curves $C$ over a Noetherian scheme $S$. Fix a reductive group scheme $G \to S$ and a relative Cartier divisor $D \subset C$ (possibly empty, and possibly
nonreduced). We consider the moduli stack $\hodge_{G}$ of principal $G$-bundles on the $S$-fibers of $C$ equipped with meromorphic $t$-connections with poles bounded by $D$ (see Subsection \ref{subsection: stack of connections}).

\begin{thm*}[= \Cref{thm: moduli space for hodge general G} + \Cref{prop: theta stratification}] \label{main thm}
    Suppose that the group $G$ satisfies the low height property (LH) (see \Cref{defn: low height property}). Then,
    \begin{enumerate}[(a)]
        \item For any fixed $d \in \pi_1(G)$, the substack $(\hodge_{G, d})^{\mathrm{ss}} \subset \hodge_{G}$ of semistable meromorphic $t$-connections of degree $d$ is an open substack and admits an adequate moduli space $\Mhodge$ (in the sense of \cite{alper_adequate}) which is quasi-projective over $\mathbb{A}^1_S$.
        \item If $G$ is split, then the moduli stack $\hodge_{G}$ admits a $\Theta$-stratification (in the sense of \cite{halpernleistner2018structure}) whose open stratum consists of the semistable meromorphic $t$-connections. Each stratum parametrizes objects in $\hodge_{G}$ such that the canonical Harder-Narasimhan parabolic reduction have a fixed numerical type.
    \end{enumerate}
\end{thm*}

Our approach to this moduli problem is from the point of view of stacks, which we believe shortens and streamlines some of the arguments. A key step in our proof is to show the compatibility of semistability with change of the group $G$ (\Cref{prop: semistability of low height representatons}), for which we generalize an argument by Balaji-Parameswaran \cite{balaji-paramewaran-tensors}. In addition, our treatment clarifies some points in their arguments.

As a consequence of the Harder-Narasimhan theory in \Cref{main thm}(b), the general machinery developed by Alper--Halpern-Leistner--Heinloth \cite{alper2019existence} yields a powerful semistable reduction theorem, analogous to the one proved by Langer \cite{langer-moduli-lie-algebroids} when the group is $\GL_n$. 
We use semistable reduction and the deformation-theoretic results in \S\ref{section: deformation theory} to prove smoothness of the natural morphism $\hodge_{G} \to \mathbb{A}^1_S$ when the relative Cartier divisor has nonempty fibers (\Cref{thm: main theorem smoothness}), generalizing the result for $\GL_n$ proven in \cite{decataldo-herrero-nahpoles}.

In \S\ref{section: hh morphism}, we restrict to the characteristic $p$ case and establish the existence and properness of the so-called Hodge-Hitchin morphism as in \cite{laszlo2001hitchin}, generalizing the analogous results of Langer \cite{langer-moduli-lie-algebroids} when the group is $\GL_n$.
\begin{thm*}[= \Cref{prop: hodge-hitchin} + \Cref{thm: properness of the hodge-hitchin morphism}] \label{second main thm}
    Suppose that the base scheme $S$ is of equi-characteristic $p>0$,  and that $G$ satisfies the low height property (LH). 
    For each fixed degree $d \in \pi_1(G)$, the formation of $p$-curvature induces a proper Hodge-Hitchin morphism $h_{\mathrm{Hod}}: \Mhodge \to A(C',\Omega_{C'/S}^1(D'))\times_S\mathbb{A}^1_S$.
\end{thm*}

The proof of the existence of the Hodge-Hitchin morphism in \Cref{second main thm} relies on a construction (\Cref{lemma: immersion of steinberg bases}) of a closed immersion from the Steinberg-Hitchin base of $G$ to that of a product of $GL_N$'s, which appears to be new.
The proof of properness of the morphism in \Cref{second main thm} relies on a criterion for properness of $\mathbb{G}_m$-equivariant morphisms (\Cref{prop: criterion for properness}) which we believe to be of independent interest.


The results in this article are used by the authors in \cite{herrero2023semistable, herrero2024topology}, and by Mark de Cataldo and the second author in \cite{de-zhang-log}, as the moduli-theoretic foundation for new development on the semistable, cohomological, and logarithmic aspects of the Non Abelian Hodge correspondence for reductive groups $G$ in characteristic $p$. In addition, we hope that our treatment of these moduli spaces will be useful for the community working on these moduli problems.

\noindent \textbf{Acknowledgements.} We would like to thank Mark Andrea de Cataldo, Roberto Fringuelli, Mirko Mauri, and Hao Sun for helpful discussions. This material is based upon work supported by the National Science Foundation under Grant No. DMS-1926686.

    \begin{subsection}{Notation}
        We work over a fixed connected Noetherian scheme $S$. Unless otherwise stated, all schemes and stacks are equipped with a morphism to $S$. For any two schemes $X, Y$ over $S$, we use the notation $X_Y$ to denote the fiber product $X \times_S Y \to Y$ thought of as a $Y$-scheme.
        
        We fix a smooth proper morphism $\pi: C \to S$ with geometrically integral fibers of dimension 1. We also fix the choice of a relative Cartier divisor $D \hookrightarrow C$ (possibly empty). 
        
        Our moduli functor of interest will live over the base $\bA^1_S = \underline{\rm{Spec}}_{S}(\mathcal{O}_{S}[t])$, which is equipped with the contracting linear action of the multiplicative group scheme $\mathbb{G}_{m,S}$ that assigns the coordinate function $t$ weight $-1$ (this means under our convention that it acts with contracting action on $\mathbb{A}^1_S$). For each $\bA^1_S$ scheme $Y$, we denote by $t_Y \in H^0(\cO_{Y})$ the pullback of the coordinate $t$ of $\mathbb{A}^1_{S}$.

        We use $G$ to denote an affine smooth group scheme over $S$. We denote by $\mathfrak{g}= \text{Lie}(G)$ the Lie algebra of $G$ over $S$, which is a vector bundle on $S$. There is a homomorphism $\Ad: G \to \text{GL}(\mathfrak{g})$ called the adjoint representation. We denote by $BG$ the classifying stack of $G$ over $S$. If $X$ is a scheme equipped with an action of $G$ and $E$ is a $G$-bundle on a scheme $Y$, then we denote by $E(X) = E \times^G X \to Y$ the associated fiber bundle with fiber $X$. In general, $E(X)$ is an algebraic space. By descent for affine morphisms, if $X \to S$ is affine, then $E(X)$ is a relatively affine scheme over $Y$. If $X \to S$ is the total space of a vector bundle and $G$ acts linearly, then $E(X) \to Y$ is the total space of a vector bundle.
    \end{subsection}
\end{section}

\begin{section}{The meromorphic Hodge moduli stack}

In this section, we lay the moduli-theoretic foundation for the meromorphic Hodge moduli stack. 
In \Cref{subsection: stack of connections}, we define meromorphic $t$-connections on $G$-bundles, and introduce their moduli stacks, by using the corresponding Atiyah bundles.
In \Cref{subsec: semistability}, we define a notion of semistability.
In \Cref{subsection: change of groups}, we prove some results regarding the morphisms between the moduli stacks induced by change of structural groups.
In \Cref{subsec: moduli spaces}, we show the existence of the corresponding adequate moduli spaces in the case where the characteristics of the base scheme $S$ are bounded below by some properties.
In \Cref{subsection: stability under change of group}, we show that the lower bound is entailed by a more familiar low height condition.

\begin{subsection}{The stack of meromorphic \texorpdfstring{$t$}{t}-connections} \label{subsection: stack of connections}
Let us recall some setup and notation from \cite{herrero2020quasicompactness}. 

\begin{notn}[Maurer-Cartan form]
 Left translation induces an isomorphism $T_{G} \cong \mathfrak{g} \otimes_{\cO_{S}} \mathcal{O}_{G}$ for the tangent sheaf $T_G$ of $G$. This yields a chain of isomorphisms  $\Omega^{1}_{G / S} \otimes_{\cO_{S}} \mathfrak{g} = \text{Hom}_{\mathcal{O}_{G}} ( T_{G}, \mathcal{O}_{G}) \otimes_{\cO_{S}} \mathfrak{g} \cong \text{Hom}_{k}(\mathfrak{g}, \mathfrak{g}) \otimes_{\cO_{S}} \mathcal{O}_{G}$. The invariant $\mathfrak{g}$-valued 1-form on $G$ that corresponds to $\text{id}_{\mathfrak{g}} \otimes 1$ under this isomorphism is called the Maurer-Cartan form. We denote it by $\omega \in   H^0\left(G, \, \Omega_{G / S}^{1} \otimes_{\cO_{S}} \mathfrak{g}\right)$. 
\end{notn}

For any $\mathbb{A}^1_S$-scheme $Y$, we denote by $\omega_Y \in H^0\left(G_Y, \,  \Omega^1_{G_Y/ Y}\otimes_{\cO_{S}} \mathfrak{g} \right)$ the Maurer-Cartan form of the base change $G_{Y} \to Y$. There exists a group scheme $\text{Aff}\left(\Omega^1_{C_Y/ Y}\otimes_{\cO_{S}} \mathfrak{g} \right)$ over $C_Y$ that classifies affine linear transformations of the locally free sheaf $\Omega^1_{C_Y/ S}\otimes_{\cO_{S}} \mathfrak{g}$, see \cite[top of p.7]{herrero2020quasicompactness}. We consider its restriction to the small \'etale site of the scheme $C_Y$. 

\begin{notn}
   We denote by $\phi: G_{C_Y} \longrightarrow \text{Aff}\left(\Omega^1_{C_Y/ Y}\otimes_{\cO_{S}} \mathfrak{g} \right)$ the homomorphism of sheaves of groups in the small \'etale site of $C_Y$ defined as follows. Given an \'etale $C_Y$-scheme $U$ and an element $g \in G_Y(U)$, we define $\phi(U) \, (g) \vcentcolon = Id_{\Omega^1_{U/Y}} \underset{\mathcal{O}_{U}}{\otimes} Ad(g) \, + \, t_Y \cdot (g^{-1})^{*} \omega_Y$, where $(g^{-1})^{*}\omega_Y \in  H^0\left(U,  \, \Omega_{U / Y}^{1}  \otimes_{\cO_{S}}  \mathfrak{g}  \right)$ denotes the pullback of the relative Maurer-Cartan under the inverse $g^{-1}: U \to G_Y$.
\end{notn}
	
Let $E$ be a $G$-bundle on $C_Y$. We denote by $\Ad(E)$ the adjoint vector bundle, associated to the representation $G \to \text{GL}(\mathfrak{g})$. The homomorphism $\phi$ induces an action of $G_{C_Y}$ on the total space of $\Omega^1_{C_Y /Y}\otimes \mathfrak{g}$ by affine linear transformations (for points in the small \'etale site of $C_Y$). We use $\phi$ to form the associated affine bundle $A_{t\dash\text{Conn}, E}:= E \times^{G_Y} \left( \Omega^1_{C_Y/ Y}\otimes_{\cO_{S}} \mathfrak{g} \right)$. Since $G$ is smooth, $E$ can be trivialized over an \'etale cover of $C_Y$, and therefore we only need to know $\phi$ on the small \'etale site of $C_Y$ in order to describe the descent data. 

By construction $A_{t\dash\text{Conn},E}$ is an \'etale torsor for the abelian sheaf $Ad(E) \otimes_{\cO_{C_Y}} \Omega^1_{C_Y/ Y}$ on $C_Y$. This torsor represents a cohomology class $\gamma_{E} \in H^1\left(Ad(E) \otimes_{\cO_{C_Y}} \Omega^1_{C_Y/ Y}\right)$, corresponding to an extension of $\cO_{C_{Y}}$-modules
\[ 0 \to \Ad(E) \otimes \Omega^1_{C \times Y/Y} \to At(E) \to \mathcal{O}_{C_{Y}} \to 0\]
The relative effective Cartier divisor $D$ induces a monomorphism of abelian sheaves $\cO_{C_Y}(-D)\hookrightarrow\cO_{C_Y}$. We denote by $\text{At}^D(E)(-D)$ the pullback extension
\[
\begin{tikzcd}
    0  \ar[r]& \Ad(E) \otimes \Omega^1_{C_Y/Y} \ar[d, symbol = \xlongequal{}] \ar[r] & At^D(E)(-D) \ar[r] \ar[d, symbol = \hookrightarrow]  & \mathcal{O}_{C_{Y}}(-D) \ar[r] \ar[d, symbol = \hookrightarrow] & 0\\
    0 \ar[r] & \Ad(E) \otimes \Omega^1_{C_Y/Y} \ar[r] & At(E) \ar[r] & \mathcal{O}_{C_Y} \ar[r] & 0
\end{tikzcd}
\]
We call the twist $\text{At}^D(E)$ the meromorphic $t$-Atiyah bundle with poles along $D$. The short exact sequence 
\[  0  \to \Ad(E) \otimes \Omega^1_{C_Y/Y}(D) \to  At^D(E) \to  \mathcal{O}_{C_{Y}} \to 0\]
is the meromorphic $t$-Atiyah sequence for the $G$-bundle $E$. It corresponds to an $\Ad(E)\otimes_{\cO_{C_Y}}\Omega^1_{C_Y}(D)$-torsor $A^D_{t\dash\text{Conn},E} \to C_Y$, which  we call the affine bundle of meromorphic $t$-connections.
\begin{defn}
\label{defn: meromorphic t-conn}
A meromorphic $t$-connection on $E$ (with poles bounded by $D$) is a section of the affine bundle $A^D_{t\dash\text{Conn},E} \to C_Y$. Equivalently, it is a splitting of the meromorphic $t$-Atiyah sequence for $E$.
\end{defn}
Let $(E,\nabla)$ be a $G$-bundle on $C_Y$ equipped with a meromorphic $t$-connection $\nabla$. For any other pair $(E', \nabla')$ on $C_Y$, an isomorphism $(E, \nabla) \to (E', \nabla')$ is defined to be an isomorphism $\psi: E \to E'$ that is compatible with the sections $\nabla, \nabla'$ of the corresponding affine bundles.

\begin{example} \label{example: meromorphic connection vector bundles}
    Suppose that $G=\GL_n$ for some positive integer $n$. Then $\GL_n$-bundles are in natural correspondence with vector bundles of rank $n$. Let $Y \to \mathbb{A}^1_S$ be a morphism of $S$-schemes corresponding to a section $t: \cO_Y \to \cO_Y$. Let $\cE$ be a rank $n$-vector bundle on $C_Y$. A meromorphic $t$-connection $\cE$ can be translated into the data of a $\cO_Y$-linear morphism of sheaves $\nabla: \cE(-D) \to \cE \otimes_{\cO_{C_Y}} \Omega^1_{C_Y/Y}$ satisfying the following: 
    
    \noindent \underline{$t$-Leibniz rule:} For any local sections $f \in H^0(U, \cO_U)$, $s \in H^0(U, \cE(-D))$ over an open $U \subset C_Y$, we have $\nabla(f \cdot s) = f \cdot \nabla s + t \cdot df \otimes \iota_{\cE}(s)$, where $d: \cO_{C_Y} \to \Omega^1_{C_Y/Y}$ denotes the exterior differential and $\iota: \cE(-D) \hookrightarrow \cE$ is the natural inclusion.
\end{example}

Similarly as in \cite[\S3.2]{herrero2020quasicompactness}, the bundle of connections $A^D_{t\dash\text{Conn},E}$ behaves well with respect to base change on $Y$. In other words, for all morphisms $T \to Y$, the torsor of meromorphic $t$-connections $A^D_{t\dash \text{Conn},E_T} \to C_T$ for the base change $E_T = E \times_{C_Y} C_T$ is canonically isomorphic to the base change $A^D_{t\dash \text{Conn},E} \times_{C_Y} C_T$. This allows to pull back meromorphic $t$-connections.
\begin{defn}
We denote by $\hodge_{G} \to \mathbb{A}^1_S$ the pseudofunctor from $(\text{Sch}/\bA^1_{S})^{op}$ into groupoids that sends an $\bA^1_S$-scheme $Y$ into the groupoid of pairs $(E, \nabla)$, where $E$ is a $G$-bundle on $C_Y$ and $\nabla$ is a meromorphic $t$-connection on $E$.
\end{defn}

Let $\Bun_{G}(C) \to S$ denote the relative moduli stack of $G$-bundles for the family of curves $C\to S$. The stack $\Bun_{G}(C)$ is a smooth algebraic stack over $S$ with affine diagonal, by \cite[Thm. 1.2]{hall-rydh-tannaka} and standard deformation theory of $G$-bundles for the smoothness (see the proof of \cite[Prop. 1]{heinloth-uniformization}). There is a forgetful morphism of pseudofunctors $F:  \hodge_{G} \to \Bun_{G}(C) \times_{S} \bA^1_S$ given by $(E, \nabla) \mapsto (E,t)$.
\begin{prop}
\label{prop: forget is affine and finite type}
The morphism $F:  \hodge_{G} \to \Bun_{G}(C) \times_{S} \bA^1_S$ is affine and of finite type. In particular, $ \hodge_{G}$ is an algebraic stack with affine diagonal and locally of finite type over $\bA^1_S$.
\end{prop}
\begin{proof}
Let $T$ be a scheme, and choose a morphism $T \to \Bun_{G}(C) \times_{S} \mathbb{A}^1_{S}$ corresponding to a morphism $T \to \mathbb{A}^1_{S}$ and a $G$-bundle $E$ on $C_{T}$. The fiber product $\hodge_{G} \times_{\Bun_{G}(C) \times_{S} \mathbb{A}^1_{S}} T$ is the functor from $T$-schemes into sets that sends a $T$-scheme $Y$ into the set of sections of the morphism $A_{t-\text{Conn}, E}^D \times_{T} Y \to C_{Y}$. Since $A_{t-\text{Conn}, E}^D \to C_{T}$ is affine and of finite type, the same holds for $\hodge_{G} \times_{\Bun_{G}(C) \times_{S} \mathbb{A}^1_{S}} T$ (by \cite[Thm. 2.3 (i)]{hall-rydh-hilbert-quot} and \cite[Thm. 1.1]{olsson-homstacks}+\cite[Thm. B]{rydh-noetherian-approximation}). Since $\Bun_{G}(C) \times_{S} \mathbb{A}^1_{S}$ has affine diagonal and is locally of finite type over $\mathbb{A}^1_{S}$, the same holds for $\hodge_{G}$.
\end{proof}

\begin{defn}
The stack of meromorphic $G$-Higgs bundles $\Higgs_{G}$ is the fiber product $0_{S} \times_{\mathbb{A}^1_{S}} \hodge_{G}$. For any $S$-scheme $T$, a $T$-point of $\Higgs_{G}$ is a pair $(E, \varphi)$ consisting of a $G$-bundle $E$ on $C_{T}$ and a section $\varphi \in H^0(\text{Ad}(E) \otimes \Omega^1_{C_{T}/T}(D))$.
\end{defn}

The naive formula ``$(t,\nabla)\mapsto t\nabla$" should define a $\bG_m$-action on $\hodge_{G}$ with zero-limits.
The following proposition gives a precise proof of this fact.
Recall that $\mathbb{G}_{m,S}$ acts linearly on $\mathbb{A}^1_{S}$ by the standard contracting action.
\begin{prop} \label{prop: g_m action and limits} \quad
\begin{enumerate}[(a)]
    \item There is an action $\mathbb{G}_{m,S} \times_{S} \hodge_{G} \to \hodge_{G}$ such that the morphism $\hodge_{G} \to \mathbb{A}^1_{S}$ is equivariant.

    \item Furthermore, for any $S$-scheme $T$ and any point $(E, \nabla) \in \hodge_{G}(T)$, the orbit morphism $\mathbb{G}_{m,T}  \to \hodge_{G}$ extends to a morphism $\mathbb{A}^1_{T} \to \hodge_{G}$. 
\end{enumerate}
 
\end{prop}
\begin{proof}
\noindent (a). We shall build the action on the fibers of the representable morphism $\hodge_{G} \to \Bun_{G}(C)\times_{S} \mathbb{A}^1_S \to \Bun_{G}(C)$ so that the composition is equivariant for the trivial action on $\Bun_{G}(C)$. Hence, we don't have to get into the subtleties of defining the coherence isomorphisms for group actions on stacks. Let $T$ be an $S$-scheme and choose a morphism $T \to \Bun_{G}(C)$ corresponding to a $G$-bundle $E$ on $C_T$. We want to define an action of $\mathbb{G}_{m}(T)$ on the $T$-points of the base change $\hodge_{G} \times_{\Bun_{G}(C)} T$, which correspond to pairs $(f, \nabla)$ of an $S$-morphism $f: T \to \mathbb{A}^1_{S}$ and an $f^*(t)$-connection $\nabla$ on $E$. Let $a\in \mathbb{G}_{m}(T)$. We define $a\cdot f$ via the fixed action of $\mathbb{G}_m(T)$ on $\mathbb{A}^1_{S}(T)$. Consider the $f^*(t)$-meromorphic Atiyah sequence for $E$
\[ 0 \to \text{Ad}(E)(D) \otimes \Omega^1_{C_T/T} \xrightarrow{i} \text{At}^{D}(E) \xrightarrow{q} \cO_{C_{T}} \to 0\]
By construction, the $(a\cdot f)^*(t)$-meromorphic Atiyah sequence for $E$ is given by
\[ 0 \to \text{Ad}(E)(D) \otimes \Omega^1_{C_T/T} \xrightarrow{i} \text{At}^{D}(E) \xrightarrow{\frac{1}{a}q} \cO_{C_{T}} \to 0\]
If we view $\nabla$ as a splitting of the $f^*(t)$-meromorphic Atiyah sequence, then $a \cdot \nabla$ is a splitting of the $(a\cdot f)^*(t)$-meromorphic Atiyah sequence. We then set $a \cdot (f, \nabla) := (a \cdot f, a\cdot \nabla)$. This yields a well-defined action of $\mathbb{G}_m(T)$ on $(\hodge_{G} \times_{\Bun_{G}(C)} T)(T)$ compatible with base change.

\noindent (b). We replace $S$ by $T$ and assume that $S=T$. The morphism $S \to \hodge_{G}$ corresponds to a section $f: S \to \mathbb{A}^1_{S}$ and an $f^*(t)$-meromorphic connection $(E, \nabla)$ on a $G$-bundle $E$.  Let $a$ denote the coordinate for the multiplicative group, so that $\mathbb{G}_{m,S} = \uSpec_{S}(\cO_{S}[a,a^{-1}])$,. The action of $\mathbb{G}_{m,S}$ on $\mathbb{A}^1_{S}$ yields an orbit morphism $o: \mathbb{G}_{m,S} \to \mathbb{A}^1_{S}$.
This morphism extends uniquely to a $\mathbb{G}_{m,S}$-equivariant morphism $\overline{o}: \mathbb{A}^1_{S} \to \mathbb{A}^1_{S}$ given by $a \mapsto a\cdot f$, where we use $a$ to denote the coordinate of the left-most copy of $\mathbb{A}^1_{S}$. Let $E|_{C_{\mathbb{A}^1_{S}}}$ denote the pullback of the $G$-bundle $E$ under the projection morphism $C_{\mathbb{A}^1_{S}} \to C$. Consider the $\overline{o}^*(t)$-meromorphic Atiyah sequence for $E_{C_{\mathbb{A}^1_{S}}}$
\[ 0 \to \text{Ad}(E)(D)_{\mathbb{A}^1_{S}} \otimes \Omega^1_{C_{\mathbb{A}^1_{S}}/\mathbb{A}^1_{S}} \to \text{At}^{D}(E_{C_{\mathbb{A}^1_S}}) \to \cO_{C_{\mathbb{A}^1_{S}}} \to 0\]
By construction, this is naturally a $\mathbb{G}_{m,S}$-equivariant extension (equivalently $\mathbb{Z}$-graded extension of graded $\cO_{C_{\mathbb{A}^1_{S}}}$-modules).
If we denote by $\gamma_{E}^{D} \in H^1(C, \text{Ad}(E)(D) \otimes \Omega^1_{C/S})$ the obstruction to the splitting of the $f^*(t)$-meromorphic Atiyah sequence for $E$, then $\gamma_{E}^{D} \otimes a \in H^1(C, \text{Ad}(E)(D) \otimes \Omega^1_{C/S}) \otimes \cO_{S}[a] \cong H^1(C_{\mathbb{A}^1_{S}}, \text{Ad}(E)(D)_{\mathbb{A}^1_{S}} \otimes \Omega^1_{C_{\mathbb{A}^1_{S}}/\mathbb{A}^1_{S}})$ is the obstruction for the $\overline{o}^*(t)$-meromorphic Atiyah sequence for $E_{C_{\mathbb{A}^1_{S}}}$ (as an ungraded extension). Since $\nabla$ yields a splitting of the $f^*(t)$-meromorphic Atiyah sequence, we have $\gamma_{E}^{D} =0$, and hence $\gamma_{E}^{D} \otimes a =0$. Since $\mathbb{G}_{m,S}$ is linearly reductive, it follows that there is an equivariant splitting of the $\overline{o}^*(t)$-meromorphic Atiyah sequence for $E_{C_{\mathbb{A}^1_{S}}}$. 
\footnote{Indeed, fix a non-equivariant splitting corresponding to a section $\cO_{C_{\mathbb{A}^1_{S}}} \to \text{At}^{D}(E_{C_{\mathbb{A}^1_{S}}})$, which is not necessarily a morphism of $\cO_{C_{\mathbb{A}^1_{S}}}$-graded modules. We can always make this splitting graded by projecting the image of the unit section $1$ to the $0$-graded piece in $\text{At}^{D}(E_{C_{\mathbb{A}^1_S}})$.} This equivariant splitting allows us to view $\overline{o}^*(t)$-meromorphic connections of $E_{C_{\mathbb{A}^1_{S}}}$ as sections of $H^0(C_{\mathbb{A}^1_{S}}, \text{Ad}(E)(D)_{\mathbb{A}^1_{S}} \otimes \Omega^1_{C_{\mathbb{A}^1_{S}}/\mathbb{A}^1_{S}}) \cong H^0(C,\text{Ad}(E)(D) \otimes \Omega^1_{C/S}) \otimes \cO_{S}[a]$. If we view $\nabla \in H^0(C, \text{Ad}(E)(D) \otimes \Omega^1_{C/S})$, then away from $0$ the $\mathbb{G}_m$-action sends $a$ to the section $a \cdot \nabla$. Hence, the $\overline{o}^*(t)$-meromorphic connection corresponding to $\nabla \otimes a \in H^0(C_{\mathbb{A}^1_{S}}, \text{Ad}(E)(D)_{\mathbb{A}^1_{S}} \otimes \Omega^1_{C_{\mathbb{A}^1_{S}}/\mathbb{A}^1_{S}})$ yields the desired extension of the orbit morphism to $\mathbb{A}^1_{S} \to \hodge_{G}$.
\end{proof}

\begin{notn}[Connected components]
    Suppose that $G$ is a split group scheme over $S$. Let $\pi_1(G)$ denote the algebraic fundamental group of $G$, which is defined as the coweight lattice of a maximal split torus $T \subset G$ quotiented by the corresponding coroot lattice. Then any geometric point $\Spec(k) \to \Bun_{G}(C)$ has an associated degree in $\pi_1(G_{k}) = \pi_1(G)$ \cite{hoffmann-connected-components}. By the definition of $d$ in \cite[Thm. 5.8]{hoffmann-connected-components} in terms of degrees of line bundles associated to Borel reductions, and by the main result in \cite{drinfeld-simpson} it follows that there is an open and closed substack $\Bun_{G}(C)_d \subset \Bun_{G}(C)$ parametrizing $G$-bundles of degree $d$. Since all the fibers of the smooth morphism $\Bun_{G}(C)_d \to S$ are geometrically connected \cite{hoffmann-connected-components} and the base $S$ is connected, it follows that $\Bun_{G}(C)_d$ is connected, and so $\Bun_{G}(C) = \bigsqcup_{d \in \pi_1(G)} \Bun_{G}(C)_d$ is the decomposition of $\Bun_{G}(C)$ into its connected components.
\end{notn}

\begin{defn}
   For any $d \in \pi_1(G)$, we set $\hodge_{G,d}$ to be the open and closed preimage of $\Bun_{G}(C)_d \times_S \mathbb{A}^1_S$ in $\hodge_{G}$ under the forgetful morphism. 
\end{defn}
\end{subsection}

\begin{subsection}{Semistability}\label{subsec: semistability}
\begin{context}
For this subsection we assume that $G$ is a reductive group scheme over $S$. 
\end{context}
Let $k$ be an algebraically closed field over $\mathbb{A}^1_{S}$. Choose a point $Spec(k) \to \hodge_{G}$, corresponding to a pair $(E, \nabla)$ of a $G$-bundle on $C_{k}$ and a meromorphic $t$-connection $\nabla$. Let $P \subset G_k$ be a parabolic subgroup, and let $E_{P}$ be a reduction of structure group of $E$ to $P$. We denote by $\Ad: P \to \GL(\mathfrak{p})$ the adjoint representation of $P$ acting on its Lie algebra $\mathfrak{p}$, and denote by $\Ad(E_{P})$ the associated adjoint bundle. By definitions there is a morphism of meromorphic $t$-Atiyah sequences
\[
\begin{tikzcd}
    0  \ar[r]& \Ad(E_P) \otimes \Omega^1_{C_k/k}(D) \ar[d, symbol = \hookrightarrow] \ar[r] & At^D(E_P) \ar[r] \ar[d, symbol = \hookrightarrow]  & \mathcal{O}_{C_{k}} \ar[r] \ar[d, symbol = \xlongequal{}] & 0\\
    0 \ar[r] & \Ad(E)(D) \otimes \Omega^1_{C_k/k} \ar[r] & At^D(E) \ar[r] & \mathcal{O}_{C_k} \ar[r] & 0
\end{tikzcd}
\]
The meromorphic $t$-connection $\nabla$ corresponds to a splitting of the bottom row.
\begin{defn} \label{defn: compatible parabolic reduction}
We say that the parabolic reduction $E_{P}$ is compatible with $\nabla$ if the splitting $\cO_{C_k} \to At^D(E)$ factors through $At^D(E_P)$.
\end{defn}
The following lemma will be useful in the proof of \Cref{prop: hn boundedness}.
\begin{lemma} \label{lemma: extension of admissible parabolic reduction is admissible}
Let $E_P$ be a parabolic reduction of $E$ compatible with $\nabla$. For any larger parabolic subgroup $P \subset P' \subset G_k$, the extension of structure group $E_{P'}$ to $P'$ is compatible with $\nabla$.
\end{lemma}
\begin{proof}
There are inclusions of meromorphic $t$-Atiyah bundles $\text{At}^{D}(E_{P}) \subset \text{At}^{D}(E_{P'}) \subset \text{At}^{D}(E)$. If the splitting $\nabla: \cO_{C_k} \to \text{At}^{D}(E)$ factors through $\text{At}^{D}(E_{P})$, then it also factors through $\text{At}^{D}(E_{P'})$.
\end{proof}

All parabolic subgroups $P \subset G_k$ admit Levi decompositions $P = L \ltimes U$, where $L$ is a split Levi subgroup, and $U$ is the unipotent radical of $P$. The conjugation action of $L$ induces a representation of $L$ on the Lie algebra $\mathfrak{u} = \text{Lie}(U)$. When restricted to the maximal central torus $Z_{L}^{0} \subset L$, the representation $\mathfrak{u}$ breaks into weight spaces. 
\begin{defn}
A cocharacter $\lambda: \mathbb{G}_m \to Z_{L}^{\circ}$ is $P$-dominant if all $\lambda$-weights of $\mathfrak{u}$ are positive.
\end{defn}

The adjoint action of $Z_{L}^{\circ}$ on $\mathfrak{g}_k$ induces an integral bilinear trace-pairing $\text{tr}_{\mathfrak{g}}$ on cocharacters of $Z_{L}^{\circ}$ (cf. \cite[Defn. 4.7]{gauged_theta_stratifications}). This yields a morphism $\text{tr}_{\mathfrak{g}}$ from cocharacters of $Z_{L}^{\circ}$ into characters of $Z_{L}^{\circ}$.
\begin{defn}
     We say that a character $\chi: P \to \mathbb{G}_m$ is $P$-dominant if it is of the form $\text{tr}_{\mathfrak{g}}(\lambda)$ for some $P$-dominant cocharacter $\lambda$.
\end{defn}

\begin{remark}
There is an alternative description of the $P$-dominant characters as follows. Choose a Borel subgroup $T \subset B \subset P$, yielding a set of simple roots $(\alpha_i)_{i \in I}$. Let $I_P$ denote the set of $i \in I$ such that the corresponding root subgroup $U_{\alpha_i}$ is not contained in the Levi subgroup $L$. Then, under the natural inclusion $X^*(Z_{L}^{\circ}) \hookrightarrow X^*(T)_{\mathbb{Q}}$, a character is $P$-dominant if it is a nonnegative linear combination $\sum_{i \in I_P} c_i \omega_i$ with $c_i\geq0$ and $\omega_i$ the fundamental weight dual to the coroot $\alpha_i^{\vee}$.
\end{remark}

\begin{defn} \label{defn: classical semistability}
The pair $(E, \nabla)$ is called semistable if for all parabolic subgroups $P \subset G_k$, all parabolic reductions $E_{P}$ compatible with $\nabla$, and all $P$-dominant characters $\chi: P \to \mathbb{G}_m$, the associated line bundle $E_P(\chi)$ has nonpositive degree on $C_k$.
\end{defn}

\begin{example} \label{example: semistability of meromorphic conenctions on vector bundles}
Suppose that $G = \GL_n$. For any geometric point $t: \Spec(k) \to \mathbb{A}^1_{S}$, a point in $\hodge_{\GL_n}(k)$ consists of a pair $(\cE, \nabla)$, where $\cE$ is a rank $n$-vector bundle on $C_{k}$ and $\nabla$ is a $k$-linear morphism $\nabla: \cE(-D) \to \cE \otimes \Omega^1_{C_{k}/k}$ satisfying the $t$-Leibniz rule (see \Cref{example: meromorphic connection vector bundles}).
Parabolic reductions of $\cE$ correspond to filtrations by subbundles
\[ 0 = \cE_0 \subset \cE_1 \subset \cE_2 \subset \ldots \subset \cE_l = \cE\]
Such a filtration is $\nabla$-compatible if and only if $\nabla(\cE_i(-D)) \subset \cE_i \otimes \Omega^1_{C_k/k}$ for all $i$. In view of this, we define a sub-$t$-connection $\cF$ of $(\cE, \nabla)$ to be a vector subbundle $\cF \subset \cE$ such that $\nabla(\cF(-D)) \subset \cF\otimes \Omega^1_{C_k/k}$. 
Then, a meromorphic $t$-connections $(\cE, \nabla) \in \hodge_{\GL_n}(k)$ is semistable if and only if for all nonzero sub-$t$-connection $(\cF, \nabla') \subset (\cE, \nabla)$, we have $\frac{\text{deg}(\cF)}{\text{rank}(\cF)} \leq \frac{\text{deg}(\cE)}{\text{rank}(\cE)}$. 
This characterization of semistability in the case of vector bundles without $t$-connection is well-known, for example see \cite[Cor. 1]{hyeon2004note} (whose definition of semistability is equivalent to ours by \cite[Lem. 2.1]{ramanathan-stable}). The analogous characterization of semistability in the context of $t$-connections can be proven by the same argument with minor modifications.  
\end{example}
\end{subsection}

\begin{subsection}{Change of group} \label{subsection: change of groups}
Let $\rho: G \to H$ be a homomorphism of smooth group schemes over $S$. The associated bundle construction induces a homomorphism of stacks
\[  \Bun_{G}(C) \times_{S} \mathbb{A}^1_{S} \to \Bun_{H}(C) \times_{S} \mathbb{A}^1_S, \; \quad (E, f) \mapsto (\rho_*(E), f)\]
Let $Y$ be a scheme over $\mathbb{A}^1_S$. For any $G$-bundle $E$ on $C_Y$, there is an induced morphism of adjoint bundles $\text{Ad}(\rho): \text{Ad}(E) \to \text{Ad}(\rho_*(E))$ which fits into a morphism of extensions
\[
\begin{tikzcd}
    0  \ar[r]& \Ad(E) \otimes \Omega^1_{C_Y/Y}(D) \ar[d, "\text{Ad}(\rho) \otimes \text{id}", labels= left] \ar[r] & At^D(E) \ar[r] \ar[d]  & \mathcal{O}_{C_{Y}} \ar[r] \ar[d, symbol = \xlongequal{}] & 0\\
    0 \ar[r] & \Ad(\rho_*(E)) \otimes \Omega^1_{C_Y/Y}(D) \ar[r] & At^D(\rho_*(E)) \ar[r] & \mathcal{O}_{C_Y} \ar[r] & 0
\end{tikzcd}
\]
For every meromorphic $t$-connection $\nabla$ on $E$ corresponding to a splitting of the top row, composition with the middle vertical morphism yields a meromorphic $t$-connection $\rho_*(\nabla)$ on $\rho_*(E)$. This induces a morphism of stacks
\[ \rho_*: \hodge_{G} \to \hodge_{H}, \; \; (E, \nabla) \mapsto (\rho_*(E), \rho_*(\nabla)) \]
\begin{lemma} \label{lemma: change of group under closed immersion is affine}
If $\rho: G \to H$ is a closed immersion of reductive group schemes over $S$, then $\rho_*: \hodge_{G} \to \hodge_{H}$ is affine and of finite type.
\end{lemma}
\begin{proof}
Consider the following diagram of $S$-stacks, where $Z$ is defined so that the square is Cartesian:
\[
\begin{tikzcd}
    \hodge_{G}  \ar[r, "i"]& Z \ar[d] \ar[r] & \hodge_{H} \ar[d] \\
    & \Bun_{G}(C)\times_{S}\mathbb{A}^1_S \ar[r, "\rho_*\times \mathrm{id}"] & \Bun_{H}(C)\times_S\mathbb{A}^1_S.
\end{tikzcd}
\]
The morphism $\rho_*: \Bun_{G}(C) \to \Bun_{H}(C)$ is affine of finite type (cf. the proof of \cite[Prop. 2.3]{herrero2020quasicompactness}). Hence, the base change $Z \to \hodge_{H}$ is affine and of finite type. We shall conclude by showing that $i: \hodge_{G} \to Z$ is a closed immersion. 
Let $T$ be an $\mathbb{A}^1_{S}$-scheme, and choose a morphism $T \to Z$ represented by a $G$-bundle $E$ on $C_{T}$ and a $t$-connection $\nabla$ on $\rho_*(E)$. 
The fiber product $T \times_{Z} \hodge_{G}$ is the $T$-functor that sends a $T$-scheme $Y$ to the set of $t$-connections $\widetilde{\nabla}$ on the base change $E|_{C_{Y}}$ such that $\rho_*(\widetilde{\nabla}) = \nabla|_{C_{Y}}$.
Since $\rho$ is a closed immersion, the vertical arrows in the following diagram are injective
\[
\begin{tikzcd}
    0  \ar[r]& \Ad(E) \otimes \Omega^1_{C_T/T}(D) \ar[d, "\text{Ad}(\rho) \otimes \text{id}", labels= left] \ar[r] & At^D(E) \ar[r] \ar[d]  & \mathcal{O}_{C_{T}} \ar[r] \ar[d, symbol = \xlongequal{}] & 0\\
    0 \ar[r] & \Ad(\rho_*(E)) \otimes \Omega^1_{C_T/T}(D) \ar[r] & At^D(\rho_*(E)) \ar[r] & \mathcal{O}_{C_T} \ar[r] 
 & 0.
\end{tikzcd}
\]
The quotient of $\text{Ad}(E) \hookrightarrow \text{Ad}(\rho_*(E))$ is the associated vector bundle $E(\mathfrak{h}/\mathfrak{g})$, and so it follows that $\text{At}^{D}(E) \subset \text{At}^{D}(\rho_*(E))$ is a subbundle. The $t$-connection $\nabla$ corresponds to a splitting $\nabla: \mathcal{O}_{C_{T}} \to \text{At}^{D}(\rho_*(E))$ of the bottom row. For any $Y \to T$, the base change $\nabla_{C_{Y}}: \mathcal{O}_{C_{Y}} \to \text{At}^{D}(\rho_*(E))|_{C_{Y}} \cong \text{At}^{D}(\rho_*(E|_{C_{Y}}))$ admits a lift $\widetilde{\nabla}$ on $E_{C_{Y}}$ if and only if $\nabla_{C_{Y}}$ factors through the subbundle $\text{At}^{D}(E)|_{C_{Y}} \subset \text{At}^{D}(\rho_*(E))|_{C_{Y}}$. If we denote by $s$ the composition $s: \cO_{C_{T}} \xrightarrow{\nabla} \text{At}^{D}(\rho_*(E)) \to \text{At}^{D}(\rho_*(E))/\text{At}^{D}(E)$, then the base change $T \times_{Z} \hodge_{G}$ is the subfunctor of $T$ consisting of those $Y \to T$ such that the base change $s|_{C_{Y}}$ is identically $0$. This is represented by a closed subscheme of $T$, as desired.
\end{proof}

\begin{lemma} \label{lemma: change of group for etale isogenies}
Suppose that $\rho: G \to H$ is an isogeny of reductive group schemes whose kernel $K$ is a finite \'etale group scheme of multiplicative type over $S$. Then the morphism of stacks $\rho_*: \hodge_{G} \to \hodge_{H}$ is quasifinite and proper with relative stabilizers that are \'etale and of mutiplicative type.
\end{lemma}
\begin{proof}
By \cite[Lem. A.4]{gauged_theta_stratifications}, the morphism $\rho_*: \Bun_{G}(C) \times_{S} \mathbb{A}^1_{S} \to \Bun_{H}(C) \times _{S} \mathbb{A}^1_{S}$ is quasifinite and proper. By working \'etale locally on $S$, we can assume that the kernel $K$ of $\rho$ is of the form $\prod_{i} \mu_{n_i}$ for some positive integers $n_i$ that are coprime to all characteristics of residue fields of $S$. The proof of \cite[Lem. A.4]{gauged_theta_stratifications} identifies the relative stabilizers of $\rho_*: \Bun_{G}(C) \to \Bun_{H}(C)$ with a closed subgroup scheme of $\prod_i \mu_i$, so it is multiplicative an \'etale. Therefore, the base change $(\Bun_{G}(C) \times \mathbb{A}^1_{S}) \times_{\Bun_{H}(C) \times_{S} \mathbb{A}^1_{S}} \hodge_{H} \to \hodge_{H}$ is quasifinite and proper with \'etale relative stabilizers of multiplicative type. Let $Z$ be the source. 
We shall conclude the proof by showing that the induced morphism $\hodge_{G} \to Z$ is an isomorphism. Choose an $\mathbb{A}^1_{S}$-scheme $T$ and a $G$-bundle $T$ on $C_{T}$. Set $\mathfrak{g}:= \text{Lie}(G)$ and $\mathfrak{h}:= \text{Lie}(H)$. The morphism $d\rho: \mathfrak{g} \to \mathfrak{h}$ is surjective because $\rho$ is faithfully flat.
Moreover, the fact that $K$ is \'etale over $S$ implies that $d\rho$ is injective. It follows that the induced morphism of meromorphic $t$-Atiyah sequences
\[
\begin{tikzcd}
    0  \ar[r]& \Ad(E) \otimes \Omega^1_{C_Y/Y}(D) \ar[d, "\text{Ad}(\rho) \otimes \text{id}", labels= left] \ar[r] & At^D(E) \ar[r] \ar[d]  & \mathcal{O}_{C_{Y}} \ar[r] \ar[d, symbol = \xlongequal{}] & 0\\
    0 \ar[r] & \Ad(\rho_*(E)) \otimes \Omega^1_{C_Y/Y}(D) \ar[r] & At^D(\rho_*(E)) \ar[r] & \mathcal{O}_{C_Y} \ar[r] & 0
\end{tikzcd}
\]
is an isomorphism. Hence, the assignment $(E, \nabla) \mapsto (\rho_*(E), \rho_*(\nabla))$ induces a bijection between meromorphic $t$-connections on $E$ and meromorphic $t$-connections on $\rho_*(E)$.
\end{proof}

\begin{notn}[Derived subgroup, abelianization, and center] \label{notn: derived subgroup, abelianization and center}
    Let $G$ be a reductive group scheme over $S$. We denote by $\cD(G)$ the derived subgroup scheme of $G$, as in \cite[Thm. 5.3.1]{conrad_reductive}. Recall that $\cD(G) \subset G$ is a semisimple normal subgroup scheme over $S$, and the quotient $G_{ab} := G/\cD(G)$ is a torus over $S$. We denote by $Z_G \subset G$ the center of $G$. This is a closed subgroup scheme of $G$, which is a group scheme of multiplicative type over $S$.
\end{notn} 

\begin{notn}[The isogeny $i$] \label{defn: isogeny}
    Let $G$ be a reductive group scheme over $S$. We shall consider the homomorphism $\text{Ad}\times q: G \to \GL(\mathfrak{g}) \times G_{ab}$, where $q: G \twoheadrightarrow G_{ab}$ is the quotient morphism. By construction, the kernel $K = Z_{\cD(G)}$ of $\text{Ad}\times q$ is a finite group scheme of multiplicative type over $S$. We use the notation $\overline{G} := G/Z_{\cD(G)}$ for simplicity. We denote by $i$ the central isogeny $i: G \to  \overline{G}$. 
\end{notn}

\begin{prop} \label{prop: semistability under etale isogeny}
Let $G$ be a reductive group scheme. Then the followings hold:
\begin{enumerate}[(a)]
    \item The isogeny $i: G \to \overline{G}$ preserves and reflects semistability. More precisely, for any geometric point $p$ of $\hodge_{G}$, we have that $p$ is semistable if and only $i_*(p)$ is a semistable geometric point of $\hodge_{\overline{{G}}}$.

    \item The morphism $j_*: \hodge_{G} \to \hodge_{G/Z_{G}}$ induced by the quotient $j: G \to G/Z_{G}$ preserves and reflects semistability as well.
\end{enumerate}
\end{prop}
\begin{proof}
\indent \textbf{(a).} Let $k$ be an algebraically closed field, and choose a point $p = (E, \nabla) \in \hodge_{G}(k)$ with image $i_*(p)= (i_*(E), i_*\nabla) \in \hodge_{\overline{G}}(k)$. Let $\overline{P}$ be a parabolic subgroup of $\overline{G}_k$. It is associated to a cocharacter $\overline{\lambda}: \mathbb{G}_m \to \overline{G}_k$, which can be lifted (up to a power) to a cocharacter $\lambda: \mathbb{G}_m \to G_k$. Then the preimage $P:= i^{-1}(\overline{P})$ is the smooth parabolic subgroup of $G_k$ corresponding to $\lambda$ (cf. the first paragraph of the proof of \cite[Prop. 3.1]{biswas-holla-hnreduction}).
Let $i_P: P \to \overline{P}$ denote the restriction. 
For any parabolic reduction $E_P$ of $E$, the associated bundle $(i_P)_*(E_P)$ is a reduction of $i_*(E)$ to the parabolic $\overline{P}$. In fact, the assignment $(i_P)_*(-)$ induces a correspondence between $P$-reductions of $E$ and $\overline{P}$-reductions of $i_*(E)$  \cite[Lem. 2.1]{biswas-holla-hnreduction}. By definition, $P$-dominant characters vanish on the center $Z_{G_k} \supset Z_{\cD(G)_k}$, and therefore any $P$-dominant character of $P$ factors uniquely through $\overline{P}$. By the definition of semistability, in order to conclude the proof of (a) it suffices to show that a $P$-reduction $E_P$ is $\nabla$-compatible if and only if $(i_P)_*(E_P)$ is $i_*(\nabla)$-compatible. For this we note that the kernel $\text{Lie}(Z_{\cD(G)_k})$ of $\text{Lie}(i): \text{Lie}(G) \to \text{Lie}(\overline{G})$ is contained in the Lie algebra of the parabolic subgroup $\text{Lie}(P)$. This implies that $\text{Lie}(i)^{-1}(\text{Lie}(\overline{P})) = \text{Lie}(P)$, and so it follows that $At^D(E_P)$ is the preimage of $At^D(i_*(E_P))$ under the morphism $At^D(E) \to At^D(i_*(E))$. We conclude that $E_P$ is $\nabla$-compatible if and only if $(i_P)_*(E_P)$ is $i_*(\nabla)$-compatible.

\noindent \textbf{(b).} The same argument as in part (a) applies. We just need to replace $Z_{\cD(G)}$ with the center $Z_{G}$.
\end{proof}

For simplicity, in this paper we often restrict our attention to group schemes such that $\text{Ad}\times q$ is well-behaved in the following sense.

\begin{defn}[Property (Z)] \label{defn: central property}
    We say that $G$ satisfies property (Z) if the kernel $Z_{\cD(G)}$ of the isogeny $i: G\to \overline{G}$ is \'etale over $S$.
\end{defn}

\begin{defn}[Property (Em)] \label{defn: embedding property}
We say that a reductive group scheme $G$ satisfies the adjoint embedding property (Em) if for all geometric points $p \in \hodge_{G}$, the image $\text{Ad}_*(p) \in \hodge_{\GL(\mathfrak{g})}$ is semistable if and only if $p$ is semistable.
\end{defn}

We will show in Subsection \ref{subsection: stability under change of group} that both of these properties are always satisfied if the characteristic of every residue field of $S$ is large enough, so that $G$ satisfies the low height property as in \Cref{defn: low height property}.

\end{subsection}

\begin{subsection}{Moduli space}\label{subsec: moduli spaces}

In the $G=\GL_N$ case, we can think of meromorphic $t$-connections as $\Lambda^{D,t}$-modules, where $\Lambda^{D,t}$ is a sheaf of rings of differential operators on  $C\times_{S} \mathbb{A}^1_S$ over $\mathbb{A}^1_{S}$ in the sense of \cite[p. 77]{Simpson-repnI}. Namely, let $\Lambda^{D}$ be the universal enveloping algebra of the sheaf of Lie algebras $T_{C/S}(D).$ It is a split almost polynomial sheaf of rings of differential operators as in \cite[p. 81]{Simpson-repnI}. 
Apply the deformation to the associated construction in \cite[p. 86]{Simpson-repnI} to $\Lambda^D$, we obtain $\Lambda^{D,t}.$ The GIT constructions in the work of Simpson and Langer imply the following. 

\begin{thm}[\cite{Simpson-repnI,langer-moduli-lie-algebroids}] \label{thm: moduli space for gln} \quad
   \begin{enumerate}[(1)] 
    \item The semistable geometric points of $\hodge_{\GL_N}$ are exactly the geometric points of an open substack $(\hodge_{\GL_N})^{ss} \subset \hodge_{\GL_N}$.
    \item For each integer $d \in \mathbb{Z}$, the open and closed substack $(\hodge_{\GL_N, d})^{ss} \subset (\hodge_{\GL_N})^{ss}$ where the underlying vector bundles has degree $d$ admits an $\mathbb{A}^1_{S}$-quasi-projective adequate moduli space (in the sense of \cite{alper_adequate}). \qed 
\end{enumerate}
\end{thm}

For the benefit of the reader, we explain in more detail the relation to the GIT construction. Openness of $(\hodge_{GL_N})^{ss}$ is given by Ramanan's argument as in \cite[Lem. 3.7]{Simpson-repnI}.
By \cite[Thm. 1.1]{langer-moduli-lie-algebroids}, there exist  quasi-projective moduli spaces $M_{\mathrm{Hod}, \GL_N,d}^D$ uniformly corepresenting the moduli functors of equivalence classes of meromorphic $t$-connections (with poles bounded by $D$).
 The moduli space $(\hodge_{GL_N,d})^{sch}$ is given by the GIT quotient $R\git GL_{N_0}$ where $R$ is the scheme constructed in \cite[Thm. 4.10]{Simpson-repnI} which represents the set-valued functor of framed meromorphic $t$-connections \cite[Lem. 4.9]{Simpson-repnI}.
  The construction relies on a boundedness result, which is proved in \cite[Thm. 0.2]{langer2004semistable} in arbitrary characteristic.
The stack quotient $R/GL_{N_0}$ is isomorphic to the stack $(\hodge_{GL_N,d})^{ss}.$
Therefore, the moduli space $(\hodge_{GL_N,d})^{sch}$ is indeed the adequate moduli space.

We use the result for $\GL_N$ and functoriality to conclude for general $G$.
\begin{thm} \label{thm: moduli space for hodge general G}
Suppose that $G$ is a reductive group satisfying properties (Em) and (Z) (see \Cref{defn: central property} and \ref{defn: embedding property}). Then,
   \begin{enumerate}[(1)] 
    \item The semistable geometric points of $\hodge_{G}$ are exactly the points of an open substack $(\hodge_{G})^{ss} \subset \hodge_{G}$.
    \item For each $d \in \pi_1(G)$, the open and closed substack $(\hodge_{G,d})^{ss} \subset (\hodge_{G})^{ss}$ where the underlying $G$-bundle has degree $d$ admits an $\mathbb{A}^1_{S}$-quasi-projective adequate moduli space $\Mhodge$. 
\end{enumerate}
\end{thm}


We provide a proof of \Cref{thm: moduli space for hodge general G} at the end of this subsection, after we discuss some necessary ingredients for the proof.

We shall use the homomorphism $\text{Ad} \times q: G \to \GL(\mathfrak{g}) \times G_{ab}$, which induces a morphism of stacks $\text{Ad}_*\times q_*: \hodge_G \to \hodge_{\GL(\mathfrak{g})\times G_{ab}}$. For the following lemma, recall that a quasi-compact and quasi-separated morphism $f: \cX \to \cY$ of algebraic stacks is called a good moduli space morphism \cite[Rmk. 4.4]{alper-good-moduli} if $f_*(\cO_{\cX}) = \cO_{\cY}$ and $f$ is cohomologically affine, i.e. the pushforward $f_*: \QCoh(\cX) \to \QCoh(\cY)$ is an exact functor between the abelian categories of quasicoherent sheaves. 
    \begin{lemma} \label{lemma: partial moduli space}
        Suppose that $G$ is split and satisfies property (Z) (see \Cref{defn: central property}). Then there is a factorization
        \[ \Ad_* \times q_*:\hodge_G \xrightarrow{h} \cB \xrightarrow{g}  \hodge_{\GL(\mathfrak{g})\times G_{ab}}\]
        where $h$ is a good moduli space morphism and $g$ is affine and of finite type.
    \end{lemma}
    \begin{proof}
    We denote by $j$ the closed immersion given by the factorization $\text{Ad}\times q: G \xrightarrow{i} \overline{G} \xrightarrow{j} \GL(\mathfrak{g}) \times G_{ab}.$ This induces a factorization $\text{Ad}_*\times q_*: \hodge_G \xrightarrow{i_*} \Bun_{G/Z_{\cD(G)}} \xrightarrow{j_*} \hodge_{\GL(\mathfrak{g})\times G_{ab}}$, where the morphism 
        $j_*$ is affine and of finite type (\Cref{lemma: change of group under closed immersion is affine}). Therefore, we can replace $\hodge_{\GL(\mathfrak{g})\times G_{ab}}$ with $\hodge_{G/Z_{\cD(G)}}$ and prove instead that there is an analogous factorization $i_*:\hodge_G \xrightarrow{h} \cB \xrightarrow{g}  \hodge_{G/Z_{\cD(G)}}$.
        By \Cref{lemma: change of group for etale isogenies}, $i_*$ is quasifinite and proper with relative stabilizers of multiplicative type. The base change of $i_*$ with any smooth scheme cover $U \to \hodge_{G/Z_{\cD(G)}}$ is a quasifinite proper stack over $U$. By the Keel-Mori theorem \cite[\href{https://stacks.math.columbia.edu/tag/0DUT}{Tag 0DUT}]{stacks-project}, the base change admits a coarse moduli space, which is quasifinite and proper (and hence finite) over $U$. Since the formation of the coarse space is compatible with flat-base change, this descends to a relative coarse space $\hodge_G \xrightarrow{h} \cB \xrightarrow{g} \hodge_{G/Z_{\cD(G)}}$, where $g$ is finite and schematic. Moreover, since the relative stabilizers of $i_*$ are linearly reductive, $\hodge_G \xrightarrow{h} \cB$ is a good moduli space morphism.
    \end{proof}
    
\begin{proof}[Proof of {\Cref{thm: moduli space for hodge general G}}]
Both properties can be checked \'etale locally on the base, and hence we may assume without loss of generality that $G$ is split.

\noindent \textbf{(1).} 
By the result for the general linear group $\GL(\mathfrak{g})$, the semistable geometric points of $\hodge_{\GL(\mathfrak{g})}$ are exactly the points of an open substack $(\hodge_{\GL(\mathfrak{g})})^{ss} \subset \hodge_{\GL(\mathfrak{g})}$. Since $G$ satisfies property (Em), the semistable geometric points of $\hodge_G$ are the points of the open preimage $(\Ad_*)^{-1}((\hodge_{\GL(\mathfrak{g})})^{ss}) \subset \hodge_G$.

\noindent \textbf{(2).}
Consider the factorization $\Ad_* \times q_*:\hodge_G \xrightarrow{h} \cB \xrightarrow{g}  \hodge_{\GL(\mathfrak{g})\times G_{ab}}$ from \Cref{lemma: partial moduli space}. There is an open substack $(\hodge_{\GL(\mathfrak{g})})^{ss} \times \hodge_{G_{ab}} \subset \hodge_{\GL(\mathfrak{g})\times G_{ab}}$, with open preimage $\mathcal{U} := g^{-1}((\hodge_{\GL(\mathfrak{g})})^{ss} \times \hodge_{G_{ab}})$. By property (Em), we have 
 $h^{-1}(\mathcal{U}) = (\hodge_G)^{ss}.$
 Hence we get a factorization $(\hodge_G)^{ss} \xrightarrow{h} \mathcal{U} \xrightarrow{g} (\hodge_{\GL(\mathfrak{g})})^{ss} \times \hodge_{G_{ab}}$, where $h$ is a good moduli space morphism and $g$ is affine and of finite type.

Since $G$ is split, we have that $G_{ab} \cong \GL_1^r$ for some $r$, and that $\hodge_{G_{ab}} \cong \left(\hodge_{\GL_1}\right)^r$. For meromorphic $t$-connections of rank $1$ we have $(\hodge_{\GL_1})^{ss} = \hodge_{\GL_1}$, since there are no proper parabolic subgroups. It follows that
\begin{align*}
    (\hodge_{\GL(\mathfrak{g})\times G_{ab}})^{ss}  = (\hodge_{\GL(\mathfrak{g})})^{ss} \times \left( (\hodge_{\GL_1})^{ss}\right)^r = (\hodge_{\GL(\mathfrak{g})})^{ss} \times \left( \hodge_{\GL_1}\right)^r
\end{align*}
The homomorphism $\text{Ad} \times q: G \to \GL(\mathfrak{g}) \times \GL_1^r$ induces a morphism of algebraic fundamental groups $ \pi_1(\Ad \times q): \pi_1(G) \to \pi_1(\GL(\mathfrak{g}) \times \GL_1^r)$ compatible with the induced morphism on connected components for $\Bun_{G}(C) \to \Bun_{\GL(\mathfrak{g}) \times \GL_1^r}(C)$. The image $d' = \pi_1(\Ad \times q)(d)$ corresponds to a tuple of integers $d'= (d_0, d_1, \ldots, d_r)$. We have that
\[(\hodge_{\GL(\mathfrak{g})\times G_{ab},d'})^{ss} = (\hodge_{\GL(\mathfrak{g}),d_0})^{ss} \times (\prod_{i=1}^r \hodge_{\GL_1, d_i})^{ss}\]
is the open and closed substack of $t$-connections of multi-degree $d'$. By \Cref{thm: moduli space for gln}, $(\hodge_{\GL(\mathfrak{g})\times G_{ab},d'})^{ss}$ admits a quasi-projective adequate moduli space $M_{{\mathrm{Hod}},\GL(\mathfrak{g})\times G_{ab}, d'}^D$ over $\mathbb{A}^1_S$. The image of $(\hodge_{G,d})^{ss}$ under $\text{Ad}\times q$ is contained in $(\hodge_{\GL(\mathfrak{g})\times G_{ab},d'})^{ss}$. Set $\mathcal{U}_{d'}:= g^{-1}((\hodge_{\GL(\mathfrak{g})\times G_{ab},d'})^{ss})$, which is an open and closed substack of $\mathcal{U}$. We have a chain of morphisms
\[ (\hodge_{G,d})^{ss} \hookrightarrow h^{-1}(\mathcal{U}_{d'}) \xrightarrow{h} \mathcal{U}_{d'} \xrightarrow{g} (\hodge_{\GL(\mathfrak{g})\times G_{ab}})^{d',ss}\]
The left-most inclusion $(\hodge_{G,d})^{ss} \hookrightarrow h^{-1}(\mathcal{U}_{d'})$ is an open and closed immersion. Therefore, the image $h((\hodge_{G,d})^{ss})$ is an open and closed substack of $\mathcal{U}_{d'}$, and the induced morphism $h: (\hodge_{G,d})^{ss} \to h((\hodge_{G,d})^{ss})$ is a good moduli space morphism. We conclude that we have a factorization
\[(\hodge_{G,d})^{ss} \xrightarrow{h} h((\hodge_{G,d})^{ss}) \xrightarrow{g}  (\hodge_{\GL(\mathfrak{g})\times G_{ab}, d'})^{ss}\]
where $h$ is a good moduli space morphism and $g$ is affine of finite type. By \cite[Lemma 5.2.11+Thm. 6.3.3]{alper_adequate} it follows that $h((\hodge_{G,d})^{ss})$ admits an adequate moduli space $X$ that is affine and of finite type over the $\mathbb{A}^1_S$-quasi-projective scheme $M_{{\mathrm{Hod}},\GL(\mathfrak{g})\times G_{ab}, d'}^D$. In particular, $X$ is quasi-projective over $\mathbb{A}^1_S$. The composition $(\hodge_{G,d})^{ss} \xrightarrow{h} h(\hodge_{G,d})^{ss}) \to X$ of two adequate moduli space morphisms is an adequate moduli space morphism, hence $X$ is an adequate moduli space for $(\hodge_{G,d})^{ss}$.
\end{proof}
\end{subsection}

\begin{subsection}{Low height property implies (Em) and (Z)}  \label{subsection: stability under change of group}

In this subsection we give some bounds on the characteristic of residue fields of points in $S$ that guarantee that $G$ satisfies the properties (Em) and (Z) as in \Cref{defn: central property} and \Cref{defn: embedding property}. For this we need to recall the definition of height of a representation (see \cite{balaji-paramewaran-tensors, balaji-deligne-parameswaran} and references therein).

 Let $H$ be a split connected reductive group over a field, and fix a maximal split torus $T \subset H$ and a Borel subgroup $B \supset T$. We denote by $T' = T \cap \cD(H)$ the corresponding maximal torus for the derived subgroup $\cD(H)$. The Borel subgroup $B$ determines a set of fundamental weights $\{\omega_i\}$ inside $X^*(T')_{\mathbb{Q}}$. A character $\chi$ of $T'$ is called ($B$-)dominant if it is a nonnegative linear combination $\chi = \sum_{i} c_i \omega_i$ with $c_i \geq 0$. For each dominant character $\chi = \sum_{i} c_i \omega_i$ in $X^*(T')_{\mathbb{Q}}$ we define $\text{ht}(\chi) = \sum_i c_i$. 
 
\begin{defn}
   Let $H$ be a split reductive group over an algebraically closed field $k$. Let $V$ be a linear representation of $H$. Let $V = \oplus_{\chi} V_{\chi}$ denote the decomposition into $T'$-weight spaces. Set
   \[ \text{ht}(V) := \text{max} \left\{ 2\text{ht}(\chi) \; \mid \; \text{$\chi$ is a dominant character with $V_{\chi} \neq 0$} \right\}\]
   We say that the representation $V$ is of low height if either $\text{char}(k)=0$ or $\text{char}(k) > \text{ht}(V)$.
\end{defn}

\begin{defn} \label{defn: low height property}
We say that a reductive group scheme $G \to S$ satisfies the low height property (LH) if for every geometric point $s \to S$ the adjoint representation of $G_s$ is of low height.
\end{defn}

By \cite[(5.2.6)]{serre2005complete}, we have that the height of the adjoint representation of $G_s$ is $2h(G_s)-2$, where $h(G_s)$ is the Coxeter number as in \cite[\S5.1]{serre2005complete}.

\begin{lemma} \label{lemma: low height property implies (Z)}
    If $G$ satisfies the low height property (LH), then it satisfies (Z) (as in \Cref{defn: central property}).
\end{lemma}
\begin{proof}
    This follows from the discussion in \cite[4.5]{balaji-deligne-parameswaran}.
\end{proof}

\begin{prop} \label{prop: semistability of low height representatons}
    Suppose that $G$ is reductive, and let $\rho: G \to \GL(V)$ be a representation of $G$, where $V$ is a vector bundle over $S$. Suppose that $\rho$ sends the neutral component $Z_{G}^{\circ}$ of the center into $Z_{\GL(V)}$ and that for all $s \in S$ the representation $G_{\kappa(s)} \to \GL(V_{\kappa(s)})$ is of low height. Then, for any semistable geometric point $p$ in $\hodge_G$ the associated geometric point $\rho_*(p)$ in $\hodge_{\GL(V)}$ is semistable.  
\end{prop}

We shall prove \Cref{prop: semistability of low height representatons} after recording the following.

\begin{coroll} \label{thm: stability under change of group}
Suppose that $G$ satisfies the low height property (LH).  Then, for all geometric points $p$ in $\hodge_{G}$, $p$ is semistable if and only if $\text{Ad}_*(p)$ is semistable (in other words, $G$ satisfies the property (Em) as in Definition \ref{defn: embedding property}).
\end{coroll}
\begin{proof}
If $p$ is semistable, then $\rho_*(p)$ is semistable by \Cref{prop: semistability of low height representatons}. For the converse, factor the adjoint representation as $G \xrightarrow{j} G/Z_{G} \xrightarrow{\Ad} \GL(\mathfrak{g})$. By \Cref{lemma: low height property implies (Z)} and \Cref{prop: semistability under etale isogeny}, we can replace $G $ with $G/Z_{G}$ and assume without loss of generality that $\Ad: G \to \GL(\mathfrak{g})$ is a closed immersion. Choose an unstable geometric point $p = (E, \nabla) \in \hodge_{G}(k)$ for some algebraically closed field $k$. By definition of instability, there exists some parabolic subgroup $P\subset G_k$, a $\nabla$-compatible parabolic reduction $E_P$ and a $P$-dominant character $\chi$ such that $\text{deg}(E_P(\chi)) >0$. After scaling, we can assume that the corresponding real cocharacter under the trace pairing $\text{tr}_{\mathfrak{g}}$ is actually a cocharacter of $\lambda: \mathbb{G}_m \to Z_L^{\circ}$. Consider $\lambda$ as a cocharacter of $\GL(\mathfrak{g})_k$ by composition, and denote by $P_{\lambda} \subset \GL(\mathfrak{g})_k$ the associated parabolic subgroup with Levi subgroup $L_{\lambda}$. We have an inclusion $P = P_{\lambda} \cap G \xrightarrow{j} P_{\lambda}$, which induces a $j_*(\nabla)$-compatible reduction  $P_{\lambda}$-reduction $j_*(E_P)$ of $j_*(E)$. The trace pairing of the standard representation $\text{tr}_{\mathfrak{g}}$ for $\GL(\mathfrak{g})$ agrees up to scaling with the trace pairing $\text{tr}_{\mathfrak{gl}(\mathfrak{g})}.$ 
Hence the dual character $\widetilde{\chi}: P_{\lambda} \to \mathbb{G}_m$ to the cocharacter $\lambda: \mathbb{G}_m \to Z_{L_{\lambda}}$ using the trace pairing $\text{tr}_{\mathfrak{gl}(\mathfrak{g})}$ agrees up to positive scaling with the one using the pairing $\text{tr}_{\mathfrak{g}}$. This means that the restriction of $\widetilde{\chi}$ to $P$ agrees with $\chi$ up to scaling. Since the degree of the line bundle $E_P(\chi)$ is positive, it follows that the degree of $j_*(E_P)(\widetilde{\chi})$ is also positive. This shows that $(j_*(E), j_*\nabla)$ is unstable.
\end{proof}

The rest of this subsection is dedicated to the proof of \Cref{prop: semistability of low height representatons}. We use the instability flag approach in \cite{rr-instability-flag, balaji-paramewaran-tensors}. Let us first recall the construction in \cite{rr-instability-flag}.

\begin{construction}[Instability reduction \cite{rr-instability-flag}] \label{consruction: instability reduction}
Suppose that $S= \Spec(k)$ for an algebraically closed field $k$. Let $\rho: G \to \GL(V)$ be a representation as in \Cref{prop: semistability of low height representatons}. Fix a $G$-bundle $E$ on $C$, with corresponding $\GL(V)$-bundle $\rho_*(E)$. Assume that we are given a maximal parabolic subgroup $Q \subset \GL(V)$, a parabolic reduction $\rho_*(E)_Q$, and a $Q$-dominant character $\chi$ such that $\text{deg}(\rho_*(E)_Q(\chi)) >0$. Out of this data we will construct a destabilizing parabolic reduction of the original $G$-bundle $E$.

\noindent $\bullet$ \textit{Some setup.} Set $F= \GL(V)/Q$ to be the flag variety. Note that $\GL(V)$-equivariant line bundles on $F$ correspond to characters of $Q$. 
We denote by $\cO_F(-\chi) \in \Pic^{\GL(V)}(X) = \Pic(X/\GL(V))$ the line bundle corresponding to the character $-\chi$. Since $\chi$ is $Q$-dominant, it follows that $\cO_X(-\chi)$ is ample on $X$. The reductive group $G$ acts on $X$ by multiplication on the left, and we view $\cO_F(-\chi)$ naturally as a $G$-equivariant bundle. 
By the assumptions on $\rho$, the neutral component of the center $Z^{\circ}_G$ acts trivially on $F$ and $\cO_F(-\chi)$. 

Let $X = \Spec(\bigoplus_{i=0}^{\infty} H^0(\cO_F(-\chi)^{\otimes i})$ denote the affine cone of $F$, which is an affine scheme equipped with an action of $G \times \mathbb{G}_m$ (the action of the second coordinate is given by the evident grading on $\cO_X$ that assigns $H^0(\cO_F(-\chi)^{\otimes i})$ weight $-i$). Note that there is a distinguished $G \times \mathbb{G}_m$-fixed point $0 \in X(k)$, and there is an induced projection $X \setminus 0 \to F$ that witnesses $X \setminus 0$ as the total space of the $\mathbb{G}_m$-torsor corresponding to the line bundle $\cO_F(\chi)$ on $F$. We denote by $\cO_X\langle 1 \rangle$ the $G \times \mathbb{G}_m$-equivariant line bundle on $X$ obtained by shifting the $\mathbb{G}_m$-weight of the structure sheaf $\cO_X$ by $1$, so that its restriction to $X \setminus 0$ corresponds to the pullback of the $G$-equivariant line bundle bundle $\cO_F(-\chi)$ on $F$ via the morphism $X \setminus 0 \to F$ equipped with its natural $G \times \mathbb{G}_m$-equivariant structure. 
The assumption that $Z^{\circ}_{G}$ acts trivially on $\cO_{F}(-\chi)$ implies that $Z_G^{\circ} \subset G \times \mathbb{G}_m$ acts trivially on $X$ and the induced $Z_G^{\circ}$-equivariant structure on $\cO_X\langle 1 \rangle $ is trivial.

Let $T \subset G$ be a maximal torus of $G$, and let $T' = T \cap \cD(G)$ be the corresponding maximal torus of the derived subgroup $\cD(G)$. 
We have a canonical isomorphism of cocharacter groups $X_*(T)_{\mathbb{R}} = X_*(T')_{\mathbb{R}} \oplus X_*(Z^{\circ}_G)_{\mathbb{R}}$. The trace form $\text{tr}_{\mathfrak{g}}$ yields a Weyl group invariant bilinear form on $X_*(T')_{\mathbb{R}}$.
We fix a choice of a rational bilinear quadratic form on $X_*(Z^{\circ}_G)_{\mathbb{R}}$, which provides us with a rational nondegenerate Weyl invariant quadratic form $b$ on $X_*(T)_{\mathbb{R}}$ extending $\text{tr}_{\mathfrak{g}}$ (the choice of extension $b$ will not matter in the end for our construction). We denote by $\widetilde{b}$ the extended bilinear form on $X_*(T \times \mathbb{G}_m)_{\mathbb{R}} = X_*(T)_{\mathbb{R}} \oplus \mathbb{R}$ given by $\widetilde{b}((x,y),(w,z)) = b(x,w) + yz$. We thus obtain a $W$-invariant extension $\widetilde{b}$ of $b$ to the space of real cocharacters of the maximal torus $T \times \mathbb{G}_m \subset G \times \mathbb{G}_m$.

\noindent $\bullet$ \textit{Kirwan stratification.} The $G \times \mathbb{G}_m$-equivariant ample line bundle $\cO_X\langle 1 \rangle$ and the rational quadratic norm $\widetilde{b}$ induce a weak $\Theta$-stratification on the quotient stack $X/(G \times \mathbb{G}_m)$ \cite[Ex. 5.3.4]{halpernleistner2018structure}. By \cite[Cor. 8.8]{balaji-paramewaran-tensors} and our assumption that the representation $V$ is of low height, this is actually a $\Theta$-stratification of $X/(G\times \mathbb{G}_m)$, which is also known as the Hesselink-Kempf-Kirwan-Ness stratification in classical GIT. It induces a decomposition into locally closed substacks:
\begin{equation} \label{equation: kirwan stratitication}
    X/(G\times \mathbb{G}_m) = \bigsqcup_{\widetilde{\lambda}\in \Lambda} Y_{\widetilde{\lambda}}^{ss}/P_{\widetilde{\lambda}},
\end{equation}
where $\Lambda$ is a finite set of cocharacters in $X_*(T \times \mathbb{G}_m)$ and each $Y^{ss}_{\widetilde{\lambda}}$ is a $P_{\widetilde{\lambda}}$-equivariant locally closed subscheme of $X$ contained in the closed attractor locus $X^{\widetilde{\lambda} \geq 0} \subset X$. 

Let us recall some properties of the stratification in \eqref{equation: kirwan stratitication}. We may write each $\widetilde{\lambda}$ as a pair $(\lambda, a_{\widetilde{\lambda}})$ for some $\lambda \in X_*(T)$ and $a_{\widetilde{\lambda}} \in \mathbb{Z} = X_*(\mathbb{G}_m)$. For $\widetilde{\lambda} =0$, the corresponding stratum $Y_0^{ss}$ is the open subscheme of $\cO_{X}\langle 1 \rangle$-semistable points in $X$. Any other stratum $Y_{\widetilde{\lambda}}^{ss}$ for $\widetilde{\lambda} \neq 0$ consists of $\cO_X\langle 1 \rangle $-unstable points. Among these there is a closed stratum $Y^{ss}_{(0,-1)} = \{0\}$ corresponding to the cocharacter $\widetilde{\lambda} = (0,-1)$. Any other $\widetilde{\lambda} = (\lambda, a_{\widetilde{\lambda}}) \neq 0, (0,-1)$ satisfies $\lambda \neq 0$.

Each pair $(\lambda, a_{\widetilde{\lambda}}) \in \Lambda$ is ``maximally destabilizing'' with respect to the linear functional 
\[\widetilde{\lambda} \mapsto \frac{-a_{\widetilde{\lambda}}}{\|\widetilde{\lambda}\|_{\widetilde{b}}} = \frac{-a_{\widetilde{\lambda}}}{\sqrt{\|\lambda\|_{b}^2 + a_{\widetilde{\lambda}}^2}}.\]
Here ``destabilizing'' means that $-a_{\widetilde{\lambda}} > 0$ for all $\widetilde{\lambda} \in \Lambda \setminus \{0\}$. Since $Z_G^{\circ}$ acts trivially on $X$ and $\widetilde{\lambda}$, the fact that $\widetilde{\lambda}$ is maximally destabilizing implies that  $\lambda \in X_*(T')_{\mathbb{Q}}$, i.e. the component of $\lambda$ in $X_*(Z_G^{\circ})_{\mathbb{Q}}$ is $0$. 

We denote by $Y_{\widetilde{\lambda}} \supset Y_{\widetilde{\lambda}}^{ss}$ the scheme-theoretic closure of $Y_{\widetilde{\lambda}}^{ss}$ in $X$, which is contained in the closed attractor locus $X^{\widetilde{\lambda} \geq 0} \subset X$. Taking the limit $\lim_{t \to 0} \widetilde{\lambda}(t) \cdot(-)$ induces a morphism of schemes $q: Y_{\widetilde{\lambda}} \hookrightarrow X^{\widetilde{\lambda} \geq 0} \to Z_{\widetilde{\lambda}}$, where $Z_{\widetilde{\lambda}} \subset X$ is a closed subscheme contained in the fixed locus $X^{\widetilde{\lambda} = 0} \subset X$ of the cocharacter $\widetilde{\lambda}$. In particular, $Z_{\widetilde{\lambda}}$ is preserved by the Levi subgroup $L_{\widetilde{\lambda}}$, and $q$ induces a morphism of stacks $q: Y_{\widetilde{\lambda}} / P_{\widetilde{\lambda}} \to Z_{\widetilde{\lambda}}/L_{\widetilde{\lambda}}$.

For any given $\widetilde{\lambda} = (\lambda, a_{\widetilde{\lambda}}) \neq 0$, consider the rational character 
\[\psi_{\widetilde{\lambda}} := \frac{-a_{\widetilde{\lambda}}}{\|\widetilde{\lambda}\|_{\widetilde{b}}^2} (\widetilde{\lambda}, - )_{\widetilde{b}}\]
in $X^*(T \times \mathbb{G}_m)_{\mathbb{Q}}$. This can actually be viewed as a rational character of $L_{\widetilde{\lambda}} = L_{\lambda} \times \mathbb{G}_m \subset G \times \mathbb{G}_m$. By construction, we may write $\psi_{\widetilde{\lambda}} = (\psi_{\lambda}, \psi_{a_{\widetilde{\lambda}}})$, where $\psi_{\lambda} = \frac{-a_{\widetilde{\lambda}}}{\| \widetilde{\lambda}\|_{\widetilde{b}}^2}(\lambda, -)_b$ in $X_*(T)_{\mathbb{Q}}$ and $\psi_{a_{\widetilde{\lambda}}} = \frac{-a_{\widetilde{\lambda}}}{\| \lambda\|_b^2 + a_{\widetilde{\lambda}}^2}$ in $\mathbb{Q} = X_*(\mathbb{G}_m)_{\mathbb{Q}}$. Since $\lambda \in X_*(T')_{\mathbb{Q}}$, the rational character $(\lambda, -)_b$ is the dual with respect to the trace form $\text{tr}_{\mathfrak{g}}$. Since $-a_{\widetilde{\lambda}} > 0$, it follows that $\psi_{\lambda}$ is a $P_{\lambda}$-dominant rational character. 

Let $\cL_{\widetilde{\lambda}}: = \cO_{Z_{\widetilde{\lambda}}}\langle 1 \rangle  \otimes \cO_{Z_{\widetilde{\lambda}}}(-\psi_{\widetilde{\lambda}})$ be the $L_{\widetilde{\lambda}} = L_{\lambda} \times \mathbb{G}_m$-equivariant (rational) line bundle on $Z_{\widetilde{\lambda}}$ obtained by twisting the line bundle $\cO_{Z_{\widetilde{\lambda}}}\langle 1 \rangle$ by the (rational) character $-\psi_{\widetilde{\lambda}}$ of $L_{\widetilde{\lambda}}$. We denote by $Z_{\widetilde{\lambda}}^{ss} \subset Z_{\widetilde{\lambda}}$ the open semistable locus with respect to the $L_{\widetilde{\lambda}}$-equivariant line bundle $\cL_{\widetilde{\lambda}}$. The inverse image $q^{-1}(Z_{\widetilde{\lambda}}^{ss}) \subset Y_{\widetilde{\lambda}}$ is exactly the open subset $Y^{ss}_{\widetilde{\lambda}}$. The quotient stack $Z_{\widetilde{\lambda}}^{ss}/L_{\widetilde{\lambda}}$ is called the center of the stratum $Y^{ss}_{\widetilde{\lambda}}/P_{\widetilde{\lambda}}$. 

\noindent $\bullet$ \textit{Instability reduction.} We now return to our $G$-bundle $E$. The data of $E$ along with the $Q$-reduction of $\rho_*(E)$ is equivalent to a morphism $f: C \to F/G = (X\setminus 0)/(G \times \mathbb{G}_m)$. By construction there is an isomorphism of the composition $C \to (X\setminus 0)/(G\times \mathbb{G}_m) \to BG$ with the $G$-bundle $E: C \to BG$. The pullback of the bundle $\cO_X\langle 1\rangle$ on $X/(G \times \mathbb{G}_m)$ under $f$ is given by $\rho_*(E)_Q(-\chi)$. Since $\text{deg}(\rho_*(E)_Q(\chi)) >0$, it follows that $f^*(\cO_X\langle 1\rangle)$ has no nonzero sections, which forces the image of $f$ to be contained in the unstable locus of $X/(G\times \mathbb{G}_m)$ with respect to the line bundle $\cO_X\langle 1 \rangle$. In particular, the generic point $\eta$ of the curve $C$ maps to a stratum $Y_{\widetilde{\lambda}}^{ss}/P_{\widetilde{\lambda}}$ in \eqref{equation: kirwan stratitication} for some $\widetilde{\lambda} = (\lambda, a_{\widetilde{\lambda}}) \neq 0$. Since $\eta$ maps to $X \setminus \{0\}$, we also have that $Y^{ss}_{\lambda} \neq \{0\} = Y_{(0,-1)}$, and hence $\lambda \neq 0$. Since the morphism $Y_{\widetilde{\lambda}}/P_{\widetilde{\lambda}} \to X/(G \times \mathbb{G}_m)$ is representable and projective, the valuative criterion for properness induces an extension $C \to Y_{\widetilde{\lambda}}/P_{\widetilde{\lambda}}$ of the lift $\eta \to Y_{\widetilde{\lambda}}/P_{\widetilde{\lambda}} \to X/(G\times \mathbb{G}_m)$. If we set $P = P_{\lambda} \subset G$, then the composition $C \to Y_{\widetilde{\lambda}} /P_{\widetilde{\lambda}} \to BP_{\lambda} \times B\mathbb{G}_m \to BP_{\lambda}$ yields a $P$-bundle $E_P$ which is in fact a parabolic reduction of $E$. We call $E_P$ the \underline{instability parabolic reduction} induced by $\rho_*(E)_Q$. 

To conclude this construction, we claim that we have $\text{deg}(E_P(\psi_{\lambda})) >0$. In other words, the pair $E_P$ and a multiple of the rational $P$-dominant character $\psi_\lambda$ jointly destabilize the $G$-bundle $E$. To see this claim, consider the composition $g: C \to Y_{\widetilde{\lambda}}/P_{\widetilde{\lambda}} \to Z_{\widetilde{\lambda}}/L_{\widetilde{\lambda}}$. The generic point of the curve lands in the $\cL_{\widetilde{\lambda}}$-semistable locus $Z^{ss}_{\lambda}/L_{\widetilde{\lambda}} \subset Z_{\widetilde{\lambda}}/L_{\widetilde{\lambda}}$, because $\eta$ was mapped to $Y^{ss}_{\widetilde{\lambda}}/P_{\widetilde{\lambda}}$. In particular $g^*(\cL_{\widetilde{\lambda}})$ has nonzero global sections, and it follows that we must have 
\begin{equation} \label{second to last equation: instability parabolic reduction construction}
    \text{deg}(g^*(\cL_{\widetilde{\lambda}})) = \text{deg}(g^*(\cO_{Z_{\widetilde{\lambda}}}\langle 1 \rangle) \otimes g^*(\cO_{Z_{\widetilde{\lambda}}}(-\psi_{\widetilde{\lambda}})) \geq 0.
\end{equation}
Using the decomposition $\psi_{\widetilde{\lambda}} = (\psi_{\lambda}, \psi_{a_{\widetilde{\lambda}}})$ and the natural identification $g^*(\cO_{Z_{\widetilde{\lambda}}}(-\psi_{\lambda})) = E_P(\psi_{\lambda})$, we have
\[ g^*(\cO_{Z_{\widetilde{\lambda}}}(-\psi_{\widetilde{\lambda}})) = g^*(\cO_{Z_{\widetilde{\lambda}}}(-\psi_{\lambda})) \otimes g^*(\cO_{Z_{\widetilde{\lambda}}}(-\psi_{a_{\widetilde{\lambda}}})) = E_P(\psi_{\lambda}) \otimes g^*\left(\cO_{Z_{\widetilde{\lambda}}}\left(\frac{a_{\widetilde{\lambda}}}{\|\widetilde{\lambda}\|_{\widetilde{b}}^2 + a_{\widetilde{\lambda}}^2}\right)\right),\]
Hence, we may rewrite \eqref{second to last equation: instability parabolic reduction construction} as follows:
\begin{equation} \label{last equation: instability parabolic reduction construction}
    \left(1 + \frac{a_{\widetilde{\lambda}}}{\|\widetilde{\lambda}\|_b^2 + a_{\widetilde{\lambda}}^2} \right) \cdot \text{deg}(g^*(\cO_{Z_{\widetilde{\lambda}}}\langle 1\rangle) + \text{deg}(E_P(\psi_{\lambda})) \geq 0.
\end{equation}
Note that $g^*(\cO_{Z_{\widetilde{\lambda}}}\langle 1 \rangle) = f^*(\cO_X\langle 1 \rangle) = \rho_*(E)_Q(-\chi)$, and hence 
\[\text{deg}(g^*(\cO_{Z_{\widetilde{\lambda}}}\langle 1 \rangle)) = \text{deg}(\rho_*(E)_Q(-\chi)) <0.\]
We also have $1 + \frac{a_{\widetilde{\lambda}}}{\|\widetilde{\lambda}\|_b^2 + a_{\widetilde{\lambda}}^2} >0$ since $\lambda \neq 0$, and hence the inequality \eqref{last equation: instability parabolic reduction construction} forces $\text{deg}(E_P(\psi_{\lambda})) > 0$.
\qed
\end{construction}

Our strategy to prove \Cref{prop: semistability of low height representatons} is to apply \Cref{consruction: instability reduction} and checking that the compatibility with meromorphic $t$-connections is preserved. In order to do this, we use some results on affine schemes with meromorphic $t$-connections, generalizing \cite{balaji-paramewaran-tensors} in the spirit of \cite[\S A]{chen-zhu-nah}. We refer the reader to \Cref{appendix: D-affine schemes} for more details on this notion, which we use in the proof of the next lemma.

\begin{lemma} \label{lemma: the instability reduction is compatible with the connection}
Let $\rho: G \to \GL(V)$ be a representation as in \Cref{prop: semistability of low height representatons}. Let $k$ be an algebraically closed field, and let $p = (E, \nabla) \in \hodge_{G}(k)$. Let $\rho_*(E)_Q$ be a parabolic reduction to a maximal parabolic subgroup $Q \subset \GL(V)_k$ such that there exists a $Q$-dominant character $\chi$ with $\text{deg}(\rho_*(E)_Q(\chi)) >0$. If the reduction $\rho_*(E)_{Q}$ is $\rho_*(\nabla)$-compatible, then the instability reduction $E_P$ in \Cref{consruction: instability reduction} is $\nabla$-compatible.
\end{lemma}
\begin{proof}
We base change to $\Spec(k)$ and assume for simplicity that $S = \Spec(k)$. The parabolic reduction $\rho_*(E)_Q$ induces a $\GL(V) \times \mathbb{G}_m$-equivariant morphism of affine $C$-schemes with meromorphic $t$-connections $v: \rho_*(E) \times_C \cM^{\times} \to X_C$,
where $X$ denotes the affine cone of $\GL(V)/Q$ and $\cM$ denotes the line bundle $\rho_*(E)_Q(-\chi)$ for a $Q$-dominant character $\chi$ (see \Cref{appendix: D-affine schemes}, especially after \Cref{notn: compatibility of section} for more details). By composing with the inclusion $E \hookrightarrow \rho_*(E)$, we get a $G \times \mathbb{G}_m$-equivariant morphism of schemes $v: E \times_C \cM^{\times} \to X_C$ with meromorphic $t$-connections.

Consider the closed $P_{\lambda}$-subscheme $Y_{\lambda} \subset X$ as in \Cref{consruction: instability reduction}. By tracing back the correspondence between parabolic reduction and sections of $E(G/P)$, we see that the instability reduction $E_P \subset E$ fits into the following Cartesian $\mathbb{G}_m$-equivariant diagram
\[
\begin{tikzcd}
   E_P \times_C \cM^{\times}  \ar[r] \ar[d] & (Y_{\lambda})_C \ar[d] \\
  E \times_C \cM^{\times} \ar[r, "v"] & X_C
\end{tikzcd}
\]
The closed immersion $(Y_{\lambda})_{C} \hookrightarrow X_C$ is compatible with the corresponding trivial connections (see \Cref{example: trivial connection on affine scheme}). By \Cref{lemma: compatibility of fiber products with closed subschemes preserved by connections},  the closed immersion $E_P \times_C \cM^{\times} \hookrightarrow E\times_C \cM^{\times}$ is preserved by the meromorphic $t$-connection on $E\times_C \cM^{\times}$. Both $\cO_{E\times_C \cM^{\times}} = \bigoplus_{ n \in \mathbb{Z}} \cO_E \otimes \cM^{\otimes n}$ and $\cO_{E_P\times_C \cM^{\times }} = \bigoplus_{ n \in \mathbb{Z}} \cO_{E_P} \otimes \cM^{\otimes n}$ are naturally graded, and by construction the meromorphic $t$-connection on $\cO_{E\times_C \cM^{\times}}$ preserves the grading. By looking at the zero graded pieces, we see that the inclusion $E_P \hookrightarrow E$ is preserved by the connection $\nabla_E$. By \Cref{lemma: compatiblity of parabolic reuductions in terms of affine schemes with connection}, it follows that the parabolic reduction $E_P$ is $\nabla$-compatible.
\end{proof}

\begin{proof}[Proof of \Cref{prop: semistability of low height representatons}]
    The geometric point $p=(E, \nabla)$ is defined over an algebraically closed field $k$. Suppose that the associated point $(\rho_*(E), \rho_*(\nabla)) \in \hodge_{\GL(V)}(k)$ is unstable. This means that the vector bundle $\rho_*(E)$ admits a subbundle $\mathcal{F} \subset \rho_*(E)$ of larger slope such that $\rho_*(\nabla)$ preserves $\mathcal{F}$. This corresponds to a $\rho_*(\nabla)$-compatible reduction $\rho_*(E)_Q$ of structure group of $\rho_*(E)$ to a maximal parabolic $Q \subset \GL(V)$ such that $\text{deg}(\rho_*(E)_Q))>0$ for a $Q$-dominant character $\chi$. By \Cref{lemma: the instability reduction is compatible with the connection}, the corresponding instability reduction $E_P$ of $E$ (\Cref{consruction: instability reduction}) is $\nabla$-compatible. By \Cref{consruction: instability reduction}, there is a $P$-dominant character $\psi$ such that $\text{deg}(E_P(\psi))>0$. Hence, the geometric point $p = (E, \nabla)$ is unstable.
\end{proof}

\end{subsection}
\end{section}
\begin{section}{Deformation theory}\label{section: deformation theory}

This section generalizes the deformation theory in \cite{decataldo-herrero-nahpoles} to the more general setting of smooth group schemes.
It also generalizes the treatment of Biswas-Ramanan for twisted Higgs bundles over the complex numbers in \cite{biswas-ramanan-infinitesimal}.

\begin{subsection}{Cech cohomology and base change}
Let $A$ be a Noetherian ring over $\mathbb{A}^1_{S}$, inducing a morphism $Spec(A) \to \mathbb{A}^1_S$. We denote by $t_{A}$ the image of the $\mathbb{A}^1_S$-coordinate $t$ inside $A$, and set $C_{A}:= C\times_{S} Spec(A)$. Let $U \to C_{A}$ be an \'etale cover with $U$ affine. For any coherent $\mathcal{O}_{C_{A}}$-module $\mathcal{E}$, we denote by $(\widecheck{C}^{\bullet}_{U}(\mathcal{E}), \delta)$ the \v{C}ech complex corresponding to the cover. We use the notation $U^{(i)} := U \times_{C_{A}} \ldots  \times_{C_{A}} U$ for the $i^{th}$ fiber product. The \v{C}ech complex has $i^{th}$-term given by $\widecheck{C}^{i+1}_{U}(\mathcal{E}) = \mathcal{E}(U^{(i)})$, with differential $\delta$.  Since $C_{A}$ is separated, the morphism $U \to C_{A}$ is affine, and therefore each $U^{(i)}$ is an affine scheme.
\begin{lemma}[{\cite[Lemma 4.1]{decataldo-herrero-nahpoles}}] \label{lemma: base change cech cohomology lemma}
With notation as above, we have the following.
\begin{enumerate}[(a)]
    \item The $i^{th}$ cohomology $\widecheck{H}^{i}_{U}(\mathcal{E})$ of the Cech complex computes the sheaf cohomology $H^i(\mathcal{E})$. The cohomology groups $\widecheck{H}^{i}_{U}(\mathcal{E})$ vanish for $i \geq 2$.
    \item Suppose that $\mathcal{E}$ is $A$-flat. Then, for any $A$-module $M$, the natural morphism 
    \[H^1(\mathcal{E}) \otimes_{A} M \to H^1(\mathcal{E} \otimes_{A} M)\]
    is an isomorphism.
    \item Suppose that $\mathcal{E}$ is $A$-flat. Let $B$ be an $A$-algebra, inducing a morphism $f: Spec(B) \to Spec(A)$. Then the natural morphism $H^1(\mathcal{E}) \otimes_{A} B \to H^1((f_C)^*(\mathcal{E}))$ is an isomorphism.
    \item With notation as in part (c), assume that $\mathcal{G}$ is another $A$-flat coherent sheaf on $C_{A}$ equipped with a $A$-linear morphism of abelian sheaves $\varphi: \mathcal{E} \to \mathcal{G}$. Then, the induced diagram
\[
\begin{tikzcd}
  H^1(\mathcal{E})\otimes_{A} B \; \; \; \; \ar[r,"H^1(\varphi)\otimes_{A} B"] \ar[d, symbol = \xrightarrow{\sim}] & \; \; \; \;  H^1(\mathcal{G})\otimes_{A} B \ar[d, symbol = \xrightarrow{\sim}] \\  H^1((f_C)^* \mathcal{E}) \; \; \ar[r, "H^1(\varphi \otimes_{A} B)"] & \; \; H^1((f_C)^* \mathcal{G})
\end{tikzcd}
\]
    commutes. \qed
\end{enumerate}
\end{lemma}
Fix an $\mathbb{A}^1_S$-morphism $p_{A}: Spec(A) \to (\hodge_{G})^{ss}$ represented by a pair $(E, \nabla)$ consisting of a $G$-bundle $E$ on $C_{A}$ and a meromorphic $t_{A}$-connection $\nabla$. We denote by $\Ad(\nabla): \Ad(E) \to \Ad(E) \otimes_{\mathcal{O}_{C}} \Omega^1_{C_A/A}(D)$ the corresponding meromorphic $t_A$-connection on the associated vector bundle $\Ad(E)$ induced by the adjoint representation $\Ad:G \to \GL(\mathfrak{g})$. Recall that $\Ad(\nabla)$ is an $A$-linear morphism of abelian sheaves.
\begin{defn} \label{defn: obstruction module}
The module of obstructions $\mathcal{Q}_{p_A}$ is the cokernel of the following morphism in sheaf cohomology.
\[\mathcal{Q}_{A} : = \text{coker}\left[ H^1(\Ad(E)) \xrightarrow{H^1(\Ad(\nabla))} H^1( \Ad(E) \otimes_{\mathcal{O}_{C_A}} \Omega^1_{C_A/A}(D)) \right]\]
\end{defn}

Note that $\mathcal{Q}_{A}$ is a finitely generated $A$-module, since $C_{A} \to A$ is proper.

\begin{coroll} \label{coroll: base change obstruction module}
Let $B$ be a Noetherian $A$-algebra, inducing a morphism $f: Spec(B) \to Spec(A)$. Consider the pullback $t_{B}$-connection $p_{B} = ((f_{C}^*(E), f^*(\nabla))$ on $C_{B}$. Then, there is a natural isomorphism $\mathcal{Q}_{p_{B}} \cong \mathcal{Q}_{p_{A}} \otimes_{A} B$.
\end{coroll}
\begin{proof}
This follows from Lemma \ref{lemma: base change cech cohomology lemma} (d), the definition of $\mathcal{Q}_{p_{A}}$, and the fact that the formation of cokernels commutes with $(-)\otimes_{A}B$.
\end{proof}

\begin{defn} \label{defn: deformation complex}
    For any $p_A = (E, \nabla)$ over $C_A$ as above, we define the following two term complex
    \[ \mathcal{C}_{p_{A}}^{\bullet} := \left[\Ad(E) \xrightarrow{\Ad(\nabla)} \Ad(E) \otimes_{\cO_{C_A}} \Omega^1_{C_A/A} \right]\]
    Here we use cohomological grading, and the leftmost element $\Ad(E)$ sits in degree $0$.
\end{defn}
 The hypercohomology spectral sequence induces an identification $\mathcal{Q}_{p_A} = \mathbb{H}^2(\mathcal{C}_{E,\nabla}^{\bullet})$. We note that for any homomorphism $\rho: G \to H$ of smooth group schemes, we get an induced morphism of complexes $\mathcal{C}_{p_{A}}^{\bullet} \to \mathcal{C}_{\rho_*(p_{A})}^{\bullet}$.
\end{subsection}

\begin{subsection}{Deformation theory of the Hodge moduli stack}
\begin{context}[Lifting problem] \label{context: lifting problem}
    Let $A$ be a local Artin algebra with maximal ideal $\mathfrak{m}$ and residue field $k$. Choose a pair $p_{A} = (E, \nabla)$ corresponding to a morphism $p_{A}: Spec(A) \to \hodge_{G}$. Let $\widetilde{A}$ be another local Artin algebra, equipped with a surjective local homomorphism $\widetilde{A} \twoheadrightarrow A$ cut out by a square-zero ideal $I$ in $\widetilde{A}$. Choose a morphism $Spec(\widetilde{A}) \to \mathbb{A}^1_{S}$ such that $\iota: Spec(A) \hookrightarrow Spec(\widetilde{A})$ is a morphism over $\mathbb{A}^1_S$. We have a diagram
\[
\begin{tikzcd}
  Spec(A) \ar[r] \ar[d, symbol = \hookrightarrow, "\iota", labels= left] &  \hodge_{G} \ar[d] \\ Spec(\widetilde{A}) \ar[r] \ar[ur, dashrightarrow] & \mathbb{A}^1_S
\end{tikzcd}
\]
We are interested in finding lifts as in the dotted arrow.
\end{context} 

\begin{defn}
    With notation as in \Cref{context: lifting problem}, we denote by $\text{Def}_{\widetilde{A}/A}(p_A)$ the groupoid of pairs $(p_{\widetilde{A}}, \theta)$ where $p_{\widetilde{A}} \in \hodge_{G}(\widetilde{A})$ and $\theta$ is an isomorphism $\theta: \iota^*(p_{\widetilde{A}}) \xrightarrow{\sim} p_{A}$.
\end{defn}

The dotted lifts in \Cref{context: lifting problem} will be governed by the hypercohomology of the complex $\mathcal{C}_{p_{A}}^{\bullet}$ (\Cref{defn: deformation complex}). We shall denote by $R\Gamma(\mathcal{C}_{E,\nabla}^{\bullet})$ the derived pushforward to $\Spec(A)$. This can be represented by a complex of $A$-modules concentrated in degree $\geq 0$ and whose $i^{th}$ cohomology is the hypercohomology $\mathbb{H}^i(\mathcal{C}_{p_{A}}^{\bullet})$. 

We shall need the following standard notion, see \cite[Cor. 1.4.17]{deligne1973formule} for the equivalence between Picard groupoids and two-term complexes.
\begin{defn}
    Let $\mathcal{C} = \left[ C^0 \xrightarrow{d} C^1\right]$ be a two term complex of abelian groups concentrated in degrees $[0,1]$. The Picard groupoid $\text{Pic}(C^{\bullet})$ associated to $\mathcal{C}$ has objects given by $C^1$. The set $\Hom_{\text{Pic}(C^{\bullet})}(x,y)$ of isomorphisms between two objects $x,y \in C^1$ consists of all the elements $f \in C^0$ such that $d(f) = y-x$. Composition is defined by adding elements of $C^0$. A morphism $g: C^{\bullet} \to M^{\bullet}$ of two term complexes induces a functor $\text{Pic}(g): \text{Pic}(C^{\bullet}) \to \text{Pic}(M^{\bullet})$. If $g$ is a quasi-isomorphism, then $\text{Pic}(g)$ is an equivalence of groupoids.
\end{defn}

\begin{prop} \label{prop: obstruction class}
With notation as in \Cref{context: lifting problem}, we have the following.
\begin{enumerate}[(a)]
    \item There is an element $\text{ob}_{p_{A}} \in \mathcal{Q}_{p_{A}}\otimes_{A}I= \mathbb{H}^2(\mathcal{C}_{p_{A}}^{\bullet} \otimes_{A} I)$ such that $\text{ob}_{p_{A}} = 0$ if and only if a lift to $C_{\widetilde{A}}$ exists. In particular, such a lift always exists if $\mathcal{Q}_{p_{A}} = 0$.
    \item Suppose that $\text{ob}_{p_A}=0$. Let $\tau_{\leq 1}R\Gamma(\mathcal{C}_{p_A}^{\bullet} \otimes_A I)$ the truncation of the derived pushforward, and choose a representative two term complex $C^{\bullet}_{p_A}$. There is an equivalence of groupoids $\text{Def}_{\widetilde{A}/A}(p_A) \simeq \text{Pic}(C^{\bullet}_{p_A})$.

    \item Let $\rho: G \to H$ be a morphism of smooth group schemes, inducing a functor of groupoids $\rho_*: \text{Def}_{\widetilde{A}/A}(p_A) \to \text{Def}_{\widetilde{A}/A}(\rho_*(p_A))$. Then for any representative $g: C^{\bullet}_{p_A} \to C^{\bullet}_{\rho_*(p_A)}$ of the induced morphism $\tau_{\leq 1}R\Gamma(\mathcal{C}_{p_A}^{\bullet} \otimes_A I) \to \tau_{\leq 1}R\Gamma(\mathcal{C}_{\rho_*(p_A)}^{\bullet} \otimes_A I)$ one can choose the equivalences in (b) so that they fit into a 2-commutative diagram:
\[
\begin{tikzcd}
  \text{Def}_{\widetilde{A}/A}(p_A) \ar[r, "\sim"] \ar[d, "\rho_*"] &  \text{Pic}(C^{\bullet}_{p_A}) \ar[d, "\text{Pic}(g)"] \\ \text{Def}_{\widetilde{A}/A}(\rho_*(p_A)) \ar[r, "\sim"] & \text{Pic}(C^{\bullet}_{\rho_*(p_A)})
\end{tikzcd}
\]
\end{enumerate}
\end{prop}
\begin{proof}
This is a standard but long exercise in classical deformation theory similar to arguments in \cite{biswas-ramanan-infinitesimal}. We present a construction of the obstruction class in (a) through the explicit use of Cech cohomology. The statements of (b) and (c) follow from similar considerations by using the representative of derived pushforward $R\Gamma(\mathcal{C}_{p_A}^{\bullet} \otimes_A I)$ coming from the total complex of the associated double Cech complex, and applying cohomological truncation.

Let $U$ be an affine \'etale cover of $C_{A}$ such that $E|_{U}$ is trivializable. By topological invariance of the \'etale site \cite[\href{https://stacks.math.columbia.edu/tag/04DZ}{Tag 04DZ}]{stacks-project}, there is an affine \'etale cover $\widetilde{U} \to C_{\widetilde{A}}$ such that the base change $\widetilde{U}_A$ is isomorphic to $U$ as a scheme over $C_A$. We use the notation and results in \Cref{lemma: base change cech cohomology lemma} freely without further mention. 
The obstruction to lifting the $G$-bundle $E$ to $C_{\widetilde{A}}$ lives in the cohomology group $H^2(\Ad(E) \otimes_{A} I)$, which is $0$ because $C$ is a curve. Hence, we can choose a lift $\widetilde{E}^0$ of $E$ to $C_{\widetilde{A}}$. We use the notation $E_{U} := E|_{U}$, $\nabla_U := \nabla|_{U}$ and $\widetilde{E}^0_{U} : = \widetilde{E}^0|_{\widetilde{U}}$. Note that we can lift the $t_{A}$-connection $\nabla_U$ to a $t_{\widetilde{A}}$-connection $\widetilde{\nabla}^0_U$ on $\widetilde{E}^0_U$. Indeed, a trivialization of $E_U$ induces a trivialization of the torsor $A^{D}_{t-\text{Conn},E_U}$. The $t_{A}$-connection $\nabla_U$ corresponds to some section $M \in \Ad(E_U) \otimes_{\mathcal{O}_{U}} \Omega^1_{U}(D)$. By the smoothness of $G$, we can lift the trivialization of $E_U$ to one of $\widetilde{E}^0_U$. Since $U$ is affine and $\widetilde{A} \to A$ is surjective, we can also lift $M$ to a section $\widetilde{M}$ in $\Ad(\widetilde{E}^0_U) \otimes_{\mathcal{O}_{\widetilde{U}}} \Omega^1_{\widetilde{U}}$, which corresponds to a lift $\widetilde{\nabla}^0_U$ on the trivialized $G$-bundle $\widetilde{E}^0_U$.

Set $c := (p_1)^*(\widetilde{\nabla}^0_U) - (p_2)^*(\widetilde{\nabla}^0_U) \in \Ad(\widetilde{E}^0_{U^{(2)}}) \otimes \Omega^1_{\widetilde{U}^{(2)}}(D)$, where $p_1,p_2: \widetilde{U}^{(2)} \to \widetilde{U}$ are the two projections. 
Since $(p_1)^*(\widetilde{\nabla}^0_U) \equiv (p_2)^*(\widetilde{\nabla}^0_{U}) = \nabla|_{U^{(2)}} (mod \, I)$, the element $c$ actually lands in the submodule $\Ad(\widetilde{E}^0_{U^{(2)}}) \otimes_{\mathcal{O}_{\widetilde{U}^{(2)}}} \Omega^1_{\widetilde{U}^{(2)}}(D) \otimes_{A} I \cong \Ad(E_{U^{(2)}}) \otimes_{\mathcal{O}_{U^{(2)}}} \Omega^1_{U^{(2)}}(D) \otimes_{A} I$. Therefore, we can view $c$ as a chain in $\widecheck{C}^1_{U}(\Ad(E) \otimes_{\mathcal{O}_{C}} \Omega^1_{C/S}(D) \otimes_{A} I)$. If we denote the natural projections by $p_{12}, p_{23}, p_{13}: U^{(3)} \to U^{(2)}$, then
\begin{gather*}
    \delta(c) = (p_{12})^*\left((p_1)^*(\widetilde{\nabla}^0_U)- (p_2)^*(\widetilde{\nabla}^0_U)\right) - (p_{13})^*\left((p_1)^*(\widetilde{\nabla}^0_U)- (p_2)^*(\widetilde{\nabla}^0_U)\right) + (p_{23})^*\left((p_1)^*(\widetilde{\nabla}^0_U)- (p_2)^*(\widetilde{\nabla}^0_U)\right) = 0
\end{gather*}
Therefore $c$ is a cocycle. Let $[c]$ be the corresponding cohomology class in $H^1(\Ad(E) \otimes_{\mathcal{O}_C} \Omega^1_{C/S}(D) \otimes_{A} I) \cong H^1(\Ad(E) \otimes_{\mathcal{O}_{C}} \Omega^1_{C/S}(D)) \otimes_{A} I$. Set $\text{ob}_{p_{A}}$ to be the image of $[c]$ in $\mathcal{Q}_{p_{A}} \otimes_{A} I$. We need to check that $\text{ob}_{p_{A}}$ does not depend on the choices (the covering and the lift $(\widetilde{E}^0, \widetilde{\nabla}^0_U)$), and that $\text{ob}_{p_{A}}=0$ if and only if there exists a lift of $(E, \nabla)$ to $C_{\widetilde{A}}$.

\noindent \textbf{Independence on the covering:} It suffices to check that we obtain the same $\text{ob}_{p_{A}}$ after taking a refinement $W_{new}$ of the covering $U$, while keeping the lift $\widetilde{E}^0$ and keeping the (refinements of) the lift $\widetilde{\nabla}^0_U$. By definition, the resulting cocycle $c_{new}$ will be the image of the old cocycle $c$ under the natural morphism on Cech complexes induced by refinement, and therefore $\text{ob}_{p_{A}}$ will not change.

\noindent \textbf{Independence on $\widetilde{\nabla}^0_U$:} Choose a different lift $\widetilde{\nabla}^1_U$. Let $c^1$ denote the cocycle defined by $c^1 = (p_1)^*(\widetilde{\nabla}^1_U)- (p_2)^*(\widetilde{\nabla}^1_U)$ Since $\widetilde{\nabla}^1_U \equiv \widetilde{\nabla}^0_U = \nabla_U (mod \, I)$, the difference $h : = \widetilde{\nabla}^1_U - \widetilde{\nabla}^0_U$ is a section of $\Ad(E_{U}) \otimes \Omega^1_{U} \otimes_{A} I$. We can view $h$ as a chain in $\widecheck{C}^0_{U} (\Ad(E) \otimes_{\mathcal{O}_{C}} \Omega^1_{C/S}(D) \otimes_{A} I)$. By definition, we have $c^1 = (p_1)^*(\widetilde{\nabla}^1_U) - (p_2)^*(\widetilde{\nabla}^1|_U)  = c + \delta(h)$. 
Therefore $c^1$ is cohomologous to $c$ in $\widecheck{C}^1_{U}(\Ad(E)\otimes_{\mathcal{O}_{C}} \Omega^1_{C/S}(D) \otimes_{A} I)$. 

\noindent \textbf{Observation 1.} Conversely, for any $0$-chain $h$ we can define $\widetilde{\nabla}^1_U = \widetilde{\nabla}^0_U +h$, and then $c^1 = c + \delta(h)$. 

\noindent \textbf{Independence on $\widetilde{E}$:} Choose a different lift $\widetilde{E}^1$ of the $G$-bundle $E$ to $C_{\widetilde{A}}$. Since the restrictions $E_U, \widetilde{E}^1_U, \widetilde{E}^0_U$ are trivializable on their corresponding affine schemes, we can choose isomorphisms $\psi: \widetilde{E}^0_U \xrightarrow{\sim} \widetilde{E}^1_U$ that restrict to the identity on $E_U$. The isomorphism $(p_2)^*(\psi)\circ (p_1)^*(\psi^{-1}): \widetilde{E}^0_{U^{(2)}} \to \widetilde{E}^0_{U^{(2)}}$ is of the form $1 + B$ for a unique section $B\in \Ad(\widetilde{E}^0_{U^{(2)}}) \otimes_{A} I \cong \Ad(E_{U^{(2)}}) \otimes_{A} I$. 
View $B$ as a cochain in $\widecheck{C}^1_{U} (\Ad(E))$. 
A similar computation as for $c$ shows that $B$ is a cocycle. 
Define the lift $\widetilde{\nabla}^1_U = \psi \circ \widetilde{\nabla}_U^0 \circ \psi^{-1}$ on $\widetilde{E}^1_U$. The new Cech cocycle $c^1$ for this choice of $\widetilde{E}^1$ and $\widetilde{\nabla}^1_U$ is given by
\[ c^1 =  (p_1)^*(\widetilde{\nabla}^1_U) - (p_2)^*(\widetilde{\nabla}^1_U) = (p_1)^*(\psi \circ \widetilde{\nabla}_U^0 \circ \psi^{-1}) - (p_2)^*(\psi \circ \widetilde{\nabla}_U^0 \circ \psi^{-1})\]
We have that $c^1$ is a cocycle in the submodule $\Ad(E) \otimes_{\mathcal{O}_{U^{(2)}}} \Omega^1_{U^{(2)}} \otimes_{A} I$. 
Since $\psi$ restricts to the identity on $E_U$, applying $(p_2)^*(\psi)^{-1} \circ (-) \circ (p_2)^*(\psi)$ does not affect the cocycles in $\Ad(E_{U^{(2)}}) \otimes_{\mathcal{O}_{U^{(2)}}} \Omega^1_{U^{(2)}} \otimes_{A} I$. Therefore we can rewrite
\begin{gather*}
    c^1 = (p_2)^*(\psi)^{-1}\circ c^1 \circ (p_2)^*(\psi) = (p_2)^{*}(\psi)^{-1} \circ (p_1)^*(\psi) \circ (p_1)^*(\widetilde{\nabla}_U^0) \circ (p_1)^*(\psi)^{-1} \circ (p_2)^*(\psi) - (p_2)^*(\widetilde{\nabla}_U^0)
\end{gather*}
This can be rewritten as $c^1 = (1-B) \circ (p_1)^*(\widetilde{\nabla}_U^0) \circ (1+B) - (p_2)^*(\widetilde{\nabla}_U^0)$. By expanding, we get $c^1=\Ad((p_1)^*(\widetilde{\nabla}^0_U))(B)+c.$
Since $B$ is a section of $\Ad(E_{U^{(2)}}) \otimes_{A} I$ (supported in $U^{(2)} \subset \widetilde{U}^{(2)}$) and $\widetilde{\nabla}_U^0$ is a $\widetilde{A}$-linear lift of $\nabla$, we can rewrite this as $c^1=\Ad(\nabla)(B)+c.$ In other words, $c^1 =  \widecheck{C}^1(\Ad(\nabla))(B) +c$,
where we write 
\[\widecheck{C}^1(\Ad(\nabla)): \widecheck{C}^1_{U}(\Ad(E) \otimes_{A} I) \to \widecheck{C}^1_{U}(\Ad(E) \otimes_{\mathcal{O}_C} \Omega^1_{C/S} \otimes_{A} I)\]
to denote the map induced by $\Ad(\nabla): \Ad(E) \to \Ad(E) \otimes_{O_{C}} \Omega^1_{C/S}(D)$ at the level of Cech cochains. Since $c^1$ differs from $c$ by the image of a cocycle in $\widecheck{C}^1_{U}(\Ad(E) \otimes_{A} I)$, it yields the same element $\text{ob}_{p_{A}}$ in $\mathcal{Q}_{p_{A}} \otimes_{A} I$.

\noindent \textbf{Observation 2.} The cocycle $B$ depends on the choice of isomorphisms $\psi$ up to the image of a $0$-chain, and so the corresponding cohomology class in $\widecheck{H}^1_{U}(\Ad(E) \otimes_{A} I)$ is well-defined. Conversely, by standard deformation theory of $G$-bundles (cf. \cite[\S2]{biswas-ramanan-infinitesimal}), every such cohomology class $[B]$ arises this way from a choice of lift $\widetilde{E}^1$. Hence, for every $B$ there is a lift $\widetilde{E}^1$ such that $c^1$ is cohomologous to $\widecheck{C}^1(\Ad(\nabla))(B) + c$.

\noindent \textbf{Relation to lifts:} If $\text{ob}_{p_{A}} = 0$, then there is $h \in \widecheck{C}^0_{U} (\Ad(E) \otimes_{\mathcal{O}_{C}} \Omega^1_{C/S} \otimes_{A} I)$ and $B \in \widecheck{C}^1_{U} (\Ad(E) \otimes_{A} I)$ such that $c = \delta(h) + \widecheck{C}^1(\Ad(\nabla))(B)$. By Observations 1 and 2 above, we can use $h$ and $B$ to choose new lifts $\widetilde{E}^1$ and $\widetilde{\nabla}^1_U$ such that the corresponding cocycle $c^1 =(p_1)^*(\widetilde{\nabla}^1_U)- (p_2)^*(\widetilde{\nabla}^1_U)$ vanishes. 
Therefore the section $\widetilde{\nabla}^1_U$ over $\widetilde{U}$ for the torsor $A^D_{t-\text{Conn},\widetilde{E}^1}$ satisfies a cocycle condition, and hence descends to a section over $C_{\widetilde{A}}$. 
This yields a meromorphic $t_{\widetilde{A}}$-connection $\widetilde{\nabla}$ on $\widetilde{E}$ that lifts $\nabla$.
Conversely, if there exists a lift $(\widetilde{E}, \widetilde{\nabla})$, then we can use $\widetilde{E}^0 := \widetilde{E}$ and $\widetilde{\nabla}^0_U := \widetilde{\nabla}$ to compute the cocycle $c$ representing $\text{ob}_{p_{A}}$. Since $c = (p_1)^*(\widetilde{\nabla})-(p_2)^*(\widetilde{\nabla}) = 0$, we have $\text{ob}_{p_{A}} = 0$. 
\end{proof}
\end{subsection}

\begin{subsection}{Deformation theory of compatible parabolic reductions}
In this subsection we assume that $G$ is reductive. Fix an algebraically closed field $k$ and a meromorphic $t$-connection $(E, \nabla) \in \hodge_{G}(k)$. Let $E_{P}$ be a parabolic reduction of $E$ that is compatible with $\nabla$. Consider the scheme of dual numbers $N=\Spec(k[\epsilon]/(\epsilon^2)) \to \Spec(k)$, and the pullback $(E|_{N}, \nabla|_{N}) \in \hodge_{G}(N)$ of the meromorphic $t$-connection. 

\begin{defn} \label{defn: deformations of parabolic reductions}
We denote by $\text{Def}_{\nabla}(E_{P})$ the groupoid of pairs $(\widetilde{E}_{P}, \theta)$, where $\widetilde{E}_{P}$ is a parabolic reduction of $E|_{N}$ compatible with $\nabla|_{N}$ and $\theta$ is an isomorphism of parabolic reductions $\theta: \widetilde{E_{P}}|_k \xrightarrow{\sim} E_{P}$.
\end{defn}

Since $P$ is a closed subgroup of $G_k$, it follows that the groupoid $\text{Def}_{\nabla}(E_{P})$ is equivalent to a set. We harmlessly identify $\text{Def}_{\nabla}(E_{P})$ with its set of isomorphism classes of objects. We shall describe the set of such deformations in terms of sheaf cohomology. 

Let $(\text{Ad}(E), \text{Ad}(\nabla))$ denote the corresponding adjoint vector bundle equipped with its induced connection. By $\nabla$-compatibility, the connection $\text{Ad}(\nabla): \text{Ad}(E) \to \text{Ad}(E) \otimes \Omega^1_{C_k/k}(D)$ preserves the vector subbundle $\text{Ad}(E_{P}) \subset \text{Ad}(E)$ formed using the adjoint representation of $P$. Hence it induces a well-defined $k$-linear map of sheaves $\text{Ad}(\nabla): \text{Ad}(E)/\text{Ad}(E_P) \to (\text{Ad}(E)/\text{Ad}(E_P))\otimes \Omega^1_{C_k/k}(D)$.
\begin{prop} \label{prop: parabolic reductions deformations}
Let $k$, $(E, \nabla)$ and $E_{P}$ as above. Then, the set of isomorphism classes of deformations $\text{Def}_{\nabla}(E_{P})$ is naturally isomorphic to the kernel of the morphism
\[ H^0(\text{Ad}(E)/\text{Ad}(E_P)) \xrightarrow{H^0(\text{Ad}(\nabla))} H^0((\text{Ad}(E)/\text{Ad}(E_P))\otimes \Omega^1_{C_k/k}(D)) \]
\end{prop}
\begin{proof}
We use \Cref{prop: obstruction class} with $A = k$, $\widetilde{A} = k[\epsilon]/(\epsilon^2)$ and $I =k\epsilon$. Let us denote by $\rho$ the inclusion $P \hookrightarrow G_k$. 
The $\nabla$-compatible parabolic reduction induces a point $p= (E_P, \nabla_{E_P}) \in \hodge_P$ such that $\rho_*(p) = (E, \nabla)$. There is an induced morphism of groupoids $\rho_* \text{Def}_{\widetilde{A}/k}(p) \to \text{Def}_{\widetilde{A}/k}(\rho_*(p))$. 
Both groupoids are nonempty, because there are trivial deformations obtained by pulling back to $N$. Hence the obstructions $\text{ob}_{p}, \text{ob}_{\rho_*(p)}$ from \Cref{prop: obstruction class} vanish. By definition $\text{Def}_{\nabla}(E_P)$ is the fiber of $\rho_*$ over the trivial deformation $(E|_N, \nabla|_N)$ in $\text{Def}_{\widetilde{A}/k}(\rho_*(p))$.

Consider the short exact sequence of two term complexes of vector bundles \[ 0 \to \mathcal{C}^{\bullet}_{p} \to \mathcal{C}^{\bullet}_{\rho_*(p)} \to  \mathcal{C}^{\bullet}_{\rho_*(p)}/\mathcal{C}^{\bullet}_{p} \to 0\]
where we have $\mathcal{C}^{\bullet}_{\rho_*(p)}/\mathcal{C}^{\bullet}_{p}  = \left[ \text{Ad}(E)/\text{Ad}(E_P) \xrightarrow{\text{Ad}(\nabla)} (\text{Ad}(E)/\text{Ad}(E_P))\otimes \Omega^1_{C_k/k}(D)\right]$.

By taking injective resolutions as abelian sheaves and pushing forward to $\Spec(k)$, we obtain a short exact sequence of complexes
\[ 0 \to C^{\bullet}_{p} \to C^{\bullet}_{\rho_*(p)} \to C^{\bullet}_{\rho_*(p)}/C^{\bullet}_{p}  \to 0\]
representing the distinguished triangle $R\Gamma(\mathcal{C}^{\bullet}_{p}) \to R\Gamma(\mathcal{C}^{\bullet}_{\rho_*(p)}) \to R\Gamma(\mathcal{C}^{\bullet}_{p}/\mathcal{C}^{\bullet}_{\rho_*(p)})$. Consider the associated morphism of Picard groupoids $\text{Pic}(\tau_{\leq 1}C^{\bullet}_{p}) \to \text{Pic}(\tau_{\leq 1}C^{\bullet}_{\rho_*(p)})$. By \Cref{prop: obstruction class}, there is a two-commutative diagram of groupoids
\[
\begin{tikzcd}
  \text{Def}_{\widetilde{A}/k}(p) \ar[r, "\sim"] \ar[d, "\rho_*"] &  \text{Pic}(\tau_{\leq 1} C^{\bullet}_{p_A}) \ar[d] \\ \text{Def}_{\widetilde{A}/k}(\rho_*(p)) \ar[r, "\sim"] & \text{Pic}(\tau_{\leq 1} C^{\bullet}_{\rho_*(p)})
\end{tikzcd}
\]
Here the trivial deformation $(E|_N, \nabla|_N)$ of $\rho_*(p)$ corresponds to the zero cocycle of $\tau_{\leq 1} C^{\bullet}_{\rho_*(p)}$. Hence $\text{Def}_{\nabla}(E_P)$ is identified with the $0$-fiber under the functor $\text{Pic}(\tau_{\leq 1}C^{\bullet}_{p}) \to \text{Pic}(\tau_{\leq 1}C^{\bullet}_{\rho_*(p)})$, and we have reduced our question to a statement of linear algebra of $k$-vector space. 
A diagram chase shows that the isomorphism classes of objects of the $0$-fiber are in bijection with elements of the zero-cohomology group $H^0(C^{\bullet}_{\rho_*(p)}/C^{\bullet}_{p})$, which by definition is the hypercohomology group $\mathbb{H}^0(\mathcal{C}^{\bullet}_{\rho_*(p)}/\mathcal{C}^{\bullet}_{p} )$. The hypercohomology spectral sequence for the two term complex $\mathcal{C}^{\bullet}_{\rho_*(p)}/\mathcal{C}^{\bullet}_{p} $ identifies $\mathbb{H}^0(\mathcal{C}^{\bullet}_{\rho_*(p)}/\mathcal{C}^{\bullet}_{p} )$ with the kernel of the morphism $H^0(\text{Ad}(\nabla))$ as in the statement of the proposition.
\end{proof}
\end{subsection}

\end{section}
\begin{section}{\texorpdfstring{$\Theta$}{Theta}-stratification and semistable reduction}
In this section, we study some more refined structure of $\hodge_G$. 
\Cref{subsec: review invariants} is a concise review of Halpern-Leistner's notion of numerical invariants and $\Theta$-stratifications.
They provide the framework in which we study instability for $\hodge_G$.
In \Cref{subsec: semistable conditions agree}, we show that the semistablity defined by a suitable numerical invariant agrees with the usual notion of semistability for $\hodge_G$.
In \Cref{subsec: mono and hn bound}, we establish two technical properties, which are crucial in constructing the $\Theta$-stratification on $\hodge_G$ and proving the semistable reduction results in \Cref{sec: semistable reduction}.
Finally, in \Cref{subsec: smooth}, we show the smoothness of the semistable locus in the case when $D$ is nonempty.

\begin{subsection}{Review of the theory of numerical invariants and \texorpdfstring{$\Theta$}{Theta}-stratifications}\label{subsec: review invariants}
    We shall use the theory of numerical invariants and $\Theta$-stratifications to develop Harder-Narasimhan theory for $\hodge_{G}$. We refer the reader to \cite[\S 2.4]{gauged_theta_stratifications} for the necessary background and \cite{halpernleistner2018structure} for more details.

    The main ``test object" used in developing the abstract stack-theoretic approach to Harder-Narasimhan theory is the following.
    \begin{defn}[The stack $\Theta$]
        We denote by $\Theta$ the quotients stack $\mathbb{A}^1_{\mathbb{Z}}/\mathbb{G}_m = \Spec(\mathbb{Z}[t])/\mathbb{G}_m$, where $\mathbb{G}_m$ acts via its standard linear contracting action on $\mathbb{A}^1_{\mathbb{Z}}$ which induces weight $-1$ on the coordinate function $t$. For any scheme $T$, we set $\Theta_T:= \Theta \times _{\Spec(\mathbb{Z})} T$.
    \end{defn}
    
    \begin{notn}
        Given a scheme $T$, there is an open immersion $1: T \hookrightarrow \Theta_T$ corresponding to the orbit of the section $T \to \mathbb{A}^1_T$ induced by the function $1$ on $T$. Similarly, there is a section $0: T \to \Theta_T$ corresponding to the function $0$. We note that $0: T \to \Theta_T$ is the composition of the universal $\mathbb{G}_m$-bundle $T \to B\mathbb{G}_{m,T}$ followed by a closed immersion $0: B \mathbb{G}_{m,T} \hookrightarrow \Theta_T$, where we abuse the notation by letting $0$ denote two morphisms.
    \end{notn}

    For the rest of this subsection, we fix an algebraic stack $\mathcal{X} \to S$ which is locally finite type and with affine relative diagonal over a Noetherian base scheme. 
    
    The following notion of filtration introduced in \cite[\S 1]{halpernleistner2018structure} and \cite{heinloth-hilbertmumford} should be thought of as an abstract generalization of (weighted) parabolic reduction of $G$-bundles on a curve to the more general setting of an arbitrary moduli stack $\mathcal{X}$ (see \Cref{prop: filtrations of Bun_G} below for a more precise explanation of this).
    \begin{defn}[Filtrations of points] \label{defn: filtration point stack} Let $k$ be a field over $S$, and let $p \in \mathcal{X}(k)$. A filtration of $p$ consists of the data of a morphism $f: \Theta_k \to \mathcal{X}$ along with an isomorphism $\psi: f(1) \xrightarrow{\sim} p$. We denote by $\text{Filt}(p)$ the set of isomorphism classes of filtrations of $p$.
    \end{defn}
    We generally leave the choice of $\psi$ implicit and refer to $f$ as the filtration of the point $p$.
\begin{defn}
    We denote by $\text{Filt}(\mathcal{X}) = \Map_S(\Theta_S, \, \mathcal{X})$ the stack of filtrations of points in $\mathcal{X}$. The stack $\text{Filt}(\mathcal{X})$ is an algebraic stack with affine relative diagonal and locally of finite type over $S$ \cite[Prop. 1.1.2]{halpernleistner2018structure}.
\end{defn} 
There is a morphism $\text{ev}_1: \text{Filt}(\mathcal{X}) \to \mathcal{X}$ given by evaluating at $1: S \to \Theta_S$.

\begin{defn}[(Weak) $\Theta$-strata]
    Let $\mathcal{U}$ be an open substack of $\mathcal{X}$. A weak $\Theta$-stratum of $\mathcal{U}$ is a union of connected components of $\text{Filt}(\mathcal{U})$ such that the restriction of $\text{ev}_1$ is finite and radicial. A $\Theta$-stratum is a weak $\Theta$-stratum such that $\text{ev}_1$ is a closed immersion.
\end{defn}
 We can think of a $\Theta$-stratum as a closed substack of $\mathcal{U}$ that is identified with some connected components of $\text{Filt}(\mathcal{U})$. Hence the closed substack has a ``moduli-theoretic interpretation": it parametrizes pairs $(p,f)$ where $p$ is point of $\mathcal{U}$ and $f$ is a filtration of $p$ satisfying certain open and closed conditions in the stack of filtrations. 

 \begin{defn}[(Weak) $\Theta$-stratification] \label{defn: theta stratification}
A (weak) $\Theta$-stratification of $\mathcal{X}$ consists of a collection of open substacks $(\mathcal{X}_{\leq c})_{c \in \Gamma}$ indexed by a totally ordered set $\Gamma$. We require the following conditions to be satisfied
\begin{enumerate}[(1)]
\item $\mathcal{X}_{\leq c} \subset \mathcal{X}_{\leq c'}$ for all $c< c'$.
\item $\mathcal{X} = \bigcup_{c \in \Gamma} \mathcal{X}_{\leq c}$.
\item For all $c$, there exists a (weak) $\Theta$-stratum $\mathfrak{S}_c \subset \text{Filt}(\mathcal{X}_{\leq c})$ of $\mathcal{X}_{\leq c}$ such that
\[ \mathcal{X}_{\leq c} \setminus \text{ev}_1(\mathfrak{S}_c) = \bigcup_{c' < c} \mathcal{X}_{\leq c'}\]
\item For every point $p \in \mathcal{X}$, the set $\left\{ c \in \Gamma \, \mid \, p \in \mathcal{X}_{\leq c}\right\}$ has a minimal element.
\end{enumerate}
\end{defn}
The main motivation behind the definition of $\Theta$-stratification is to abstract a notion ``moduli-theoretic" locally closed stratifications of the stack $\mathcal{X}$, where each stratum satisfies certain strong structural properties reminiscent of the stratifications in GIT arising in the work of Hesselink, Kempf, Kirwan and Ness, as well as the classical Harder-Narasimhan stratification of the stack of $G$-bundles on a curve arising in the work of Atiyah-Bott \cite{atiyah-bott-yangmills}.

One of the main insights in \cite{halpernleistner2018structure} is the definition of a general notion of numerical invariant on a stack $\mathcal{X}$. Such numerical invariants can be used in practice for two purposes:
\begin{enumerate}[(1)]
    \item defining a locus of ``semistable points" in $\mathcal{X}$, and
    \item constructing an induced Harder-Narasimhan $\Theta$-stratification of $\mathcal{X}$.
\end{enumerate}

We proceed to review the theory of numerical invariants on $\mathcal{X}$. We will need the following.

\begin{defn}[(Nondegenerate) Graded points]\label{defn: nondegenerate}
    Let $q \geq 1$ be a positive integer. Let $k$ be a field over $S$. A ($k$-valued) $q$-graded point of the stack $\mathcal{X}$ is a morphism $B\mathbb{G}^q_{m,k} \rightarrow \mathcal{X}$. We say that a morphism $g: B\mathbb{G}^q_{m,k} \rightarrow \mathcal{X}$ is nondegenerate if the induced homomorphism $\gamma: \mathbb{G}_{m,k}^q \rightarrow \text{Aut}(g({\Spec(k))})$ has finite kernel.
\end{defn}
Similarly as in the case of filtrations, the mapping stack of $q$-graded points $\text{Grad}^q(\mathcal{X}) \vcentcolon = \text{Map}_S(B\mathbb{G}_{m,S}^q, \, \mathcal{X})$ is an algebraic stack locally of finite presentation over $S$ \cite[Prop. 1.1.2]{halpernleistner2018structure}. 

 A ($\mathbb{R}$-valued) numerical invariant is an assignment of a function $\nu_{\gamma}: \mathbb{R}^{q} \setminus \{0\} \to \mathbb{R}$ for each nondegenerate $g: B\mathbb{G}^q_{m,k} \rightarrow \mathcal{X}$, satisfying some compatibility conditions as follows.
\begin{defn} \label{defn: general numerical invariant}
A numerical invariant $\nu$ on the stack $\mathcal{X}$ is an assignment defined as follows. Let $k$ be a field and let $p \in \mathcal{X}(k)$. Let $\gamma: \mathbb{G}_{m,k}^q \rightarrow \text{Aut}(p)$ be a homomorphism of $k$-groups with finite kernel. 
The numerical invariant $\nu$ assigns to this data a scale-invariant function $\nu_{\gamma}: \mathbb{R}^q\setminus \{0\} \rightarrow \mathbb{R}$ such that
\begin{enumerate}[(1)]
    \item $\nu_{\gamma}$ is unchanged under field extensions $k \subset k'$.
    \item $\nu$ is locally constant in algebraic families. In other words, let $T$ be a scheme. Let $\xi: T \rightarrow \mathcal{X}$ be a morphism and let $\gamma: \mathbb{G}_{m,T}^{q} \to \text{Aut}(\xi)$ be a homomorphism of $T$-group schemes with finite kernel. Then, as we vary $t \in T$, the function $\nu_{\gamma_{t}}$ is locally constant in $T$.
    \item Given a homomorphism $\phi: \mathbb{G}^w_{m,k} \rightarrow \mathbb{G}^q_{m,k}$ with finite kernel, the function $\nu_{\gamma \circ \phi}$ is the restriction of $\nu_{\gamma}$ along the inclusion $\mathbb{R}^w \hookrightarrow \mathbb{R}^q$ induced by $\phi$.
\end{enumerate}
\end{defn}

For any field $k$ over $S$, we say that a filtration $f: \Theta_{k} \rightarrow \mathcal{X}$ of some $k$-point $p$ is nondegenerate if the restriction $f|_0: B\mathbb{G}_{m,k} \hookrightarrow \Theta_k \to \mathcal{X}$ is nondegenerate as in \Cref{defn: nondegenerate}.
\begin{notn}
    Given a numerical invariant $\nu$ on $\mathcal{X}$, we regard $\nu$ as a function $|\text{Filt}^{\mathrm{ndegn}}(\mathcal{X})| \to \mathbb{R}$ on the set of nondegenerate filtrations by defining $\nu(f) \vcentcolon = \nu_{f|_{0}}(1) \in \bR$.
\end{notn}

\begin{defn}[Semistability] \label{defn: semistability numerical invariant}
    Let $\nu$ be a numerical invariant on $\mathcal{X}$. Let $k$ be an algebraically closed field over $S$. Let $p \in \mathcal{X}(k)$. We say that $p$ is semistable if all nondegenerate filtrations $f$ of $p$ satisfy $\nu(f) \leq 0$. Otherwise, we say that $p$ is unstable. 
\end{defn}

\begin{remark}
We will mostly consider nondegenerate filtrations, because these are the ones relevant for stability. We will sometimes omit the adjective ``nondegenerate".
\end{remark}

Note that although \Cref{defn: general numerical invariant} involves data for all $q \geq 1$, only the $q=1$ data is used to define semistability.

Our next goal is to explain a way to define numerical invariants for a given stack $\mathcal{X}$ in practice, in terms of a line bundle on $\mathcal{X}$ and a so-called ``norm on graded points". 

\begin{notn}
    Let $L$ be a line bundle on $\mathcal{X}$. Given a morphism $g: B\mathbb{G}^q_{m,k} \rightarrow \mathcal{X}$, the pullback line bundle $g^*(L)$ amounts to a character in $X^*(\mathbb{G}^q_{m,k})$. Under the natural identification $X^*(\mathbb{G}^q_{m,k}) \cong \mathbb{Z}^q$, we can interpret this as a $q$-tuple of integers $(w^{(i)})_{i=1}^q$, which we call the weight of $g^*(\mathcal{L})$. 
\end{notn}  

\begin{notn}
    When $q=1$, we use the notation $\wt(g^*L) := w^{(1)}$ to denote the unique weight as above. If $f: \Theta_k \to \mathcal{X}$ is a filtration of a point in $\mathcal{X}$, then we may use the notation $\wt(L)(f) = \wt((f|_0)^*(L))$.
\end{notn}

\begin{notn}
    Given a line bundle $L$ on $\mathcal{X}$, we define for every $g: B\mathbb{G}^q_{m,k} \rightarrow \mathcal{X}$ an $\mathbb{R}$-linear function $\wt(L)_g : \mathbb{R}^q \to \mathbb{R}$ given by
\[ \wt(L)_g\left((r_i)_{i=1}^q\right) = \sum_{i=1}^q r_i \cdot w^{(i)} .\]
\end{notn}
Notice that the assignment $\wt(L): g \mapsto \wt(L)_g$ satisfies most of the properties of a numerical invariant (\Cref{defn: general numerical invariant}), except that the functions $\wt(L)_g$ are not scale invariant. It is important for $\nu$ to be scale invariant in order to set up the optimization problem that will define the associated Harder-Narasimhan $\Theta$-stratification. In order to obtain a scale invariant $\nu$, we use a rational quadratic norm on graded points as in \cite[Defn. 4.1.12]{halpernleistner2018structure}. 

\begin{defn}[Rational quadratic norm on graded points]
\label{defn: rational_quadratic_norm}
A rational quadratic norm $b$ on graded points of the stack $\mathcal{X}$ is an assignment defined as follows. Let $k$ be a field over $S$ and let $p \in \mathcal{X}(k)$. Let $\gamma: \mathbb{G}_{m,k}^q \to \text{Aut}(p)$ be a homomorphism with finite kernel. Then $b$ assigns to this data a positive definite quadratic norm $b_{\gamma}: \mathbb{R}^q \to \mathbb{R}$ with rational coefficients such that
\begin{enumerate}[(1)]
    \item $b_{\gamma}$ is unchanged under field extensions $k \subset k'$.
    \item $b$ is locally constant in algebraic families. In other words, choose any scheme $T$, a morphism $\xi: T \rightarrow \mathcal{X}$ and a homomorphism $\gamma: \mathbb{G}_{m,T}^{q} \rightarrow \text{Aut}(\xi)$ of $T$-group schemes with finite kernel. Then, as we vary $t \in T$, we require that the function $b_{\gamma_{t}}$ is locally constant in $T$.
    \item Given a homomorphism $\phi: \mathbb{G}^w_{m,k} \rightarrow \mathbb{G}^q_{m,k}$ with finite kernel, the norm $b_{\gamma \circ \phi}$ is the restriction of $b_{\gamma}$ along the inclusion $\mathbb{R}^w \hookrightarrow \mathbb{R}^q$ induced by $\phi$.
\end{enumerate}
\end{defn}

\begin{notn} \label{nont: numerical invariant in terms of line bundle and norm on graded points}
    Given a line bundle $L$ on $\mathcal{X}$ and a rational quadratic norm on graded points $b$, we denote by $\nu = \wt(L)/\sqrt{b}$ the numerical invariant on $\mathcal{X}$ defined as follows. For all nondegenerate $g: B\mathbb{G}^q_{m,k} \to \mathcal{X}$ with corresponding morphism $\gamma: \mathbb{G}^q_{m,k} \to \text{Aut}(g(\Spec(k)))$, we set
\[ \nu_{\gamma}(\vec{r}) = \frac{\wt(L)_{g}(\vec{r})}{\sqrt{b_{\gamma}(\vec{r})}}  \]
\end{notn}

Our next goal is to explain how to use a numerical invariant $\nu$ to define a (weak) $\Theta$-stratification on $\mathcal{X}$ (see \cite[\S 4.1]{halpernleistner2018structure} for more details). 

\begin{notn}
    For any unstable geometric point $p: \Spec(k) \to \mathcal{X}$, we set $M^{\nu}(p)$ to be the supremum of $\nu(f)$ over all filtrations $f$ of $p$ (if such supremum exists). If $p$ is semistable, then by convention we set $M^{\nu}(p) = 0$.
\end{notn}

\begin{defn}[Harder-Narasimhan filtration]
    For any unstable geometric point $p$ of $\mathcal{X}$, a filtration $f$ of $p$ is called a Harder-Narasimhan filtration if $\nu(f) = M^{\nu}(p)$.
\end{defn}

\begin{notn}
    For any $c \in \mathbb{R}_{\geq 0}$, we set $\mathcal{X}_{\leq c}$ to be the set of all points $p$ satisfying $M^{\nu}(p) \leq c$, and we let $\cX^{\nu\dash\rm{ss}} := \cX_{\leq 0}$ denote the set of semistable points. 
\end{notn}

\begin{defn} \label{defn: numerical invariant Theta stratification}
    We say that a numerical invariant $\nu$ defines a (weak) $\Theta$-stratification if the following conditions hold:
    \begin{enumerate}[(1)]
        \item every unstable geometric point has a Harder-Narasimhan filtration that is unique up to pre-composing with a ramified covering $\Theta_k \xrightarrow{[n]} \Theta_k$.
        \item For each $c\in\mathbb{R}_{\geq 0}$, the set $\cX_{\leq c}$ underlies an open substack of $\cX$, which comes from a (weak) $\Theta$-stratification as in \Cref{defn: theta stratification}. Furthermore, each (weak) stratum $\mathfrak{S}_c \subset \Filt(\cX_{\leq c})$ is an open and closed substack of Harder-Narasimhan filtrations $f$ with $\nu(f)=c$.
    \end{enumerate} 
\end{defn} 

If a numerical invariant $\nu$ defines a $\Theta$-stratification, then the Harder-Narasimhan filtration of any point is defined over the field of definition of that point. However, if it is only a weak $\Theta$-stratification, in general the Harder-Narasimhan filtration is only defined over a finite purely inseparable field extension.

Next, we explain a set of intrinsic criteria developed in \cite[\S5]{halpernleistner2018structure} that guarantee that a given numerical invariant defines a weak $\Theta$-stratification as in \Cref{defn: numerical invariant Theta stratification}. This is contained in the following theorem, whose hypotheses will be explained below.
\begin{thm}[{\cite[Theorem B(1)]{halpernleistner2018structure}}] \label{thm: theta stability paper theorem}
Let $\nu = \wt(L)/\sqrt{b}$ be a numerical invariant on $\mathcal{X}$ defined by a line bundle $L$ and a rational norm on graded points $b$ as in \Cref{nont: numerical invariant in terms of line bundle and norm on graded points}. If $\nu$ is strictly $\Theta$-monotone, then it defines a weak $\Theta$-stratification of $\mathcal{X}$ if and only if it satisfies the HN boundedness condition.
\end{thm}

We end this subsection by explaining the hypotheses in Theorem \ref{thm: theta stability paper theorem}. 

\begin{notn}[Weighted lines]
    Let $\kappa$ be a field and let $a \geq 1$ be an integer. We denote by $\mathbb{P}^1_{\kappa}[a]$ the $\mathbb{G}_m$-scheme $\mathbb{P}^{1}_{\kappa}$ equipped with the $\mathbb{G}_m$-action determined by the equation $t \cdot [x:y] = [t^{-a}x : y]$. We set $0 = [0: 1]$ and $\infty = [1:0]$.
\end{notn} 

\begin{notn}[``Rigidified" $\Theta_R$] \label{notn: rigidified theta and str}
Let $R$ be a complete discrete valuation ring over $S$. Choose a uniformizer $\varpi$ of $R$. We define $Y_{\Theta_{R}} \vcentcolon = \mathbb{A}^1_{R}$ equipped with its standard $\mathbb{G}_m$-action. By definition, we have $\Theta_{R} =\mathbb{A}^1_{R} / \mathbb{G}_m$. Note that $\mathbb{A}^1_{R}$ contains a unique $\mathbb{G}_m$-invariant closed point cut out by the ideal $(t, \varpi)$. We will denote this point by $\mathfrak{o}$.
\end{notn}

\begin{defn}[Strict $\Theta$-monotonicity] \label{defn: strictly theta monotone and STR monotone} A numerical invariant $\nu$ on $\mathcal{X}$ is strictly $\Theta$-monotone if the following condition holds. 

Let $R$ be any complete discrete valuation ring over $S$. Choose a map $\varphi: \Theta_R \setminus \mathfrak{o} \rightarrow \mathcal{X}$. Then, after possibly replacing $R$ with a finite DVR extension, there exists a reduced and irreducible $\mathbb{G}_m$-equivariant scheme $\Sigma$ with maps $f: \Sigma \rightarrow Y_{\Theta_R}$ and $\widetilde{\varphi}: \Sigma/ \mathbb{G}_m \rightarrow  \mathcal{X}$ such that
\begin{enumerate}[({M}1)]
    \item The map $f$ is proper, $\mathbb{G}_m$-equivariant, and its restriction induces an isomorphism $f : \, \Sigma_{Y_{\Theta_R} \setminus \mathfrak{o}} \xrightarrow{\sim} Y_{\Theta_R} \setminus \mathfrak{o}$.
    \item The following diagram commutes
\begin{figure}[H]
\centering
\begin{tikzcd}
  \left(\Sigma_{Y_{\Theta_R} \setminus \mathfrak{o}}\right)/ \, \mathbb{G}_m \ar[rd, "\widetilde{\varphi}"] \ar[d, "f"'] & \\   \Theta_R \setminus \mathfrak{o} \ar[r, "\varphi"'] &  \mathcal{X}
\end{tikzcd}
\end{figure}
    \item Let $\kappa$ denote a finite extension of the residue field of $R$. For any $a \geq 1$ and any finite $\mathbb{G}_m$-equivariant morphism $\mathbb{P}^1_{\kappa}[a] \to \Sigma_{\mathfrak{o}}$, we have $\nu\left( \;\widetilde{\varphi}|_{\infty / \mathbb{G}_m} \;\right) >  \nu\left(\; \widetilde{\varphi}|_{0 / \mathbb{G}_m} \;\right)$.
\end{enumerate}
\end{defn}

We use the following version of HN boundedness as in \cite[Def. 2.29]{gauged_theta_stratifications}.
\begin{defn}[HN Boundedness]\label{defn: HN boundedness}
We say that a numerical invariant $\nu$ on $\mathcal{X}$ satisfies the HN boundedness condition if the following condition is always satisfied:

For any affine Noetherian scheme $T$ over $S$, and any $g\in \mathcal{X}(T)$, there exists a quasi-compact open substack $\mathcal{U}_{T} \subset \mathcal{X}$ such that the following holds.
For all geometric points $t \in T$ with residue field $k$ and all nondegenerate filtrations $f: \Theta_{k} \rightarrow \mathcal{X}$ of the point $g(t)$ with $\nu(f)>0$, there exists another filtration $f'$ of $g(t)$ satisfying $\nu(f') \geq \nu(f)$ and $f'(0) \in \mathcal{U}_{T}$.
\end{defn}

\end{subsection}

\begin{subsection}{Semistability as \texorpdfstring{$\Theta$}{Theta}-semistability}\label{subsec: semistable conditions agree}
For the rest of this section, we make the following assumption.
\begin{context}
    We assume that $G \to S$ is a split reductive group.
\end{context}

Fix the choice of a maximal split torus $T \subset G$ over $S$. 
The constant Weyl group scheme $W$ acts on $T$, and therefore on its group of characters $X^*(T)$. We fix once and for all a choice of a $W$-invariant rational quadratic norm on the vector space of real characters $X_*(T)_{\mathbb{R}}$. By \cite[Lem. 2.20]{gauged_theta_stratifications} this defines a rational quadratic norm on graded points of the stack $BG$ (as in \Cref{defn: rational_quadratic_norm}).
\begin{defn}
We denote by $b$ the rational quadratic norm on graded points of $\hodge_{G}$ obtained by pullback via the composition $\hodge_{G} \xrightarrow{Forget} \Bun_{G} \to BG$. 
\end{defn}

\begin{defn} \label{defn: determinant line bundle}
We define the line bundle $\cD(\mathfrak{g})$ on the stack $\Bun_{G}(C)$ as follows. Let $E_{univ}$ denote the universal $G$-bundle on $C \times_S \Bun_{G}(C)$ and let $\pi: C \times_S \Bun_{G}(C) \to \Bun_{G}(C)$ denote the second projection. Then, we set $\cD(\mathfrak{g}) := \text{det}\left( R\pi_* \text{Ad}(E_{univ}) \right)^{\vee}$,
where we use the usual definition of determinant $\text{det}$ of the perfect complex $R\pi_* \text{Ad}(E_{univ})$. We shall also denote by $\cD(\mathfrak{g})$ the line bundle on $\hodge_{G}$ obtained by pulling back via the composition $\hodge_{G} \to \Bun_{G}(C) \times_{S} \mathbb{A}^1_{S} \xrightarrow{\text{pr}_1} \Bun_{G}(C)$.
\end{defn}

\begin{defn} \label{defn: numerical invariant on hodge}
We define a numerical invariant $\mu$ on $\hodge_{G}$ given by $\mu:= -\wt(\cD(\mathfrak{g}))/\sqrt{b}$ (see \Cref{nont: numerical invariant in terms of line bundle and norm on graded points}).
\end{defn}
We shall show that the notion of semistability in \Cref{defn: classical semistability} agrees with semistability for the numerical invariant $\mu$ as in \Cref{defn: semistability numerical invariant}. We start by describing filtrations of points in $\Bun_{G}(C)$ (in the sense of \Cref{defn: filtration point stack}).
\begin{defn}[{\cite[Def. 3.13]{rho-sheaves-paper}}]
Let $k$ be a field and let $E$ be a $G$-bundle in $\Bun_{G}(C)(k)$. A weighted parabolic reduction of $E$ consists of a pair $(E_P, \lambda)$ of a reduction of structure group $E_{P}$ of $E$ into some parabolic subgroup $P \subset G_k$ and a $P$-dominant cocharacter $\lambda: \mathbb{G}_m \to G_k$.
\end{defn}
Since $G$ is split, we have the following description of filtrations of $\Bun_{G}(C)$.
\begin{prop}[{\cite[Lem. 1.13]{heinloth-hilbertmumford}}] \label{prop: filtrations of Bun_G}
Let $E$ be a geometric point of $\Bun_{G}(C)$. There is a bijection between the set of isomorphism classes of filtrations of $E$ and the set of weighted parabolic reductions of $E$ up to conjugation. \qed
\end{prop}
Since the forgetful morphism $\hodge_{G} \to \Bun_{G}(C)$ is representable and separated (\Cref{prop: forget is affine and finite type}), for any $k$-point $p:=(E, \nabla)$ there is an induced inclusion of isomorphism classes of filtrations $\text{Filt}(p) \hookrightarrow \text{Filt}(E)$. The next proposition determines this subset of $\text{Filt}(E)$.
\begin{prop} \label{prop: filtrations of hodge}
Let $f$ be a filtration of the $G$-bundle $E$, corresponding to a weighted parabolic reduction $(E_{P}, \lambda)$. Then $f$ lifts to a unique filtration of $(E, \nabla)$ if and only if $E_{P}$ is compatible with $\nabla$. 
\end{prop}
\begin{proof}
Let $(E_P, \lambda)$ be a weighted parabolic reduction of $E$, corresponding to a $\mathbb{G}_{m,k}$-equivariant $G$-bundle $\widetilde{E}$ on $C_{\mathbb{A}^1_{k}}$. The induced $\mathbb{G}_{m,k}$-equivariant vector bundle $\text{Ad}(\widetilde{E})$ on $C_{\mathbb{A}^1_{k}}$ corresponds, via the Rees construction (cf. \cite[Prop. 1.0.1]{halpernleistner2018structure} \cite[Prop. 3.8]{rho-sheaves-paper}), to a $\mathbb{Z}$-weighted decreasing sequence $(\text{Ad}(E)_i)_{i \in \mathbb{Z}}$ of vector subbundles $\text{Ad}(E)_i \subset \text{Ad}(E)$. By tracing back the Rees construction for $G$-bundles in \cite[1.F.b]{heinloth-hilbertmumford}, we have $\text{Ad}(E)_0 = \text{Ad}(E_P) \subset \text{Ad}(E)$. On the other hand, the $\mathbb{G}_{m,k}$-equivariant structure sheaf $\cO_{C_{\mathbb{A}^1_{k}}}$ corresponds via the Rees construction to the $\mathbb{Z}$-indexed filtration by vector subbundles $(\cO_{C_{k}})_i \subset \cO_{C_{k}}$ defined by
\[ (\cO_{C_{k}})_i := \begin{cases} 0 & \text{if $i >0$} \\ \cO_{C_{k}} & \text{if $i \leq 0$} \end{cases} \]
The $\mathbb{G}_{m,k}$-equivariant meromorphic $t$-Atiyah bundle $\text{At}^{D}(\widetilde{E})$ defines a $\mathbb{Z}$-indexed filtration of $\text{At}^D(E)$ such that the meromorphic $t$-Atiyah sequence
\[ 0 \to \text{Ad}(E)(D) \otimes \Omega^1_{C_{k}/k} \to \text{At}^{D}(E) \to \cO_{C_k} \to 0\]
is compatible with all the induced fitrations. The $0^{th}$ piece of the filtration is given by the meromorphic $t$-Atiyah sequence for $E_{P}$.
We have that $\nabla$ yields a $\mathbb{G}_{m,k}$-equivariant splitting of the restriction of the meromorphic $t$-Atiyah sequence of $\widetilde{E}$ to $\mathbb{A}^1_{k}\setminus 0$. The filtration of $E$ corresponding to $(E_{P}, \lambda)$ lifts to a filtration of the point $(E, \nabla) \in \hodge_{G}(k)$ if and only if this splitting $\nabla$ extends to the whole $C_{\mathbb{A}^1_{k}}$. This is equivalent, via the Rees construction, to requiring that the splitting $\nabla: \cO_{C_{k}} \to \text{At}^{D}(E)$ factors through the $0^{th}$-filtered piece $\text{At}^{D}(E)_0 \subset \text{At}^{D}(E)$. Since $\text{At}^{D}(E)_0 = \text{At}^{D}(E_{P})$, this is equivalent to the parabolic reduction $E_{P}$ being compatible with $\nabla$.
\end{proof}

\begin{prop} \label{prop: comparison classical and Theta semistability}
Let $p=(E, \nabla)$ be a geometric point of $\hodge_{G}$. Then $p$ is semistable as in Definition \ref{defn: classical semistability} if and only if it is $\mu$-semistable for the numerical invariant $\mu$ in \Cref{defn: numerical invariant on hodge}.
\end{prop}
\begin{proof}
This is a consequence of the computation of the weight of the line bundle $\cD(\mathfrak{g}) \in \text{Pic}(\Bun_{G}(C))$ done in \cite[1.F.c]{heinloth-hilbertmumford} and \Cref{prop: filtrations of hodge}.
\end{proof}

As a consequence of \Cref{prop: comparison classical and Theta semistability}, the semistable locus of $\hodge_{G}$ with respect to the numerical invariant $\mu$ coincides with the ``classical" semistable locus in the sense of Definition \ref{defn: classical semistability}.
\end{subsection}
\begin{subsection}{Monotonicity and HN boundedness}\label{subsec: mono and hn bound}
We use the technique of ``infinite dimensional GIT" \cite{torsion-freepaper, rho-sheaves-paper, gauged_theta_stratifications} to show the following.
\begin{prop} \label{prop: monotonicity general}
Suppose that $G$ satisfies the central property (Z) (\Cref{defn: central property}). Then the numerical invariant $\mu$ is strictly $\Theta$-monotone (as in \Cref{defn: strictly theta monotone and STR monotone}).
\end{prop}

Let $\overline{{G}}$ be as in Definition \ref{defn: isogeny}.
The representation $\overline{G} \hookrightarrow  \GL(\mathfrak{g}) \times G_{ab} \to \GL(\mathfrak{g})$, can be used to define a line bundle $\cD(\mathfrak{g})$ on $\hodge_{\overline{G}}$, given by the dual of the determinant of the derived pushforward of the universal associated vector bundle on $C\times_S \mathrm{Bun}_{GL(\Fg)}(C)$ with fibers isomorphic to $\mathfrak{g}$, as in \Cref{defn: determinant line bundle}. 
Notice that the image $\overline{T} \subset \overline{G}$ of the maximal split torus $T$ is itself a maximal split torus of $\overline{G}$. The restriction $T \to \overline{T}$ is an isogeny of tori, and so it induces an isomorphism $X_*(T)_{\mathbb{R}} \to X_*(\overline{T})_{\mathbb{R}}$ of real cocharacters compatible with the Weyl group actions. We use this to view $b$ as a Weyl invariant quadratic norm on $X_*(\overline{T})_{\mathbb{R}}$.
\begin{defn}
The numerical invariant $\overline{\mu}$ on $\hodge_{\overline{G}}$ is defined to be $\overline{\mu} = \frac{-\wt(\cD(\mathfrak{g}))}{\sqrt{b}}$.
\end{defn}
Note that the numerical invariant $\overline{\mu}$ is just the result of applying \Cref{defn: numerical invariant on hodge} for the group scheme $\overline{G}$. By definition, the numerical invariant $\mu$ on $\hodge_{G}$ is the pullback of $\overline{\mu}$ under the natural morphism $\hodge_{G} \to \hodge_{\overline{G}}$. We shall show the monotonicity for $\overline{\mu}$ first, and then use it afterward to conclude monotonicity of $\mu$ on $\hodge_{G}$. 

We need the following notion from \cite{torsion-freepaper, gauged_theta_stratifications}. 
Choose an $\mathbb{A}^1_{S}$-scheme $T$, and a $T$-flat Cartier divisor $Z \subset C_{T}$. 
Fix a $\GL(\mathfrak{g})$-bundle $\cE$ on $C_{T}$.
\begin{defn}
Let $\text{Gr}_{\cE,Z}$ be the functor from affine $T$-schemes into sets that sends $X \to T$ to the set of isomorphism classes of tuples $(\cF, \psi)$ consisting of a $\GL(\mathfrak{g})$-bundle $\cF$ on $C_{X}$ and an isomorphism of restrictions $\psi: \cF|_{C_{X} \setminus Z_{X}} \xrightarrow{\sim} \cE|_{C_{X} \setminus Z_{X}}$.
\end{defn}

Let $E$ be a $\overline{G}$-bundle on $C_{T} \setminus D$ equipped with a meromorphic $t$-connection $\nabla$. 
Assume that $\rho_*(E)$ is the restriction $(\cE, \cH)|_{C_{T} \setminus Z_T}$ of some $\GL(\mathfrak{g})\times G_{ab}$-bundle $(\cE, \cH)$ on $C_{T}$.
\begin{defn} \label{defn: affine grassmannian for log connections}
Let $\text{Gr}_{E, \nabla, Z}$ be the functor from affine $T$-schemes into sets that sends $X \to T$ to the set of isomorphism classes of tuples $(\cF, \psi, F, \widetilde{\nabla})$, where 
\begin{enumerate}[(1)]
    \item $(\cF, \psi)$ is an $X$-point of $\text{Gr}_{\cE, Z}$.
    \item $F$ is a $\overline{G}$-reduction of structure group of $(\cF, \cH)$ such that $\psi$ induces an isomorphism of restrictions $F|_{C_{X} \setminus Z_{X}} \xrightarrow{\sim} E|_{C_{X} \setminus Z_{X}}$. 
    \item $\widetilde{\nabla}$ is a meromorphic $t$-connection on $F$ that restricts on $C_{X} \setminus Z_{X}$ to $\nabla|_{C_{X} \setminus Z_{X}}$ under the identification $F|_{C_{X} \setminus Z_{X}} \xrightarrow{\sim} E|_{C_{X} \setminus Z_{X}}$.
\end{enumerate}
\end{defn}
There exists a morphism of pseudofunctors $\text{Gr}_{E, \nabla, Z} \to \hodge_{G} \times_{\mathbb{A}^1_{S}} T$ given by $(\cF, \psi, F, \widetilde{\nabla}) \mapsto (F, \widetilde{\nabla})$. We also denote by $\cD(\mathfrak{g})$ the pullback of $\cD(\mathfrak{g})$ to $\text{Gr}_{E, \nabla, Z}$.
\begin{lemma} \label{lemma: affine grassmannian is ind-projective}
We have that $\text{Gr}_{E, \nabla, Z}$ is an ind-projective ind-scheme over $T$, and that the line bundle $\cD(\mathfrak{g})$ is ample on each closed projective subscheme of $\text{Gr}_{E, \nabla, Z}$.
\end{lemma}
\begin{proof}
There is a forgetful morphism $\text{Gr}_{E, \nabla, Z} \to \text{Gr}_{\cE, Z}$ given by $(\cF, \psi, F, \widetilde{\nabla}) \mapsto (\cF, \psi)$. There is an analogous line bundle $\cD(\mathfrak{g})$ on $\text{Gr}_{\cE, Z}$ such that $\cD(\mathfrak{g})$ on $\text{Gr}_{E, \nabla, Z}$ is the pullback under the morphism above. By the proof of \cite[Lem. 3.4]{gauged_theta_stratifications}, $\text{Gr}_{\cE, Z}$ is an ind-projective ind-scheme and the line bundle $\cD(\mathfrak{g})$ is $T$-ample on each closed projective subscheme of $\text{Gr}_{\cE, Z}$ (notice that in \cite{gauged_theta_stratifications} the determinant line bundle involves a square-root $\sqrt{\Omega^1_{C/S}}$ which has no effect on the ampleness by the proof of \cite[Lem. 3.4]{gauged_theta_stratifications}). We shall conclude by showing that the forgetful morphism $\text{Gr}_{E, \nabla, Z} \to \text{Gr}_{\cE, Z}$ is a closed immersion.

We define some intermediate functors. Set $\text{Gr}_{E, Z}$ to be the functor parametrizing $(\cF, \psi, F)$ as in Definition \ref{defn: affine grassmannian for log connections}, but without the data of $\widetilde{\nabla}$. 
On the other hand, define $\text{Gr}_{E, Z, \rho_*(\nabla)}$ to be the functor parametrizing tuples $(\cF, \psi, F, \widetilde{\rho_*(\nabla)})$, where $(\cF, \psi, F)$ is a point in $\text{Gr}_{E, Z}$ and $\widetilde{\rho_*(\nabla)}$ is a meromorphic $t$-connection on $(\cF, \cH)|_{C_X}$ extending $\rho_*(\nabla)_{\cE}|_{C_{X} \setminus Z_{X}}$ under the 
identification $(\psi,\mathrm{id}_{\mathcal{H}}): (\cF, \cH)|_{C_{X} \setminus Z_{X}} \xrightarrow{\sim} (\cE, \cH)|_{C_{X} \setminus Z_{X}}$. The morphism $\text{Gr}_{E, \nabla, Z} \to \text{Gr}_{\cE, Z}$ factors as a composition
\begin{gather*}
    \text{Gr}_{E, \nabla, Z} \to \text{Gr}_{E, D, \rho_*(\nabla)} \to \text{Gr}_{E,Z} \to \text{Gr}_{\cE, Z}, \; \; \; (\cF, \psi, F, \widetilde{\nabla}) \mapsto (\cF, \psi,F, \rho_*(\widetilde{\nabla})) \mapsto (\cF, \psi, F) \mapsto (\cF, \psi) 
\end{gather*} 
We show that all of these morphisms are closed immersions. 

By definition, the left-most morphism fits into a Cartesian diagram
\[
\begin{tikzcd}
    \text{Gr}_{E, \nabla,Z} \ar[r] \ar[d]& \text{Gr}_{E, Z, \rho_*(\nabla)} \ar[d]\\
    \hodge_{\overline{G}} \ar[r] & \hodge_{\GL(\mathfrak{g})\times G_{ab}} \times_{\Bun_{\GL(\mathfrak{g}) \times G_{ab}}(C)} \Bun_{\overline{G}}(C) 
\end{tikzcd}
\]
where the right vertical arrow is given by 
\[(\cF, \psi, F, \widetilde{\rho_*(\nabla)}) \mapsto (((\cH, \cF)|_{C_X}, \widetilde{\rho_*(\nabla)}), F).\]
It follows from the proof of Lemma \ref{lemma: change of group under closed immersion is affine} that the bottom horizontal morphism is a closed immersion, and therefore $\text{Gr}_{E, Z, \nabla} \to \text{Gr}_{E, Z, \rho_*(\nabla)}$ is a closed immersion. 

For the morphism $\text{Gr}_{E, Z, \rho_*(\nabla)} \to \text{Gr}_{E, Z}$, choose a $T$-scheme $X$ and an $X$-point $(\cF, \psi, F) \in \text{Gr}_{E, Z}(X)$. 
Consider the affine bundle 
$A^{D}_{t-Conn, (\cF, \cH|_{C_{X}})}/ C_{X}$ 
of meromorphic $t$-connections on the $G_{ab} \times \GL(\mathfrak{g})$-bundle $(\cF, \cH|_{C_{X}})$. 
The meromorphic $t$-connection $\rho_*(\nabla)|_{C_{X} \setminus Z_{X}}$ yields a section $s:C_{X} \setminus Z_{X} \to A^{D}_{t-Conn, (\cH|_{C_{X}} \times \cF)}$. 
The fiber product $\text{Gr}_{E, Z, \rho_*(\nabla)} \times_{\text{Gr}_{E, Z}} X$ is the subfunctor of $X$ parametrizing morphisms $R \to X$ such that the base change $s_{R}: C_{R} \setminus Z_{R} \to A^{D}_{t-Conn, ( \cH|_{C_{R}} \times \cF|_{C_{R}})}$ extends to a section $C_{R} \to A^{D}_{t-Conn, ( \cH|_{C_{R}} \times \cF|_{C_{R}})}$. Since $A^{D}_{t-Conn, ( \cH|_{C_{R}} \times \cF|_{C_{R}})} \cong A^{D}_{t-Conn, (\cH|_{C_{X}} \times \cF)} \times_{X} R$, it follows from \cite[Lem. 4.12]{rho-sheaves-paper} that $\text{Gr}_{E, Z, \rho_*(\nabla)} \times_{\text{Gr}_{E, Z}} X \to X$ is represented by a closed subscheme of $X$.

Finally, applying \cite[Lem. 3.5]{gauged_theta_stratifications} to the affine scheme $A:= \GL(\mathfrak{g})/\overline{G}$, we have that $\text{Gr}_{E,Z} \to \text{Gr}_{\cE, Z}$ is also a closed immersion.
\end{proof}

\begin{lemma} \label{lemma: monotonicity g overline}
$\overline{\mu}$ is strictly $\Theta$-monotone on $\hodge_{\overline{G}}$.
\end{lemma}
\begin{proof}
This follows from the same ``infinite dimensional GIT" argument as in \cite[\S3]{gauged_theta_stratifications} by using the ind-projective ind-scheme $\text{Gr}_{\cH \times \cE, Z}$ and \Cref{lemma: affine grassmannian is ind-projective}.
\end{proof}

\begin{proof}[Proof of Proposition \ref{prop: monotonicity general}]
The numerical invariant $\mu$ on $\hodge_{G}$ is the pullback of the numerical invariant $\overline{\mu}$ under morphism $\hodge_{G} \to \hodge_{\overline{G}}$. By the assumption (Z) on the kernel $Z_{\cD(G)}$ and Lemma \ref{lemma: change of group for etale isogenies}, $\hodge_{G} \to \hodge_{\overline{G}}$ is quasifinite and proper. Hence, we are in a similar setup as in the proof of \cite[Prop. 3.6]{gauged_theta_stratifications}. Since $\overline{\mu}$ is monotone (Lemma \ref{lemma: monotonicity g overline}), we conclude monotonicity of $\mu$ in exactly the same way as in the proof of \cite[Prop. 3.6]{gauged_theta_stratifications}.
\end{proof}

\begin{prop} \label{prop: hn boundedness}
The numerical invariant $\mu$ satisfies HN-boundedness (as in \Cref{defn: HN boundedness}).
\end{prop}
\begin{proof}
This follows similarly as in the arguments of \cite[\S4.1]{gauged_theta_stratifications}. For the benefit of the reader, we elaborate on how to apply the proof in this setup. We will need some notation, see \cite{gauged_theta_stratifications} for more details on the following. Fix a Borel subgroup $B \supset T$ containing the maximal split torus $T$. We denote by $\Phi \subset X^*(T)$ and $\Phi^{+} \subset X^*(T)$ the subsets of roots and positive roots respectively. Let $\Delta = \{\alpha_i\}_{i \in I}$ denote the corresponding simple roots. We denote by $T'$ the intersection of $T$ with the derived subgroup $\cD(G) \subset G$. Let $P \supset B_k$ be a parabolic subgroup of $G_k$ containing the $B_k$. To such parabolic we can associate a subset $I_{P} \subset I$ consisting of those simple roots $\alpha_i$ such that the root subgroup $U_{-\alpha_i} \subset G_k$ is not contained in $P$. In fact, the subset $I_P \subset I$ uniquely determines $P$. Recall that a real cocharacter $\lambda \in X^*(T)_{\mathbb{R}}$ is called $P$-dominant if we have $\langle \lambda, \alpha_i \rangle >0$ for $i \in I_P$, and $\langle \lambda, \alpha_i \rangle =0$ for $i \notin I_P$.

For any geometric point $x = (E, \nabla) \in \hodge_{G}(k)$, we denote by $|\DF(\hodge_{G}, x)|$ the geometric realization of the degeneration fan of $x$, as defined in \cite[Defn. 3.2.2]{halpernleistner2018structure}. 
Similarly, we use the notation $|\DF(\Bun_{G}(C), E)|$ for the degeneration fan of $E$. 
A reduction of structure group $E_{P}$ to a parabolic $P \supset B_k$ yields a cone $\sigma_{E_{P}} \subset |\DF(\Bun_{G}(C), E)|$ consisting of all weighted parabolic reductions $(E_{P}, \lambda)$ with $\lambda$ a positive real combination of $P$-dominant cocharacters. The degeneration fan $|\DF(\Bun_{G}(C), E)|$ is a union of the cones $\sigma_{E_{P}}$, where $E_{P}$ runs over all parabolic reductions. The assignment $(E_{P}, \lambda) \mapsto \lambda$ defines an identification of $\sigma_{P}$ with the cone spanned by $P$-dominant cocharacters inside $X_*(T)_{\mathbb{R}}$. We denote by $\overline{\sigma}_{E_{P}}$ the closure of $\sigma_{E_{P}}$ inside $X_*(T)_{\mathbb{R}}$.
A point $\lambda \in \overline{\sigma}_{E_{P}} \subset X_*(T)_{\mathbb{R}}$ can be viewed as an element $(E_{P'}, \lambda) \in |\DF(\Bun_{G}(C), E)|$, where $E_{P'}$ is the extension of structure group of $E_{P}$ to the unique parabolic $P' \supset P$ such that $\lambda$ is $P'$-dominant. Therefore, we can view $\overline{\sigma}_{E_{P}} \subset |\DF(\Bun_{G}(C), E)|$. We view $|\DF(\hodge_{G}, x)| \hookrightarrow |\DF(\Bun_{G}(C), E)|$, and consider the intersection $\sigma_{E_{P},x}:= |\DF(\hodge_{G}, x)| \cap \sigma_{E_{P}}$. 
By \Cref{prop: filtrations of hodge}, a parabolic reduction $(E_{P}, \lambda)$ belongs to $|\DF(\hodge_{G}, x)|$ if and only if $E_{P}$ is compatible with $\nabla$. This is independent of $\lambda$, and so we have either $\sigma_{E_{P}, x} = \sigma_{E_{P}}$ or $\sigma_{E_{P},x} = \emptyset$. By \Cref{lemma: extension of admissible parabolic reduction is admissible}, if $\sigma_{E_{P}} \subset |\DF(\hodge_{G}, x)|$, then we have $\overline{\sigma}_{E_{P}} \subset |\DF(\hodge_{G}, x)|$. Therefore $|\DF(\hodge_{G}, x)|$ is a union of closures of cones $\overline{\sigma}_{E_{P}}$ for parabolic reductions $E_P$ compatible with $\nabla$. 
By the computation of the weight of $\cD(\mathfrak{g})$ in \cite[1.F.c]{heinloth-hilbertmumford}, the numerical invariant $\mu = - \wt(\cD(\mathfrak{g}))/\sqrt{b}$ agrees on each $\overline{\sigma}_{E_{P}}$ with the numerical invariant defined in \cite[\S4.1]{gauged_theta_stratifications} for the adjoint representation $V= \mathfrak{g}$ (notice that in \cite{gauged_theta_stratifications} the determinant line bundle involves a square-root $\sqrt{\Omega^1_{C/S}}$ which has no effect on the weight computation for the adjoint representation). Since $\hodge_{G} \to \Bun_{G}(C)$ is quasi-compact, in order to prove the sought-after boundedness of a subset of filtrations in $\hodge_{G}$, it is enough to show boundedness of the underlying parabolic reductions $E_{P}$ in $\Bun_{G}(C)$. We can now conclude HN boundedness by the same optimization argument in $\overline{\sigma}_{E_{P}} \subset |\DF(\hodge_G,x)|$ as in \cite[\S4.1]{gauged_theta_stratifications}. 
\end{proof}
\end{subsection}
\begin{subsection}{\texorpdfstring{$\Theta$}{Theta}-stratification and semistable reduction}\label{sec: semistable reduction}

\begin{prop} \label{prop: weak theta stratification}
Suppose that $G$ satisfies the central property (Z). Then the numerical invariant $\mu$ defines a well-ordered weak $\Theta$-stratification on $\hodge_{G}$.
\end{prop}
\begin{proof}
Proposition \ref{prop: monotonicity general} and Proposition \ref{prop: hn boundedness} verify the hypotheses of \Cref{thm: theta stability paper theorem}, which implies the existence of the weak $\Theta$-stratification. The open strata are indexed by the value of the numerical invariant $\mu(f_{can})$ at a Harder-Narasimhan maximizing filtration $f_{can}$, and hence they are indexed by nonnegative real numbers. 
For the well-ordering property, we only need to show that such maximal values $\mu(f_{can})$ lie inside a discrete subset of $\mathbb{R}$, or equivalently show that the squares $\mu(f_{can})^2$ lie in a discrete subset of $\mathbb{R}$. 
Using the notation of \cite[\S4.1]{gauged_theta_stratifications}, Riemann-Roch implies that the numerical invariant for a weighted parabolic reduction $(E_{P}, \lambda)$ is given by $\mu = -\wt(\cD(\mathfrak{g}))/\sqrt{b}= (\lambda, d_{E_{P}})_{\mathfrak{g}} / \| \lambda\|_b$, where $(-,-)_{\mathfrak{g}}$ is the trace form associated to the adjoint representation $\mathfrak{g}$ and $d_{E_{P}}$ is the degree of the $P$-bundle $E_{P}$ as in \Cref{defn: deg of par} below. 
By the explicit optimization arguments in \cite[\S4.1]{gauged_theta_stratifications}, 
we have $\lambda=\phi(d_{E_{P}})$ up to scaling, 
where $\phi: X_*(T)_{\mathbb{R}} \to X_*(T)_{\mathbb{R}}$ is the ``adjoint" morphism determined by the equation $(\phi(v), w)_{b} = (v, w)_{\mathfrak{g}}$ for any $v,w \in X_*(T)_{\mathbb{R}}$. 
This shows that $\mu(f_{can})^2 = \|\phi(d_{E_{P}})\|_{b}^2$. 
Notice that $\phi$ is actually a matrix with rational entries $\phi: X_*(T)_{\mathbb{Q}} \to X_*(T)_{\mathbb{Q}}$. Also, there is a lattice $L \subset X_*(T)_{\mathbb{Q}}$ containing every element of the form $d_{E_{P}}$, because all such $d_{E_{P}}$ pair integrally with every root and every character of $G$. Hence, we have $\|\phi(d_{E_{P}})\|_b^2 \subset \|\phi(L)\|_{b}^2$. Now $\phi(L)$ is a lattice in $X_*(T)_{\mathbb{Q}}$ and so we have $\|\Phi(L)\|_b^2 \subset \frac{1}{N}{ \mathbb{Z}}$ for $N$ depending only on $G$ and $b$, thus concluding the desired discreteness.
\end{proof}

By \Cref{prop: weak theta stratification}, for any algebraically closed field $k$ over $S$ and any $k$-point $(E, \nabla) \in \hodge_G(k)$, there is a canonical maximally destabilizing filtration $(E_P, \lambda)$, where the $\nabla$-compatible parabolic reduction $E_P$ is unique, and the $P$-dominant cocharacter $\lambda$ is unique up to scaling. We call $E_P$ the canonical parabolic reduction.

\begin{defn}[Degree of a parabolic bundle]
\label{defn: deg of par}
Let $k$ be an algebraically closed field over $S$. Let $P \subset G_k$ be a parabolic subgroup, and fix a $P$-bundle $E_P$ on $C_k$. Let $L$ denote the Levi quotient of $P$. We define $d_{E_P}: X^*(P) \to \mathbb{Z}$ by sending any character $\chi: P_{\lambda} \to \mathbb{G}_m$ to $d_{E_P}(\chi) := \text{deg}(E_P(\chi))$. Using the natural identification $X^*(P)_{\mathbb{Q}} \cong X^*(Z_{L}^{\circ})_{\mathbb{Q}}$, we view $d_{E_{P}}$ a rational cocharacter in $X_*(Z_{L}^{\circ})_{\mathbb{Q}}$. 
\end{defn}

\begin{prop} \label{prop: description of canonical parabolic reduction}
    Let $k$ be an algebraically closed field over $S$, and let $(E, \nabla)$ be point in $\hodge_G(k)$. Let $E_P$ denote the canonical $\nabla$-compatible parabolic reduction of $E$. Then the following conditions are satisfied:
    \begin{enumerate}[(1)]
        \item The degree $d_{E_P} \in X_*(Z_{L}^{\circ})_{\mathbb{Q}}$ is $P$-dominant.
        \item Let $L$ denote the Levi quotient of $P$. Then the associated $L$-bundle with meromorphic $t$-connection $(E_L, \nabla_L)$ under the quotient morphism $P \twoheadrightarrow L$ is semistable.
    \end{enumerate}
\end{prop}
\begin{proof}
    This follows from a similar argument as in the proofs of (i) and (ii) in the proof of \cite[Prop. 6.11]{rho-sheaves-paper}. The proof of (i) (in our case (1)) goes through in the same way.
    For the proof of (ii) (in our case (2)) one just needs to observe that the lifted parabolic reduction under $q: P \to \overline{P}$ is $\nabla$-compatible.
\end{proof}

\begin{prop}[Behrend's conjecture] \label{prop: behrends conjecture}
Suppose that $G$ satisfies the low height property (LH) (see \Cref{defn: low height property}). Let $k$ be an algebraically closed field, and let $(E, \nabla) \in \hodge_{G}(k)$ be a meromorphic $t$-connection. Let $(E_{P}, \lambda)$ be the canonical HN weighted parabolic HN reduction provided by Proposition \ref{prop: weak theta stratification}. Then the set of first order deformations $\text{Def}_{\nabla}(E_{P})$ (Definition \ref{defn: deformations of parabolic reductions}) consists of only one element.
\end{prop}
\begin{proof}
We denote by $t$ the image of $(E,\nabla)$ under $\hodge_{G} \to \mathbb{A}^1_S$. Let $\mathfrak{p}$ denote the Lie algebra of the parabolic subgroup $P \subset G_k$. The adjoint representation induces an action of $P$ on $\mathfrak{g}_k/\mathfrak{p}$ via a homomorphism $\rho: P \to \GL(\mathfrak{g}_k/\mathfrak{p})$. 
The associated vector bundle $\rho_*(E_P):= \Ad(E)/\Ad(E_P)$ is equipped with the meromorphic $t$-connection induced by $\Ad(\nabla)$. 
By \Cref{prop: parabolic reductions deformations}, deformations of the compatible parabolic reduction $E_P$ are in bijection with sections $s: \cO_{C_k} \to \Ad(E)/\Ad(E_P)$ such that $\Ad(\nabla)(s) =0$. 
Such flat sections correspond to homomorphisms of vector bundles with meromorphic $t$-connection $s: \cO_{C_k}\to \Ad(E)/\Ad(E_P)$, where $\cO_{C_k}$ is equipped with the standard meromorphic $t$-connection $t \cdot  d$ induced by the exterior differential $d$. To conclude the proposition, we shall show that any such $s$ is $0$. As in the proofs of \cite[Thm. 4.1, Prop. 6.10]{biswas-holla-hnreduction}, there is a filtration of $\mathfrak{g}_k/\mathfrak{p}$ by $P$-subrepresentations
\[ 0 \subset V_1 \subset V_2 \subset V_3, \subset \ldots \subset V_{n-1} \subset V_n = \mathfrak{g}_k/\mathfrak{p}\]
such that for each associated graded piece $V_i/V_{i+1}$ the unipotent radical $U\subset P$ acts trivially. Taking the associated bundle construction for the $P$-parabolic reduction $(E_P, \nabla_P)$, we obtain a filtration of $\Ad(E)/\Ad(E_P)$ by subbundles with meromorphic $t$-connections
\[ 0 \subset E_P(V_1) \subset E_P(V_2) \subset E_P(V_3) \subset \ldots \subset E_P(V_{n-1}) \subset E_P(V_n) = \Ad(E)/\Ad(E_P)\]
For any index $i$, the associated graded bundle $E_P(V_i/V_{i-1})$ is equipped with an induced meromorphic $t$-connection $\nabla_i$. It suffices to show that for each $i$ there are no nonzero homomorphisms $(\cO_{C_k}, t \cdot d)\to (E_P(V_i/V_{i-1}), \nabla_i)$ compatible with the connections. Since the unipotent radical acts trivially, the representation $P \to \GL(V_i/V_{i-1})$ factors through the Levi quotient $L$. In particular, we can think of $E_P(V_i/V_{i-1})= E_L(V_i/V_{i-1})$ as bundles associated to the corresponding Levi bundle $(E_L, \nabla_L)$ equipped with the $t$-meromorphic connection $\nabla_L$. The height of $V_i/V_{i-1}$ viewed as a $L$-representation is smaller than the characteristic of $k$ whenever $\text{char}(k) \neq 0$, because it is a subquotient of the low height representation $\mathfrak{g}$. Therefore by \cite[Lect.4, Thm.6]{serre20031998}, $V_i/V_{i-1}$ is a semisimple representation of $L$, and after breaking $V_i/V_{i-1}$ into a direct sum of simple representations we may assume that the associated graded pieces $V_i/V_{i-1}$ are simple. By Schur's lemma, this means that $L \to \GL(V_i/V_{i-1})$ maps the center of $L$ into the center of $\GL(V_i/V_{i-1})$. Since $(E_L, \nabla_L)$ is semistable (\Cref{prop: description of canonical parabolic reduction}), it follows from \Cref{prop: semistability of low height representatons} that $(E_L(V_i/V_{i-1}), \nabla_i)$ is semistable. On the other hand, we have that the degree coweight $d_{E_P}$ is $P$-dominant (\Cref{prop: description of canonical parabolic reduction}), which means that it acts with negative weights on $\mathfrak{g}/\mathfrak{p}$. It follows that the degree of the vector bundle $E_L(V_i/V_{i-1}) = E_P(V_i, V_{i-1})$ is strictly negative. By a standard argument (cf. \cite[Lem. 1.3.3]{huybrechts.lehn}), there are no nonzero homomorphisms from the degree 0 semistable $(\cO_{C_k}, t\cdot d)$ into the negative degree semistable $(E_L(V_i/V_{i-1}), \nabla_i)$, as desired.
\end{proof}

\begin{remark}
    The low height property (LH) in our assumptions in \Cref{prop: behrends conjecture} amounts to a lower bound on the characterisitics of the ground fields depending on the type of the reductive group. Heinloth \cite{heinloth-behrend-conjecture} obtained better bounds on the characteristic of the ground field in order for Behrend's conjecture to hold for the Harder-Narasimhan reductions of $G$-bundles (without the data of a $t$-connection). We note that it is possible to modify the argument found in \cite[\S2]{heinloth-behrend-conjecture} for classical groups and use the analysis of the heights of restrictions of the adjoint representation to Levi subgroups performed in \cite[\S3]{heinloth-behrend-conjecture} in order to obtain the same bounds as in \cite[Thm. 1]{heinloth-behrend-conjecture} for the moduli of $G$-bundles with $t$-connections. We have preferred to include a more uniform proof with worse bounds in our \Cref{prop: behrends conjecture}.
\end{remark}

As a consequence of the arguments in the proof of \Cref{prop: behrends conjecture}, we obtain the converse of \Cref{prop: description of canonical parabolic reduction}. 
\begin{prop} \label{prop: description of canonical parabolic reduction strong}
    Suppose that $G$ satisfies the low height property (LH). Let $k$ be an algebraically closed field over $S$, and let $(E, \nabla)$ be point in $\hodge_G(k)$. A $\nabla$-compatible reduction $E_P$ to a parabolic subgroup $P \subset G_k$ is the canonical parabolic reduction if and only if both of the following conditions are satisfied
    \begin{enumerate}[(1)]
        \item The degree $d_{E_P} \in X_*(Z_{L}^{\circ})_{\mathbb{Q}}$ is $P$-dominant.
        \item The associated $L$-bundle with meromorphic $t$-connection $(E_L, \nabla_L)$ under the quotient morphism $P \twoheadrightarrow L$ is semistable.
    \end{enumerate}
\end{prop}
\begin{proof}
In view of \Cref{prop: description of canonical parabolic reduction}, it suffices to show that there is a unique compatible parabolic reduction $E_P$ satisfying (1) and (2). For this, we can run the argument in \cite[Thm. 4.1]{biswas-holla-hnreduction} using the filtration of the vector bundle with meromorphic $t$-connection $E_P(\mathfrak{g}_k/\mathfrak{p})$ that we obtained in the proof of \Cref{prop: behrends conjecture}.
\end{proof}

\begin{thm} \label{prop: theta stratification}
Suppose that $G$ satisfies the low height property (LH). Then $\mu$ defines a well-ordered $\Theta$-stratification on $\hodge_{G}$.
\end{thm}
\begin{proof}
By Proposition \ref{prop: weak theta stratification}, $\mu$ defines a well-ordered weak $\Theta$-stratification. Showing that it is a $\Theta$-stratification amounts to proving that for any unstable geometric point $(E,\nabla) \in \hodge_{G}$, the canonical Harder-Narasimhan filtration is rigid \cite[Lem.2.1.7(1)]{halpernleistner2018structure}. In our case, the canonical filtration is given by the HN weighted parabolic reduction $E_{P}$ compatible with $\nabla$. Relative deformations of the $\nabla$-compatible parabolic reduction $E_{P}$ with respect to the fixed meromorphic connection $(E,\nabla)$ are given by the set $\text{Def}_{\nabla}(E_P)$ (Definition \ref{defn: deformations of parabolic reductions}), which contains only the trivial deformation by Proposition \ref{prop: behrends conjecture}.
\end{proof}

\begin{prop}[Semistable reduction] \label{prop: semistable reduction}
Suppose that $G$ satisfies the low height property (LH). Let $R$ be a discrete valuation ring. Choose an $R$-point $f: \Spec(R) \to \hodge_{G}$ such that the image $f(\eta)$ of the generic point $\eta \in \Spec(R)$ belongs to the semistable locus $(\hodge_{G})^{ss}$. Then there exists an extension of DVR $R\subset R'$ with finite extension of quotient fields $K(R)\subset K(R')$, and an $R'$-point $\widetilde{f}: \Spec(R') \to (\hodge_{G})^{ss}$ and an isomorphism between the restrictions of $f$ and $\widetilde{f}$ to $\mathrm{Spec}(K'(R))$.
\end{prop}
\begin{proof}
This follows from Proposition \ref{prop: theta stratification} and \cite[Thm. 6.5]{alper2019existence}.
\end{proof}
\end{subsection}

\begin{subsection}{Smoothness of the semistable locus}\label{subsec: smooth}

\begin{thm} \label{thm: main theorem smoothness}
Suppose that $G$ satisfies the low height property (LH), and that the $S$-fibers of the divisor $D$ are nonempty. Then the morphism $(\hodge_{G})^{ss} \to \mathbb{A}^1_{S}$ is smooth.
\end{thm}
\begin{proof}
We use the lifting criterion for smoothness \cite[\href{https://stacks.math.columbia.edu/tag/0DP0}{Tag 0DP0}]{stacks-project} for the finite type morphism $(\hodge_{G})^{ss} \to \mathbb{A}^1_B$. By \cite[\href{https://stacks.math.columbia.edu/tag/02HT}{Tag 02HT}]{stacks-project} it suffices to show the existence of lifting for square-zero thickenings of local Artin algebras. We are reduced to finding a lift as in \Cref{context: lifting problem} with $\hodge_G$ replaced by $(\hodge_G)^{ss}.$

Using \Cref{prop: obstruction class}(a), \Cref{prop: g_m action and limits} and \Cref{prop: semistable reduction}, the same degeneration argument as in the proof of \cite[Prop. 5.6]{decataldo-herrero-nahpoles} reduces us to the case when $A=k$ and $p$ lies over $0\in \mathbb{A}^1_k$, and so it is represented by a meromorphic Higgs bundle $(E, \varphi) \in (\Higgs_{G})^{ss}$. Now we can conclude as in the proof of smoothness in \cite{decataldo-herrero-nahpoles,decomposition-higgs} to show $\mathcal{Q}_p= \mathbb{H}^2(\mathcal{C}^{\bullet}_p)=0$. Indeed, by Serre duality we have that $\mathbb{H}^2(\mathcal{C}^{\bullet}_p)$ is dual to $\mathbb{H}^{-1}((\mathcal{C}^{\bullet}_p)^{\vee} \otimes \Omega^1_{C_k/k})$. The hypercohomology spectral sequence identifies $\mathbb{H}^{-1}((\mathcal{C}^{\bullet}_p)^{\vee} \otimes \Omega^1_{C_k/k})$ with the kernel $K$ of
\[ H^0(\Ad(\varphi) \otimes id_{\cO(-D)}): H^0(\Ad(E) (-D)) \to H^0(\Ad(E) \otimes \Omega^1_{C_k/k})\]
Elements of $K$ can be interpreted in terms of homomorphisms of Higgs bundles. Consider $\cO_{C_k}$ equipped with the zero Higgs field $0$. Then, elements of $K$ are equivalent to homomorphisms $\cO_{C_k} \to \Ad(E)(-D)$ compatible with the Higgs fields $0$ and $\Ad(\varphi)\otimes id_{\cO(-D)}$ respectively. By \Cref{thm: stability under change of group}, since $G$ satisfies the low height property (LH) we have that $(\Ad(E)(-D), \varphi\otimes \text{id}_{\cO(-D)})$ is a semistable Higgs bundle. Since $D$ has nonempty $S$-fibers, the vector bundle $\Ad(E)(-D)$ has negative slope, and therefore there a no nonzero homomorphisms $(\cO_{C_k}, 0) \to (\Ad(E)(-D), \varphi \otimes \text{id}_{\cO(-D)})$ from the semistable Higgs bundle $(\cO_{C_k}, 0)$ of slope $0$. We conclude that $0 = K^{\vee} = \mathbb{H}^2(\mathcal{C}^{\bullet}_p) =\mathcal{Q}_p$, which implies the existence of the desired lift by \Cref{prop: obstruction class}(a).

\end{proof}
\end{subsection}

\end{section}

\begin{section}{The Hodge-Hitchin morphism}\label{section: hh morphism}

In this section, we establish the so-called Hodge-Hitchin morphism in characteristic $p$ which is larger than the Coxeter number of $G$, and show the properness of the induced morphism on the moduli space.
These results are fundamental for the development of the ramified versions of the Non Abelian Hodge Theory in prime characteristic \cite{de-zhang-log, li2024tame}.

In this section, we assume that $S$ is of characteristic $p>2$, and $G$ is a split reductive group scheme over $S.$ Under this assumption, $G$ is root-smooth as in \cite[\S4.1.1]{bouthier2019torsors}.
    We also assume that for every geometric point $\bar{s}$ of $S,$ $p$ is not a torsion prime for the root datum for $G_{\bar{s}}$ in the sense of \cite[\S 4.1.12]{bouthier2019torsors}.
    This assumption is satisfied if $p$ does not divide the order of the Weyl group of $G_{\bar{s}}$ for any geometric point $\bar{s}$ of $S.$

\begin{subsection}{Hitchin morphism and \texorpdfstring{$p$}{p}-curvature} \label{subsection: Hitchin morphism and p-curv}

    Let $\mathfrak{c}=\mathfrak{g}\git G := \Spec_S(\Sym_{\cO_S}(\mathfrak{g})^G)$ be the GIT quotient by the adjoint action of $G.$
     By \cite[Thm. 4.1.10]{bouthier2019torsors} and \cite[Thm. 3]{demazure1973invariants}, we have the Chevalley isomorphism of $S$-schemes:
    \begin{equation}
    \label{eqn: proto hitchin base}
        \mathfrak{c}=\Fg\git G\cong \mathfrak{t} \git W \cong \mathbb{A}_S^N.
    \end{equation}
Note that the last isomorphism is not canonical. Whenever we use it in the proofs, we fix one particular isomorphism.
    The isomorphisms in (\ref{eqn: proto hitchin base}) are $\bG_m$-equivariant for a suitably weighted $\bG_m$-action on $\mathbb{A}_S^N.$
    Consider the twist 
    $\Fc_D:= \Fc\times^{\bG_m}\Omega_{C/S}^1(D),$
    which is an affine scheme over $C.$
    The Hitchin base $A(C,\Omega_{C/S}^1(D))$ is defined to be the functor of $C$-sections of $\Fc_D.$ 
    The stack of Higgs bundles $\Higgs_{G}$ is isomorphic to the stack of $C$-sections of $(\Fg/G)_D:= (\mathfrak{g} \otimes \Omega^1_{C/S}(D))/G.$
The Hitchin morphism $h: \Higgs_{G}\to A(C,\Omega^1_{C/S}(D))$
is induced by the natural morphism $\Fg/G\to \Fg\git G.$

Our next goal is to define $p$-curvature. We will need the following.

\begin{prop}[Deligne's identity {\cite[Prop. 5.3]{katz1970nilpotent}}] 
Let $A$ be an associative algebra of characteristic $p.$ Let $f,\partial\in A.$
If the elements $(ad(\partial))^n(f)\in A$ mutually commute for $n\ge 0,$ then we have the identity
        \begin{equation}
            \label{eqn: deligne id}
            (f\partial)^p=f^p\partial^p+f(ad(\partial)^{p-1})(f^{p-1})\partial.
        \end{equation}
    \end{prop}
    
    Since we are in characteristic $p,$ the $p$-th iteration of any derivation $\partial$ in $T_{C_Y/Y}$ is again a derivation, which we denote by $\partial^{[p]}.$

    \begin{lemma}
    \label{deligne-id}
        Given a derivation $\partial$ which is a local section of $T_{C_Y/Y}(-D),$ the derivation $\partial^{[p]}$ is also a local section of $T_{C_Y/Y}(-D).$
    \end{lemma}
    \begin{proof}
        We are immediately reduced to the case where $Y$ is an affine scheme with function ring $R$, and we can replace $C_Y$ with $\mathbb{A}^1_R$ upon choosing a local coordinate $x$ around a component of the Cartier divisor $D.$
        In this coordinate, $D$ is given by $x^n$ for some $n>0.$
        We are reduced to showing that for any derivation $\partial,$ the element $(x^n\partial)^{[p]}$ lies inside the right ideal generated by $x^n$ in the Weyl algebra $R\langle x,\partial_x\rangle.$ This follows from Deligne's identity \eqref{eqn: deligne id}, which gives us $(x^n\partial)^{[p]}=x^{pn}\partial^{[p]}+x^n\Big( \partial^{p-1}\cdot(x^{pn-p})\Big)\partial,$
        where $\partial^{p-1}\cdot(x^{pn-p})$ means we apply the differential operator $\partial^{p-1}$ to $x^{pn-p}.$
    \end{proof}

 Let $Y$ be an $S$-scheme and $(E,\nabla)$ a $Y$-point of $\hodge_{G}.$
    By \Cref{example: induced connection in the total space of G-bundle}, with the evident modifications over a base $S$ that is not a field,
    we see that the data of $(E,\nabla)$ yields a $\cO_{C_Y}$-linear morphism $\nabla: T_{C_Y/Y}(-D)\to End_{\cO_Y}(\cO_E,\cO_E),$
    which satisfies the $t$-Leibniz rule and the multiplicative Leibniz rule, and is compatible with the $G_{C_Y}$-action. \Cref{deligne-id} makes the following definition meaningful.
    \begin{defn}
        The $p$-curvature $\Psi(\nabla)$ of $(E,\nabla)$ is the $\cO_Y$-linear assignment that takes a derivation $\partial$ in $T_{C_Y/Y}(-D)$ to the section of $End_{\cO_Y}(\cO_E,\cO_E)$ defined by the formula
        \begin{equation*}
            \Psi(\nabla)(\partial):= \Big(\nabla(\partial)\Big)^p-t^{p-1}\nabla(\partial^{[p]}).
        \end{equation*}
    \end{defn}

    \begin{lemma}
    \label{lemma: ad-p-linear}
        The morphism $\Psi(\nabla)$
        satisfies the following:
        \begin{enumerate}
            \item[(a)] There is a factorization
            \begin{equation*}
                \Psi(\nabla): T_{C_Y/Y}(-D)\to Ad(E)\hookrightarrow Der_{\cO_{C_Y}}(\cO_E) \hookrightarrow End_{\cO_Y}(\cO_E,\cO_E),
            \end{equation*}
            where we have identified $Ad(E)$ with the $G$-invariant derivations on $\cO_E.$
            \item[(b)] The morphism $\Psi(\nabla)$ is $p$-linear, i.e., given any local function $f$ on $C_Y,$ we have that $\Psi(\nabla)(f\partial)=f^p\Psi(\nabla)(\partial).$
            The morphism $\Psi(\nabla)$ is also additive.
        \end{enumerate}
    \end{lemma}
    \begin{proof}
       Let $\partial$ be a local section of $T_{C_Y/Y}(-D).$
       We first show that $\Psi(\nabla)(\partial)$ is $\cO_{C_Y}$-linear.
       Let $e$ be a local section of $\cO_E$ and $f$ a local function on $C_Y.$ We then have 
       \begin{align*}
          \Psi(\nabla)(\partial)(fe) & = \Big(\nabla(\partial)\Big)^p(fe)-t^{p-1}\nabla(\partial^{[p]})(fe) \\
           &= \Big( \sum_{i=0}^p {p \choose i}t^i\partial^i(f)(\nabla(\partial))^{p-i}(e) \Big)-t^{p-1}\Big(t\partial^p(f) e+f\nabla(\partial^{[p]})(e)\Big)\\
           & = \Big(f(\nabla(\partial))^p(e)+t^p\partial^p(f)e\Big)-\Big( t^p\partial^p(f)e+t^{p-1}f\nabla(\partial^{[p]})(e)\Big)\\
           & = f\Psi(\nabla)(e).
       \end{align*}

       That $\Psi(\nabla)(\partial)$ is a derivation of $\cO_E$ follows from that $\nabla(\partial)$ is a derivation of $\cO_E.$

       We now show that $\Psi(\nabla)(\partial)$ is $G$-invariant.
       Let $a:\cO_E\to \cO_E\otimes_{\cO_{C_Y}} \cO_{G_{C_Y}}$ be the morphism defined by the $G$-action.
       Let $\nabla^{can}$ be the canonical $t$-connection on $\cO_{G_{C_Y}}$ that comes from $td$ on $\cO_{C_Y}.$
       Since $\nabla(\partial)$ is $G$-invariant, we have: 
       $a\Big(\nabla(\partial)(e)\Big)=(\nabla\otimes\nabla^{can})\Big(a(e)\Big)$,
       where the tensor connection is defined in \Cref{example: tensor connection}.
       Since $\Psi(\nabla^{can})=0,$ we have: 
       $a\Big(\Psi(\nabla)(\partial)(e)\Big)=\Psi(\nabla)\Big(a(e)\Big)$.
       We have thus finished the proof of item (a) in the lemma.

       We now show item (b).
       We first show that $\Psi(\nabla)$ is additive.
       Let $o:=\sum_{i=1}^{p-1}s_i(\nabla(\partial),\nabla(\partial')),$ where $s_i$ are the universal Lie polynomials.
       By \cite[(5.2.7-8)]{katz1970nilpotent}, we have: 
       \[\Psi(\nabla)(\partial+\partial')=\Big((\nabla(\partial)^p+\nabla(\partial')^p+o\Big)-t^{p-1}\Big(\nabla(\partial^{[p]})+\nabla((\partial')^{[p]})+o\Big)=\Psi(\nabla)(\partial)+\Psi(\nabla)(\partial').\]
       Finally, we show that $\Psi(\nabla)$ is $p$-linear.
       Using Deligne's identity (\ref{eqn: deligne id}), we have:
       \begin{align*}
           \Psi(\nabla)(f\partial) &= \Big(f\nabla(\partial)\Big)^p-t^{p-1}\nabla\Big( (f\partial)^{[p]}\Big) \\
           &= f^p \nabla(\partial)^p+f(ad(\nabla(\partial))^{p-1})(f^{p-1})\nabla(\partial) -t^{p-1}\nabla\Big( f^p\partial^{[p]}+f(ad(\partial)^{p-1})(f^{p-1})\partial\Big)\\
           &= f^p\nabla(\partial)^p+ft^{p-1}\partial^{p-1}(f^{p-1})\nabla(\partial)-t^{p-1}\nabla(f^p\partial^{[p]}+f\partial^{p-1}(f^{p-1})\partial)\\
           &=f^p\Psi(\nabla)(\partial).
       \end{align*}
    \end{proof}

\begin{notn}[Frobenii] 
    Let $F_S:S\to S$ be the absolute Frobenius morphism on $S.$
    Given any $S$-scheme $X,$ we denote the Frobenius twist $X':=X\times_{S,F_S} S.$
    We have the following diagram of schemes:
    \begin{equation*}
        \xymatrix{
        X\ar[r]^-{F_{X/S}} \ar[dr]& X' \ar[r]^-{\sigma_X} \ar[d]& X\ar[d]\\
        & S \ar[r]^-{F_S} & S.
        }
    \end{equation*}
    \end{notn}
 
    \begin{lemma}
    \label{lemma: F pullback cl}
     Let $D':= \sigma_C^*D$ be the pullback Cartier divisor on $C'/S.$
Pullback by $F_{C/S}$ defines a closed embedding of total spaces of $S$-vector bundles:
    \begin{equation*}
        F_{C/S}^*: A(C', \Omega_{C'/S}^1(D'))\hookrightarrow A(C,\Omega^1_{C/S}(D)^{\otimes p}).
    \end{equation*}
    \end{lemma}
    
    \begin{proof}
By the Chevalley isomorphism (\ref{eqn: proto hitchin base}), we have an identification of functors $A(C',\Omega_{C'/S}^1(D')) = \prod_{i} \Gamma_{C'/S}(\Omega_{C'/S}^1(D')^{\otimes n_i})$, where $\Gamma_{C'/S}(\Omega_{C'/S}^1(D')^{\otimes n_i})$ is the functor of sections of the bundle $\Omega_{C'/S}^1(D')^{\otimes n_i}$ relative to the morphism $C' \to S$ and $n_i$ are certain nonnegative integers. Similarly we have $A(C,\Omega^1_{C/S}(D)^{\otimes p}) = \prod_i\Gamma_{C/S}(\Omega^1_{C/S}(D)^{\otimes n_i \cdot p})$. By the base change criterion \cite[Rmk. 3.11.2]{illusie2005grothendieck}, the Hitchin base $A(C,\Omega^1_{C/S}(D)^{\otimes p})$ is represented by an $S$-vector bundles whenever we have $H^1(C_s,\Omega_{C_s/s}^1(D_s)^{\otimes N})=0$ for each point $s\in S$ and $N>0$.
This vanishing is satisfied unless 

\noindent (1$)_s$ $g(C_s)=1$ and $D_s=\emptyset;$

\noindent (2$)_s$ $g(C_s)=0$ and $deg(D_s)=2;$ 

\noindent (3$)_s$ $g(C_s)=0$ and $deg(D_s)<2.$

Using the finiteness and flatness of $D/S$ and the connectedness of $S,$ we have that if $(i)_s, i=1,2,3,$ is true for some $s\in S,$ then it is true for all $s\in S.$
In cases (1) and (2), we have $\Omega_{C/S}^1(D)\cong \cO_C$ and the desired base change isomorphism follows from \cite[Ex. 3.11]{kleiman2005picard}.
In case (3), the Hitchin base functor is 0, thus it is an $S$-vector bundle tautologically. The same argument with $C'$ replacing $C$ shows that $A(C',\Omega_{C'/S}^1(D'))$ is also represented by an $S$-vector bundle.

By \cite[\href{https://stacks.math.columbia.edu/tag/0FW2}{Tag 0FW2}]{stacks-project}, the morphism $F_{C/S}$ is flat, in addition to finite and surjective. Hence $ (F_{C/S})_*F_{C/S}^*(\Omega_{C'/S}^1(D')^{\otimes N} ) = (F_{C/S})_*(\Omega^1_{C/S}(D)^{\otimes pN})$ is a vector bundle on $C'$ for all $N$. To finish the proof, it suffices to show that each pullback morphism $\Gamma_{C'/S}(\Omega_{C'/S}^1(D')^{\otimes n_i}) \to \Gamma_{C/S}(\Omega_{C/S}^1(D)^{\otimes p \cdot n_i})= \Gamma_{C'/S}( (F_{C/S})_*(\Omega^1_{C/S}(D)^{\otimes p \cdot n_i}))$ is a closed immersion. This morphism is induced by the unit morphism of sheaves $\Omega_{C'/S}^1(D')^{\otimes n_i} \to (F_{C/S})_*F_{C/S}^*(\Omega_{C'/S}^1(D')^{\otimes n_i})$. It suffices to show that the unit morphism is a subbundle inclusion (i.e. it is injective and the cokernel is a vector bundle). Since both the source and target of the unit are vector bundles, by the slicing criterion for flatness \cite[\href{https://stacks.math.columbia.edu/tag/00ME}{Tag 00ME}]{stacks-project} it suffices to show that the unit is injective after base change to any field valued point of $C'$. But, since $F_{C/S}$ is finite, the formation of the unit commutes with base change on $C'$. And since $F_{C/S}$ is surjective, it follows that the corresponding unit is injective when restricted to any point of $C'$.
    \end{proof}

\subsection{Existence of the Hodge-Hitchin morphism}
\label{subsection: de Rham- Hitchin morphism}

 By \Cref{lemma: ad-p-linear}, $\Psi(\nabla)$ defines an $\cO_{C_Y}$-linear morphism $F_{C_Y/Y}^*\Big( T_{C_Y'/Y}(-D')\Big)\to Ad(E)$.
Therefore, the pair $(E,\Psi(\nabla))$ defines a $(\Omega_{C_Y/Y}^1(D))^{\otimes p}$-twisted Higgs bundle on $C_Y/Y.$


In the case where $S$ is a point, recall the Coxeter number $h(G)$ of $G$ as defined in \cite[\S5.1]{serre2005complete}.
If $G$ is a torus, then $h(G):=1$. Otherwise $h(G)$ is defined as the maximum of the Coxeter numbers of the simple quotients of $G$.
For simple groups, the number $h(G)$ according to the type of $G$ is: $A_n, h=n+1; B_n, C_n: h=2n; D_n: h=2n-2; G_2: h=6; F_4, E_6: h=12; E_7: h=18; E_8: h=30$.

\begin{prop}
    \label{prop: hodge-hitchin}
Assume that for each geometric point $\overline{s}$ of $S$, we have $p>h(G_{\overline{s}})$.  
Then there is the following commutative diagram of $S$-stacks:
    \begin{equation}
    \label{diagram: factorization}
        \xymatrix{
        \hodge_{G} \ar[r]^-{\Psi} \ar[d]_-{\exists ! \; \widetilde{h_{Hod}}} &
        \text{Higgs}(C, (\Omega_{C/S}^1(D))^{\otimes p}) \ar[d]^-{h}\\
        A(C',\Omega_{C'/S}^1(D')) \ar[r]_-{F_{C/S}^*} &
        A(C, (\Omega_{C/S}^1(D))^{\otimes p}).
        }
    \end{equation}
\end{prop}

\begin{remark}\label{rmk: lh implies coxeter}
Since the height of the adjoint representation of $G_{\overline{s}}$ is $2h(G_{\overline{s}})-2$ \cite[(5.2.6)]{serre2005complete}, we see that the low height condition implies the assumption of \Cref{prop: hodge-hitchin}.    
\end{remark}

\begin{defn}
\label{defn: hodge hitchin}
    The Hodge-Hitchin morphism
    \[h_{Hod}:\hodge_G\to A(C',\Omega_{C'/S'}^1(D'))\times_S\mathbb{A}^1_S\]
    is defined to be the product of  $\widetilde{h_{Hod}}$ as in (\ref{diagram: factorization}) and the structural morphism $\hodge_G\to \mathbb{A}^1_S.$
\end{defn}

The goal of the rest of this subsection is to prove \Cref{prop: hodge-hitchin}. 
We use a variation of Steinberg's proof of the Chevalley Isomorphism Theorem as in \cite[Thm. 23.1]{humphreys2012introduction}.

Let us start with a reminder on some invariant theory in \cite{bourbaki2002lie}.
\begin{remark}
\label{rmk: invariants}
    Let $k$ be a field of characteristic $p$. 
    Let $V$ be a finite dimensional $k$-vector space with an action of a finite group $W$ generated by pseudo-reflections.
    The ring of invariants $Sym(V)^W$ is a graded $k$-algebra of finite type \cite[Ch. V, \S5.2]{bourbaki2002lie}.
    Choose a minimal set of homogeneous generators $\alpha_1,...,\alpha_m$ of the ideal $R_+$ of $R$.
    Let $n_i:=\deg(\alpha_i)$.
    \cite[Ch. V, \S5.3]{bourbaki2002lie} entails that if $p\nmid n_i$ and if $p\nmid |W|$, then $m=\mathrm{dim}(V)$, $\prod_{i=1}^m n_i=|W|$, and $Sym(V)^W$ is isomorphic to the polynomial algebra $k[\alpha_1,...,\alpha_m]$.
\end{remark}

\begin{notn}
Let $k$ be a field, and let $U$ be a $k$-vector space equipped with an action of a finite group $W$. For any element $f \in U$, we denote by $sym(f):=\sum_{w\in W}w(f)$ the sum of the elements in the orbit of $f$.
\end{notn}

\begin{lemma}
\label{lemma: invariants}
We assume the notation in \Cref{rmk: invariants}.
Assume that $p>n_i, i=1,...,m$ and that $p\nmid |W|$. 
Then the ring of invariants $\Sym(V)^W$ is generated as a $k$-algebra by finitely many elements of the form $sym(v_i^{n_i})$ with $v_i\in V$.
\end{lemma}
\begin{proof}
Recall the polarization identity in combinatorics, see \cite[Ex.6.50(d)]{grinberg2020notes} for a proof. 
Namely, given any commutative ring $X$ and elements $x_i\in X,$ we have the following identity in $X$:
\[n!\cdot x_1...x_n=\sum_{I\subset \{1,...,n\}}(-1)^{n-|I|}\Big(\sum_{i\in I} x_i\Big)^n.\]
Therefore, for each $p> n,$ the image of the morphism $(\cdot)^n: V\to \Sym^n(V)$, which sends an element $v\in V$ to its $n$-th power $v^n$, spans $\Sym^n(V).$
It follows that each $\alpha_i$ can be written as linear combinations of the elements of the form $(v_i)^{n_i}$ for some $v_i\in V$.
Finally, we conclude by using $p\nmid |W|$ to write $\alpha_i=\frac{1}{|W|}sym(\alpha_i)$.
\end{proof}

In the following lemmata, we assume that $G$ is a split connected reductive group over a field $k.$
We fix a maximal split torus $T \subset G$ with corresponding Weyl group $W$. We also fix a set $\Phi^+$ of positive roots. 
For each dominant weight $\lambda\in X^+(T),$
let $L(\lambda)$ be the highest weight representation \cite[Prop. II.2.6]{jantzen_representations}.
Let $\phi_{\lambda}: G\to GL_N$ denote the morphism corresponding to the representation $L(\lambda).$
Let $\Phi_{\lambda}:\Fg\to \Fgl_N$ be the corresponding $k$-morphism of schemes of Lie algebras. 
Let $(\cdot)^n: \Fgl_N\to \Fgl_N$ be the morphism of $k$-schemes defined by $n$-th power of matrices.
Let $tr:\Fgl_N\to \Fgl_1$ be the $k$-morphism of schemes defined by taking the trace of matrices.
Let $tr(\Phi_{\lambda}^n):\Fg\to \Fgl_1$ be the morphism of $k$-schemes defined by the composition 
\[\Fg\xrightarrow{\Phi_{\lambda}} \Fgl_N\xrightarrow{(\cdot)^n}\Fgl_N\xrightarrow{tr}\Fgl_1. \]
Let $Lie(\lambda):\Ft\to \Fgl_1$ be the $k$-morphism of schemes defined by the weight $\lambda.$
Let $(Lie(\lambda))^n: \Ft\to \Fgl_1$ be the morphism of $k$-schemes defined by the composition $\Ft\xrightarrow{Lie(\lambda)} \Fgl_1\xrightarrow{(\cdot)^n}\Fgl_1$.
The natural action of $W$ on $T$ descends to $\Ft,$ so that we have a morphism of $k$-schemes $sym((Lie(\lambda))^n):=\sum_{w\in W}w((Lie(\lambda))^n):\Ft\to \Fgl_1$. 

\begin{lemma}
    There exist natural numbers $c(\lambda,\theta)$ which give the following identity of morphisms of $k$-schemes $\Ft\to \Fgl_1:$
    \begin{equation}
    \label{eqn: tr and sym}
tr(\Phi_{\lambda}^n)|_{\Ft}=sym((Lie(\lambda))^n)+\sum_{\theta\in X^+(T),\; \theta <\lambda} c(\lambda,\theta)sym((Lie(\theta))^n).
    \end{equation}
\end{lemma}
\begin{proof}
  Given any $\theta\in X^*(T),$ let $L(\lambda)_{\theta}$ be the corresponding weight space.
By \cite[Prop. II.2.4(b)]{jantzen_representations}, we have that $dim(L(\lambda)_{\lambda})=1.$
By \cite[\S II.1.19(1)]{jantzen_representations}, for any  $w\in W$, we have that $dim(L(\lambda)_{\theta})=dim(L(\lambda)_{w(\theta)}).$
We see that the equality (\ref{eqn: tr and sym}) holds on the level of $A$-points for any $k$-algebra $A.$ 
\end{proof}

     The following lemma provides a method to do change of groups for Hitchin bases. 

\begin{lemma}
\label{lemma: immersion of steinberg bases}
   Let $k$ be a field with $char(k)=p.$
   Let $G$ be a split reductive algebraic group over $k.$
   Assume either $p=0$ or $p>h(G)$.
   Then there is a homomorphism $\phi_G: G\to \prod_{i=1}^mGL_{N_i}$ such that the induced morphism on the Steinberg-Hitchin bases $\Xi_{\Fg}: \Fg\git G\hookrightarrow \prod_{i=1}^m \Fgl_{N_i}\git GL_{N_i}$
   is a closed immersion.
\end{lemma}

\begin{proof}
Choose a split maximal torus $T$. Let $\Ft^*$ be the dual of the $k$-vector space underlying $Lie(T).$
    The affine scheme $\Ft$ has function ring $Sym(\Ft^*).$
    The affine scheme $\Ft \git W$ has function ring $Sym(\Ft^*)^W.$
    Let $\{t_l''\}$ be a basis  of $\Ft^*$ defined by integral dominant characters $\tau_l''\in X^+(T).$
    In this case, the degrees $n_i$ in \Cref{rmk: invariants} can be checked at \cite[\S3.7, Table 1]{humphreys1992reflection}.
    We see that the maximal $n_i$ equals $h(G)$ and the product of $n_i$ equals $|W|$.
    Therefore, the assumptions of Lemma \ref{lemma: invariants} are satisfied and we have that 
    the ring $Sym(\Ft^*)^W$ is generated as a $k$-algebra by finitely elements of the form $sym(t_l^{n_l})$ with $t_l\in \Ft^*$ and $n_l<p,$ such that each $t_l=Lie(\tau_l)$ for some dominant integral weight $\tau_l$.
    For each $\tau_l,$ let $\theta(l)_{j},j=1,...,m_l$ be the dominant integral weights such that the natural number  $c(\tau_l, \theta(l)_j)$ in (\ref{eqn: tr and sym}) in positive. 
    Let $I$ be the finite set consists of all the dominant integral weights $\tau_l$'s and $\theta(l)_j$'s. 
    Then we define $\phi_G$ to be the product of the highest weight representations
    \[\phi_G:= \prod_{\lambda\in I}\phi_{\lambda}: G\to \prod_{\lambda\in I}GL(L(\lambda))=: \prod_{i=1}^m GL_{N_i}.\]
We are left to show that $\Xi_{\Fg}$ is a closed immersion.
For each $i=1,...,m,$ let $T_i$ be a split maximal torus containing the image of $T.$
    Let $W_i$ be the Weyl group for $(GL_{N_i},T_i).$
    Let $\Ft_i$ be the corresponding schemes of Lie algebras.
    By the Chevalley Isomorphism (\ref{eqn: proto hitchin base}), to conclude the proof it suffices to show that the induced morphism of affine $k$-schemes $\Xi_{\Ft}: \Ft \git W \hookrightarrow \prod_i \Ft_i \git W_i$
    is a closed immersion. 
    We check that the morphism on the function rings is surjective. 
    Let $\Fgl_i$ be the Lie algebra of $GL_{N_i}.$
    It suffices to show that each $sym(t_l^{n_l})$ can be lifted to a $\prod GL_{N_i}$-invariant function on $\prod \Fgl_i. $
    We do upward induction on the partial order on the $\tau_l$'s.
    When $\tau_l$ is minimal, by (\ref{eqn: tr and sym}), we have that $sym(t_l^{n_l})$ lifts to the
    $GL(L(\tau_l))$-invariant function $tr(\Phi_{\tau_l}^{n_l})$ on $\Fgl(L(\tau_l)).$
    When $\tau_l$ is not minimal, we use (\ref{eqn: tr and sym}) again to see that the function on the right-hand side $sym(t_l^{n_l})+o(<\lambda)$ lifts to the $GL(L(\tau_l))$-invariant function $tr(\Phi_{\tau_l}^{n_l})$ on $\Fgl(L(\tau_l)),$
    but the lower order term $o(<\lambda)$ already lifts by the induction assumption. 
    Therefore, the function $sym(t_l^{n_l})$ also lifts.
\end{proof}

Any homomorphism $\phi_G: G \to \prod_{i=1}^mGL_{N_i}$ as above induces a morphism between the corresponding Hitchin bases and moduli stacks of Higgs bundles. In the following lemma, we keep the set-up as in Lemma \ref{lemma: immersion of steinberg bases}, and let $L'=\Omega_{C'}^1(D')$ and $L^{\otimes p}:= F_{C/k}^* L'.$
\begin{lemma}
\label{lemma: cartesian sq hit base}
    We have the following Cartesian square of closed immersions of Hitchin bases:
    \begin{equation}
    \label{eqn: four embeddings}
        \xymatrix{
        A(C', G, L') \ar@{^{(}->}[r]^-{\Xi_{\Fg}}  \ar@{^{(}->}[d]_-{F_{X/k}^*}&
        \prod_{i=1}^m A(C', GL_{N_i}, L') \ar@{^{(}->}[d]^-{F_{X/k}^*} \\
        A(C,G,L^{\otimes p}) \ar@{^{(}->}[r]_-{\Xi_{\Fg}} &
        \prod_{i=1}^m A(C, GL_{N_i}, L^{\otimes p}).
        }
    \end{equation}
\end{lemma}
\begin{proof}
    By Lemma \ref{lemma: F pullback cl}, the vertical morphisms in (\ref{eqn: four embeddings}) are closed immersions induced by embeddings of $k$-vector spaces. 
    By Lemma \ref{lemma: immersion of steinberg bases} and  (\ref{eqn: proto hitchin base}), the twisted $\Xi_t,$
    \[\Xi_{\Ft}\times^{\bG_m} L: \Ft \git W\times^{\bG_m} L\hookrightarrow (\prod_{i=1}^m \Ft_i \git W)\times^{\bG_m} L\]
    is induced by an embedding of locally free sheaves: $\bigoplus_{i=1}^N L^{\otimes n_i}\hookrightarrow \bigoplus_{i=1}^m \bigoplus_{j=1}^{N_i} L^{\otimes j}.$
Therefore, the horizontal morphisms in (\ref{eqn: four embeddings}) are also closed immersions.
    The Cartesianness of (\ref{eqn: four embeddings}) follows from the identity of $k$-vector subspaces of $\bigoplus_{i=1}^m\bigoplus_{j=1}^{N_i}H^0(C, L^{\otimes pj}):$
    \[\Big(F_{C/k}^*\bigoplus_{i=1}^m\bigoplus_{j=1}^{N_i} H^0(C', (L')^{\otimes j})\Big)\cap \bigoplus_{i=1}^N H^0(C, L^{\otimes pn_i})= F_{C/k}^* \bigoplus_{i=1}^N H^0(C',(L')^{\otimes n_i}).\]
\end{proof}

\begin{proof}[Proof of \Cref{prop: hodge-hitchin}]
By a similar argument as in \cite[Appendix]{decataldo-herrero-nahpoles} we can assume that the triple $(Y,C_Y,E)$ is the pullback of a triple $(\hat{S}, \widetilde{C}, \widetilde{E}),$ where $\hat{S}$ is an integral affine scheme over $\mathbb{F}_p.$
Furthermore, after choosing a trivialization of $E$ etale locally over $C_Y,$ the $t$-connection $\nabla$ descends to $\widetilde{\nabla}$ on $\widetilde{E}.$ By Cartier descent \cite[Thm. 5.1]{katz1970nilpotent}, checking that $h(E,\Psi(\nabla))$ lands in $F^*A(C')$ is equivalent to checking that the section $\nabla^{can}(h(E,\Psi(\nabla)))$ vanishes, where $\nabla^{can}$ is the Cartier descent connection, so it is a local problem on $C_Y.$ Therefore, it suffices to check the factorization for all quadruples of the form $(\hat{S}, \widetilde{C},\widetilde{E}, \widetilde{\nabla}),$ where $\hat{S}$ is an integral affine Noetherian scheme over $\mathbb{F}_p,$ $\widetilde{C}$ is a quasi-projective smooth curve over $\hat{S}$, and $(\widetilde{E},\widetilde{\nabla})$ is a $G$-bundle with a $t$-connection.

Let $a:\widetilde{C}\to \hat{S}$ be the structure morphism.
We have that $\nabla^{can}(h(\widetilde{E},\Psi(\widetilde{\nabla})))$ is a section of $a_*(\bigoplus_i \Omega^1_{\widetilde{C}}(D)^{pn_i}),$
which is a locally free sheaf on $\hat{S}.$
Since $\hat{S}$ is integral, this section vanishes if it vanishes fiberwise over points on $\hat{S}.$
We are thus reduced to the case where $S$ is the spectrum of a field $k$ with characteristic $p>|W|.$

Let $\phi_G: G\to \prod_{i=1}^m GL_{N_i}$ be as in Lemma \ref{lemma: immersion of steinberg bases}.
It induces the following solid commutative diagram of $S$-stacks:
\[
\xymatrix{
        A(C', G, L') \ar@{^{(}->}[rrr]^-{\Xi_{\Fg}}  \ar@{^{(}->}[dd]_-{F_{C/k}^*}&&&
        \prod_{i=1}^m A(C', GL_{N_i}, L') \ar@{^{(}->}[dd]^-{F_{C/k}^*} \\
 & \hodge_G(C) \ar[r]^-{(\pi_G)_*} \ar[dl]_-{h\circ \Psi} \ar@{.>}[lu]^-{\widetilde{h_{Hod}}} &  \hodge_{\prod_i GL_{N_i}}(C) \ar[dr]^-{h\circ \Psi} \ar@{.>}[ur]_-{\widetilde{h_{GL}}}&
        \\
        A(C,G,L^{\otimes p}) \ar@{^{(}->}[rrr]_-{\Xi_{\Fg}} &&&
        \prod_{i=1}^m A(C, GL_{N_i}, L^{\otimes p}).
}
\]
By \cite[Cor. 5.7]{langer-moduli-lie-algebroids} and \cite[Appendix]{decataldo-herrero-nahpoles} the factorization marked by the dashed arrow $\widetilde{h_{GL}}$ exists.
Therefore, by the Cartesianness of the outer square as in \Cref{lemma: cartesian sq hit base}, the desired factorization $\widetilde{h_{Hod}}$ exists. 
\end{proof}

\subsection{Properness of the Hodge-Hitchin morphism}

\begin{lemma}
    The Hodge-Hitchin morphism $h_{Hod}$ is $\bG_m$-equivariant for the natural $\bG_m$-actions.
\end{lemma}
\begin{proof}
Lemma \ref{prop: g_m action and limits} shows that the natural morphism $\hodge_C(D)\to \mathbb{A}^1_S$ is $\bG_m$-equivariant.
 It suffices to show that all three other morphisms $h$, $\psi$, $F_{C/S}^*$ in (\ref{diagram: factorization}) are $\mathbb{G}_m$-equivariant. 
The Hitchin morphism $h$ is $\bG_m$-equivariant, because of the $\bG_m$-equivariance of the 
    natural morphism $\Fg/G\to \Fg\git G.$ On the other hand, the Frobenius pullback $F_{C/S}^*$ is by definition $\bG_m$-equivariant for the standard $\bG_m$-action on $A(C',\Omega^1_{C'/S}(D'))$ and the $p$-th power of the standard $\bG_m$-action on $A(C,\omega_{C/S}^1(D)^{\otimes p}).$
   Finally, we consider the $p$-curvature morphism $\Psi.$
    Let $Y$ be an $S$-scheme and $(E,\nabla_t)$ be a $t$-connection on $C_Y.$
    Let $t'\in \bG_m(Y).$
    Then $t\cdot (E,\nabla)$ is the $tt'$-connection given by $(E,t'\nabla_t).$
    For any derivation $\partial$ in $T_{C_Y/Y}(-D),$ we have 
    \[\Psi(t'\nabla_t)(\partial)=\Big( t'\nabla_t(\partial)\Big)^p-(t't)^{p-1}t'\nabla_t(\partial^{[p]})=(t')^p\Psi(\nabla_t)(\partial).\]
    Therefore, the morphism $\Psi$ is $\bG_m$-equivariant for the standard action on $\hodge_G(C)$ and the $p$-th power of the standard action on $\text{Higgs}(C, (\Omega_{C/S}^1(D))^{\otimes p}).$
\end{proof}

\begin{lemma}
\label{lemma: gm equi sch}
Let $G$ be a split reductive group satisfying properties (Em) and (Z). The $\bG_m$-action on the stack $(\hodge_{G,d})^{ss}$ induces a $\mathbb{G}_m$-action on its quasi-projective moduli space $M_{\mathrm{Hod}, G, d}^D$ constructed in \Cref{thm: moduli space for hodge general G}(2).
    Furthermore, assume that for each geometric point $\overline{s}$ of $S$, we have $p>h(G_{\overline{s}})$, then there is an induced Hodge-Hitchin morphism 
    \begin{equation}
    \label{eqn: hodge hitchin scheme}
        h_{Hod}: M_{\mathrm{Hod}, G, d}^D \to A(C',\Omega_{C'/S'}^1(D'))\times_S \mathbb{A}^1_S
    \end{equation}
    which is $\bG_m$-equivariant for the natural $\bG_m$-actions.
\end{lemma}

\begin{proof}
    Since taking the adequate moduli space commutes with flat base change \cite[Prop. 5.2.9]{alper_adequate}, the projection 
    \begin{equation}
    \label{eqn: hodge times bgm}
        (\hodge_{G,d})^{ss}\times_S\bG_m\to M_{\mathrm{Hod}, G, d}^D \times_S\bG_m
    \end{equation}
    is an adequate moduli space morphism. 
    Therefore, by the universal property of adequate moduli spaces, the composition
    \[\bG_m\times_S (\hodge_{G,d})^{ss}\xrightarrow{act} (\hodge_{G,d})^{ss}\to M_{\mathrm{Hod}, G, d}^D\]
    factors through (\ref{eqn: hodge times bgm}) and defines a $\mathbb{G}_m$-action $\bG_m\times_S M_{\mathrm{Hod}, G, d}^D\to M_{\mathrm{Hod}, G, d}^D$. 
    Using \Cref{prop: hodge-hitchin} and the universal property for the adequate moduli spaces $M_{\mathrm{Hod}, G, d}^D$ and $\bG_m\times_S M_{\mathrm{Hod}, G, d}^D$ again, we obtain the desired Hodge-Hitchin morphism (\ref{eqn: hodge hitchin scheme}) and its $\bG_m$-equivariance.  
\end{proof}

\begin{lemma}
\label{lemma: 0-fiber of hh and frobenius}
    There is an isomorphism $\alpha: A(C',\Omega^1_{C'/S}(D'))\xrightarrow{\sim} A(C,\Omega^1_{C/S}(D))'$ of $S$-schemes,
    where the latter is the Frobenius twist of the Hitchin base for $C.$
    Furthermore, the $0_{\mathbb{A}^1_S}$-fiber of the Hodge-Hitchin morphism admits the following factorization:
    \begin{equation}
    \label{diagram: 0 fiber of hh}
        \xymatrix{
        \Higgs_{G} \ar[d]_-{h_{Hod,0_{\mathbb{A}^1_S}}} \ar[r]^-{h} &
        A(C,\Omega_{C/S}^1(D)) \ar[d]^-{F_{A/S}} \ar[dl]_-{\sigma_C^*}\\
        A(C',\Omega_{C/S}^1(D')) \ar[r]_-{\alpha} &
        A(C,\Omega_{C/S}^1(D))'.
        }
    \end{equation}
\end{lemma}
\begin{proof}
    The Cartesian diagram of schemes
    \begin{equation*}
        \xymatrix{
        C' \ar[r]^-{\sigma_C}\ar[d]_-{\pi'} & C\ar[d]^-{\pi}\\
        S\ar[r]_-{F_S} & S. 
        }
    \end{equation*}
    induces isomorphisms of $S$-vector bundles $A(C',\Omega_{C'/S}^1(D'))\cong \pi'_* \sigma_C^* \Big((\Fg\git G)\times^{\bG_m} \Omega_{C}^1(D)\Big)$ and $A(C,\Omega_{C/S}^1(D))'\cong F_S^*\pi_* \Big((\Fg\git G)\times^{\bG_m} \Omega_C^1(D)\Big).$
    Therefore, the existence of the isomorphism $\alpha$ follows from flat base change \cite[\href{https://stacks.math.columbia.edu/tag/02KH}{Tag 02KH}]{stacks-project}.

    We now show the commutativity of (\ref{diagram: 0 fiber of hh}).
    Let $Y$ be an $S$-scheme.
    Let $s$ be a section of $(\Fg\git G)\times^{\bG_m}\Omega_{C_Y/Y}^1(D).$
    Then $\alpha^{-1}\circ F_{A/S}$ is given by the pullback by the morphism $\sigma_{C}\times F_Y: C_Y'=C'\times_S Y\to C\times_S Y=C_Y.$
    We thus have the commutativity of the lower triangle. 
    In view of \Cref{lemma: F pullback cl}, it suffices to show the commutativity of the upper triangle after composing with the closed immersion $F_{C/S}^*: A(C',\Omega_{C/S}^1(D'))\hookrightarrow A(C,\Omega_{C/S}^1(D)^{\otimes p}).$
    After the composition, the morphism that passes through $h$ becomes $F_C^*\circ h.$
    We are then reduced to show that given a Higgs bundle $(E,\phi)$ on $C_Y,$
    we have an equality of sections:
    \begin{equation}
    \label{eqn: frobenius pull back of higgs}
        F_C^*h(E,\phi)=h(E,\phi^p)\in A(C,\Omega_{C/S}^1(D)^{\otimes p})(Y).
    \end{equation}
    This can be checked \'etale locally on $C_Y$, so we are free to pass to an affine \'etale cover and fix trivializations of $\Omega_{C_Y/Y}^1(D)$ and $E$.
    Then $\phi$ is identified with a section $g$ of the trivial sheaf $\Fg\otimes_S\cO_{C_Y}.$
    The $p$-th power $\phi^p$ is identified with the section $(g)^p,$ where $()^p$ is the $p$-th power map in the $p$-Lie algebra $\Fg.$
    Let $t$ be a section of $\mathfrak{t}$ that has the same image as $g$ in $\Fg\git G\cong \Ft \git W$ (\ref{eqn: proto hitchin base}).
    Let $(t)^p$ be the $p$-th power of $t$ in the $p$-Lie algebra $\mathfrak{t}.$
    Since the actions of $W$ and $G$ commute with the $p$-th power map, the images of $(t)^p$ and $(g)^p$ coincide in $\Fg\git G.$
    Since $(\cdot)^p$ on $\mathfrak{t}$ is induced from Frobenius pullback, the desired equality (\ref{eqn: frobenius pull back of higgs}) follows.
\end{proof}

\begin{lemma}
\label{lemma: higgs universal homeo}
Suppose that $G$ satisfies property (LH).
    The stack $(\Higgs_{G,d})^{ss}$ of semistable meromorphic Higgs bundles of degree $d$ admits a quasi-projective moduli space $M_{\mathrm{Higgs}, G, d}^D$.
    Moreover, there is a canonical universal homeomorphism between quasi-projective schemes:
    \begin{equation*}
        M_{\mathrm{Higgs}, G, d}^D\xrightarrow{u} (M_{\mathrm{Hod}, G, d}^D)\times_{\mathbb{A}^1_S}0_S.
    \end{equation*}
\end{lemma}
\begin{proof}
  We have the isomorphism of $S$-stacks:
    \begin{equation}
        \label{eqn: higgs uni}
        (\Higgs_{G,d})^{ss}\cong (\hodge_{G,d})^{ss}\times_{M_{\mathrm{Hod}, G, d}^D} (M_{\mathrm{Hod}, G, d}^D\times_{\mathbb{A}^1_S}0_S).
    \end{equation}
    \cite[Prop. 5.2.9]{alper_adequate} then implies the existence of the adequate moduli space $M_{\mathrm{Higgs}, G, d}^D$ and the adequate (thus universal) homeomorphism $u$ as above. 
    The fact that $M_{\mathrm{Higgs}, G, d}^D$ is a quasi-projective scheme follows similarly as in the proof of \Cref{thm: moduli space for hodge general G}. 
\end{proof}

\begin{lemma}
\label{lemma: for the third step}
    The commutative diagram (\ref{diagram: 0 fiber of hh}) of $S$-stacks gives rise to the following commutative diagram of $S$-schemes:
     \begin{equation*}
        \xymatrix{
        M_{\mathrm{Higgs}, G, d}^D \ar[rr]^-{h}\ar[d]_-{u}&& A(C,\Omega_{C/S}^1(D))  \ar[d]_-{\sigma_C^*}\\
        (M_{\mathrm{Hod}, G, d}^D)\times_{\mathbb{A}^1_S} 0_{S} \ar[rr]_-{h_{Hod,0_{\mathbb{A}^1_S}}}  &&   A(C',\Omega_{C/S}^1(D')),
        }
    \end{equation*}
    where $u$ is the morphism defined in \Cref{lemma: higgs universal homeo}.
\end{lemma}
\begin{proof}
Via a diagram chase, we are reduced to showing that the natural morphism $h_{Hod, 0_{\mathbb{A}^1_S}}: M_{\mathrm{Higgs}, G, d}^D\to A(C')$ factors through  the morphism $u.$ The desired factorization follows from \cite[Prop. 5.2.9]{alper_adequate} and the uniqueness of the adequate moduli space, similarly as in the proof of \Cref{lemma: higgs universal homeo}.
\end{proof}

\begin{thm}[Properness of the Hodge-Hitchin morphism] \label{thm: properness of the hodge-hitchin morphism}
Let $G$ be a split reductive group satisfying property (LH).
    The Hodge-Hitchin morphism \eqref{eqn: hodge hitchin scheme} is proper.
\end{thm}
\begin{proof}
First, the Hodge-Hitchin morphism exists because of \Cref{prop: hodge-hitchin} and \Cref{rmk: lh implies coxeter}.
    By \Cref{prop: criterion for properness},
    it suffices to show that: (1) the $\bG_m$-action on $(\hodge_{G,d})^{sch}$ has zero limits (\Cref{defn: zero limits and contracting action});
    (2) the $\bG_m$-action on $A(C',\Omega^1_{C'/S}(D'))$ is contracting (\Cref{defn: zero limits and contracting action}); 
    (3) the base change $(\hodge_{G,d})^{sch}|_{o_{A}\times_S 0_{\mathbb{A}^1}}$ is proper over $o_{A}\times_{S}0_{\mathbb{A}^1}.$

    By \Cref{prop: g_m action and limits}, the $\bG_m$-action on the stack $\hodge_{G}$ has zero limits. 
    Therefore, by semistable reduction (\Cref{prop: semistable reduction}), the $\bG_m$-action on the semistable part $(\hodge_G)^{ss}$ has zero limits.
    Since the adequate moduli space morphism $(\hodge_{G,d})^{ss}\to M_{\mathrm{Hod}, G, d}^D$ is $\bG_m$-equivariant by \Cref{lemma: gm equi sch}, it follows that the $\bG_m$-action on $M_{\mathrm{Hod}, G, d}^D$ also has zero limits as in (1). 

    Under the isomorphism $\Fg\git G\cong \mathbb{A}_S^{N}$ as in (\ref{eqn: proto hitchin base}), the natural $\bG_m$-action is contracting since it is of positive weight on each $\mathbb{A}_S^1$-factor. 
    Using trivializations Zariski locally over $C'$, we see that the $\bG_m$-action on the twist $(\Fg\git G)\times^{\bG_m}\Omega_{C'/S}^1(D')$ is also contracting.
    The Hitchin base $A(C',\Omega_{C/S}^1(D'))$ is the pushforward of $(\Fg\git G)\times^{\bG_m}\Omega_{C'/S}^1(D')$ to $S,$ thus the $\bG_m$-action on it is also contracting as in (2).

    Finally, we show (3).
    By \Cref{lemma: for the third step} and \Cref{lemma: 0-fiber of hh and frobenius}, it suffices to show that the Hitchin morphism $h: M_{\mathrm{Higgs}, G, d}^D\to A(C,\Omega_{C/S}^1(D))$ is proper, which is established in \cite[Cor. 6.21, Rmk. 6.22]{alper2019existence}.
\end{proof}

\end{subsection}
\end{section}

 \appendix

\begin{section}{Compatible parabolic reductions and affine schemes with connections} \label{appendix: D-affine schemes}
In this section we provide results needed for the proof of \Cref{lemma: the instability reduction is compatible with the connection}. For simplicity, we assume in this section that $S = \Spec(k)$ for some algebraically closed field $k$.

\begin{defn}
\label{defn: conn on aff sch}
Let $t \in k$, and let $X \to C$ be a relatively affine scheme with coordinate ring $\cO_X$. A meromorphic $t$-connection on $X$ is the data of a $k$-linear morphism of sheaves $\nabla: \cO_X(-D) \to \cO_X \otimes_{\cO_{C}} \Omega^1_{C/k}$ satisfying the $t$-Leibniz rule (as in \Cref{example: meromorphic connection vector bundles}) and the multiplicative Leibniz rule $\nabla(f \cdot h) = h \cdot \nabla(f) + f \cdot \nabla(h)$ for all local sections $f, h \in H^0(U , \cO_X)$ over an open $U \subset C$.

For any two affine $C$-schemes $X, Y$ equipped with meromorphic $t$-connections $\nabla_X, \nabla_Y$, a morphism $f:X \to Y$ of schemes is a morphism of schemes with meromorphic $t$-connections if the following diagram commutes
\[
\begin{tikzcd}
    \cO_X(-D)  \ar[r, "f^*(-D)"] \ar[d,"\nabla_X"] & \cO_Y(-D) \ar[d, "\nabla_Y"] \\
   \cO_X \otimes_{\cO_C} \Omega^1_{C/k} \ar[r, "f^* \otimes \text{id}"] & \cO_Y \otimes_{\cO_C} \Omega^1_{C_k/k}
\end{tikzcd}
\]
\end{defn}

\begin{example}[Product connection]
\label{example: tensor connection}
    If $X,Y$ are two affine $C$-schemes equipped with meromorphic $t$-connections $\nabla_X, \nabla_Y$, then the fiber product $X \times_C Y$ is naturally equipped with a tensor connection $\nabla_X \otimes \nabla_Y$ which is given by $(\nabla_X\otimes \nabla_Y)(f\otimes h)=\nabla_X(f)\otimes h+ f\otimes \nabla_Y(h)$ for any local sections $f\in H^0(U,\cO_X)$ and $h\in H^0(U,\cO_Y)$ over an open $U\subset C.$
\end{example}

\begin{example}[Trival connection] \label{example: trivial connection on affine scheme}
    Let $X$ be an affine $k$-scheme. Choose a basis of $\cO_X$ as a $k$-vector space, inducing a decomposition $\cO_X = \bigoplus_{i \in I} k e_i$. The base change $X_C$ is equipped with a canonical meromorphic $t$-connection defined as follows 
    \[ \cO_{X_C}(-D) \cong \bigoplus_{i \in I} \cO_C(-D) \hookrightarrow \bigoplus_{i \in I} \cO_C \xrightarrow{d^{\oplus I}} \bigoplus_{i \in I} \Omega^1_{C/k} \cong \cO_X \otimes_{\cO_C} \Omega^1_{C/k} \xrightarrow{t \cdot(-)} \cO_X \otimes_{\cO_C} \Omega^1_{C/k} \]
    We call this the trivial meromorphic $t$-connection. It does not depend on the choice of basis.
\end{example}

\begin{example}[Associated connections] \label{example: associated connection on affine schemes}
    Let $X$ be an affine $k$-scheme with an action of $G$. Let $(E, \nabla)$ be a $G$-bundle with a meromorphic $t$-connection. Then the associated fiber bundle $E(X) \to C$ is equipped with a meromorphic $t$-connection as follows. Write $\cO_X$ as a union of finite $G$-submodules $\cO_X = \bigcup_{i \in I} V_i$, so that we have homomorphisms $\rho_i : G \to \GL(V_i)$. By the functoriality of connections (as in Subsection \ref{subsection: change of groups}) and the description of connections for $\GL(V_i)$-bundles (as in \Cref{example: meromorphic connection vector bundles}), each associated bundle $E(V_i)$ is equipped with a meromorphic $t$-connection $\nabla_i: E(V_i)(-D) \to E(V_i) \otimes \Omega^1_{C/k}$. These connections fit together to yield a meromorphic $t$-connection $\nabla_{E(X)}$ on $\cO_{E(X)} = \bigcup_{i \in I} E(V_i)$. Since the multiplication morphism on $\cO_X$ is a morphism of $G$-modules, it follows that $\nabla_{E(X)}$ satisfies the multiplicative Leibniz rule.
    
    We call this the associated meromorphic $t$-connection on $E(X)$. If the action of $G$ on $X$ is trivial, then we recover the trivial meromorphic $t$-connection on $X_C$. This construction is functorial: for any equivariant morphism of affine $k$-schemes $X \to Y$ equipped with $G$-actions, the associated morphism $E(X) \to E(Y)$ is a morphism of affine schemes with meromorphic $t$-connections.
\end{example}

\begin{example} \label{example: induced connection in the total space of G-bundle}
    Let $(E, \nabla)$ be a $G$-bundle equipped with a meromorphic $t$-connection in the sense of \Cref{defn: meromorphic t-conn}. Then the total space of the bundle $E \to C$ is equipped with a meromorphic $t$-connection $\nabla_E$ as in \Cref{defn: conn on aff sch}. Indeed, we can apply \Cref{example: associated connection on affine schemes} to the action of $G$ on itself by multiplication. 
\end{example}

\begin{defn} \label{defn: closed subscheme preserved by connection}
    Let $X \to C$ be a relatively affine $C$-scheme equipped with a meromorphic $t$-connection $\nabla$. We say that a closed subscheme $Y \subset X$ is preserved by $\nabla$ if there is a (unique) meromorphic $t$-connection $\nabla_Y$ such that the inclusion $Y \hookrightarrow X$ is a morphism of schemes with meromorphic $t$-connections.

    Equivalently, if $\cI_Y \subset \cO_X$ is the ideal sheaf of $Y$, then $Y$ is preserved by $\nabla$ if $\nabla(\cI_Y(-D)) \subset \cI_Y \otimes_{\cO_C} \Omega^1_{C/k}$. The connection induced on the quotient $\cO_Y = \cO_X / \cI_Y$ is the unique meromorphic $t$-connection on $Y$ compatible with the inclusion $Y \hookrightarrow X$.
\end{defn}

The following lemma is a direct consequence of \Cref{defn: closed subscheme preserved by connection}.
\begin{lemma} \label{lemma: compatibility of fiber products with closed subschemes preserved by connections}
    Let $f: X \to Y$ be a morphism of relatively affine $C$-schemes equipped with meromorphic $t$-connections $\nabla_X, \nabla_Y$. Let $Z \hookrightarrow Y$ be a closed subscheme that is preserved by the connection $\nabla_Y$. Then the fiber product $Z\times_Y X \hookrightarrow X$ is preserved by the connection $\nabla_X$. \qed
\end{lemma}

\begin{example} \label{example: inclusion of associated affine schemes is compatible with connections}
Let $X$ be an affine $k$-scheme equipped with a $G$-action. Let $Y$ be a closed subscheme $Y \subset X$ that is preserved by the $G$-action. Let $(E, \nabla)$ be a $G$-bundle on $C$ equipped with a meromorphic $t$-connection. Then the closed subscheme $E(Y) \subset E(X)$ is preserved by the associated meromorphic $t$-connection on $E(X)$.
\end{example}

\begin{lemma} \label{lemma: compatiblity of parabolic reuductions in terms of affine schemes with connection}
    Let $(E, \nabla)$ be a $G$-bundle on $C$ equipped with a meromorphic $t$-connection. Let $P \subset G$ be a parabolic subgroup, and choose a parabolic reduction $E_P$. Then $E_P$ is $\nabla$-compatible if and only if the closed immersion of total spaces $E_P \hookrightarrow E$ is preserved by the induced meromorphic $t$-connection $\nabla_E$ (as in \Cref{example: induced connection in the total space of G-bundle}).
\end{lemma}
\begin{proof}
    Choose a one-parameter subgroup $\lambda: \mathbb{G}_m \to G$ such that $P = P_{\lambda}$ is the associated parabolic. For any representation $V$ of $G$, we get an induced filtration of $P$-representations indexed by the integers
    \[ \ldots V^{\lambda \geq 2} \subset V^{\lambda \geq 1} \subset V^{\lambda \geq 0} \subset V^{\lambda \geq -1} \ldots \subset V\]
    In particular, we have a filtration by $P$-representations for the regular representation $\cO_G$. By taking associated bundles for $E_P$, we get a filtration of $\cO_C$-modules
      \begin{equation} \label{eqn: filration of O_E}
      \ldots \cO_E^{\lambda \geq 2} \subset \cO_E^{\lambda \geq 1} \subset \cO_E^{\lambda \geq 0} \subset \cO_E^{\lambda \geq -1} \ldots \subset \cO_E
      \end{equation}
      The ideal sheaf of $E_P \subset E$ is the submodule $\cO_E^{\lambda \geq 1}$, which by assumption is preserved by the meromorphic $t$-connection $\nabla_E$.
      
      \noindent \textbf{Claim.} For all integers $j$, the subsheaf
      $\cO_E^{\geq j}$ is preserved by $\nabla_E$. 
      
      The one-parameter group $\lambda: \mathbb{G}_m \to G$ induces a grading $\cO_G = \bigoplus_{n \in \mathbb{Z}} \cO_G^{\lambda = n}$. After possibly scaling the weights, we shall assume without loss of generality that $\cO_G^{\lambda = -1} \neq 0$. Using this and the fact that $\cO_G$ is an integral domain, we get the following equalities for all $j \geq 0$
      \[ \cO_G^{\lambda \geq -j} = \left\{ f \in \cO_G \; \mid \; f \cdot h \in \cO_G^{\lambda \geq -j+1} \text{ for all $h \in \cO_G^{\lambda \geq 1}$}\right\}\]
      \[ \cO_G^{\lambda \geq j} = \left\{ f \in \cO_G \; \mid \; f \cdot h \in \cO_G^{\lambda \geq 1} \text{ for all $h \in \cO_G^{\lambda \geq -j+1}$}\right\}\]
      By taking associated bundles for $E_P$, we get the following equalities (to be interpreted in terms of local sections over opens $U \subset C$)
      \begin{equation} \label{eqn: first equality parabolic lemma}
          \cO_E^{\lambda \geq -j} = \left\{ f \in \cO_E \; \mid \; f \cdot h \in \cO_E^{\lambda \geq -j+1} \text{ for all $h \in \cO_E^{\lambda \geq 1}$}\right\}
      \end{equation}
      \begin{equation} \label{eqn: second equation parabolic lemma}
          \cO_E^{\lambda \geq j} = \left\{ f \in \cO_E \; \mid \; f \cdot h \in \cO_E^{\lambda \geq 1} \text{ for all $h \in \cO_E^{\lambda \geq -j+1}$}\right\}
      \end{equation}
      Since $\nabla_E$ is compatible with multiplication, the first equation \eqref{eqn: first equality parabolic lemma} implies that if $\cO_E^{ \lambda \geq -j+1}$ is preserved by $\nabla_E$, then $\cO_E^{ \lambda \geq -j}$ is also preserved. Using induction and $j=0$ for the base case, we conclude that $\cO_E^{\lambda \geq -j}$ is preserved by $\nabla_E$ for all $j \geq -1$. But then using compatibility of multiplication and the second equation \eqref{eqn: second equation parabolic lemma}, we get that $\cO_E^{\lambda \geq j}$ is also preserved for $j \geq -1$, thus concluding the proof of the \textbf{Claim}.

      Now, choose a finite dimensional representation $V \subset \cO_G$ such that the induced morphism $\rho: G \to \GL(V)$ is a closed immersion. Let $Q_{\lambda} \subset \GL(V)$ be the parabolic subgroup associated to the one-parameter subgroup $\mathbb{G}_m \xrightarrow{\lambda} G \xrightarrow{\rho} \GL(V)$. By construction, we have $P = G \cap Q_{\lambda}$. The homomorphism $\rho^P: P \to Q_{\lambda}$ induces a $Q_{\lambda}$-reduction $\rho^P_*(E_P)$ of the associated vector bundle $E(V)$. This parabolic reduction corresponds to the nontrivial steps in the filtration
      \[ 0 \subset \ldots E(V)^{\lambda \geq 2} \subset E(V)^{\lambda \geq 1} \subset E(V)^{\lambda \geq 0} \subset E(V)^{\lambda \geq -1} \ldots \subset E(V) \]
      obtained by intersecting $E(V) \subset \cO_E$ with the corresponding filtration \eqref{eqn: filration of O_E} of $\cO_E$. By the \textbf{Claim}, every subbundle $E(V)^{\lambda \leq j} \subset E(V)$ is preserved by the associated connection $\rho_*(\nabla) = \nabla_E|_{E(V)}$. It follows that the $Q_{\lambda}$-reduction $\rho_*(E_P)$ of $E(V)$ is $\rho_*(\nabla)$-compatible (see \Cref{example: semistability of meromorphic conenctions on vector bundles}). The equality $\text{Lie}(P) = \text{Lie}(G) \cap \text{Lie}(Q_{\lambda})$ implies $At^D(E_P) = At^D(E) \cap At^D(\rho_*(E_P))$ for the meromorphic $t$-Atiyah bundles. Since the reduction $\rho_*(E_P)$ is $\rho_*(\nabla)$-compatible, it follows by \Cref{defn: compatible parabolic reduction} that $E_P$ is $\nabla$-compatible.
\end{proof}

Next, we express compatibility of parabolic reductions in terms of the associated section $s: C \to E(G/P)$. Since $G/P$ is not affine, we consider its affine cone.

\begin{notn}
\label{notn: compatibility of section}
Let $\chi$ denote a $P$-dominant character of the parabolic subgroup $P \subset G$. Let $\cL = \cO(-\chi)$ denote the corresponding ample $G$-equivariant bundle on $G/P$. We denote by $X = \Spec( \bigoplus_{n \geq 0} H^0(\cL^{\otimes n}))$ the affine cone, equipped with its natural $G$-action. The grading on $\cO_X$ induces an action of $\mathbb{G}_m$. We have a distinguished point $0 \in X^P$ and $\mathbb{G}_m$ acts freely on the open complement $X \setminus 0$. By construction $G/P = (X\setminus 0) / \mathbb{G}_m$.

For any line bundle $\cM$ on $C$, we denote by $\cM^{\times} = \Spec_C( \bigoplus_{ n \in \mathbb{Z}} \cM^{\otimes n})$ the total space of the associated $\mathbb{G}_m$-bundle. Let $E$ be a $G$-bundle on $C$. By thinking of $X$ as a $G \times \mathbb{G}_m$-scheme, we can form the associated fiber bundle $E(X)_{\cM} : = \Spec_C( \bigoplus_{ n \geq 0} E(H^0(\cL^{\otimes n})) \otimes \cM^{\otimes n})$. If $E$ and $\cM$ are equipped with $t$-meromorphic connections $\nabla, \nabla_{\cM}$, then $E(X)_{\cM}$ comes equipped with an associated connection (see \Cref{example: associated connection on affine schemes}), which can be described on each graded piece of $\cO_{E(X)_{\cM}}$ by using the tensor connection $\nabla_{E(H^0(\cL^{\otimes n}))} \otimes \nabla_{\cM}^{\otimes n}$.

If the $G$-bundle $E$ is trivial, we use the simpler notation $X_{\cM} : = \Spec_C( \bigoplus_{ n \geq 0} H^0(\cL^{\otimes n}) \otimes_k \cM^{\otimes n})$. If the line bundle $\cM$ is equipped with a meromorphic $t$-connection $\nabla_{\cM}$, then the affine $C$-scheme $X_{\cM}$ comes naturally equipped with a meromorphic $t$-connection obtained by using on each graded component the tensor power connection $\nabla_{\cM}^{\otimes n}$.
\end{notn}

Let $(E, \nabla)$ be a $G$-bundle equipped with a meromorphic $t$-connection. Parabolic reductions of $E$ correspond to sections of $s: C \to E(G/P)$. If we set $\cM := s^*(\cL)$, then we obtain an induced $\mathbb{G}_m$-equivariant morphism of $C$-schemes
\[\widetilde{s}: C \to  \cM^{\times} (E(X \setminus 0)) \hookrightarrow E(X)_{\cM}\]
If $E_P$ is the corresponding parabolic reduction, then by definition $\cM = E_P(-\chi)$. In particular, if $E_P$ is $\nabla$-compatible, then the associated line bundle $\cM$ comes equipped with an associated meromorphic $t$-connection, and hence the target $E(X)_{\cM}$ is equipped with a meromorphic $t$-connection. We equip $C$ with the trivial meromorphic $t$-connection (see \Cref{example: trivial connection on affine scheme}).

\begin{lemma}
    With notation as above, if the parabolic reduction $E_P$ is $\nabla$-compatible, then the associated morphism $\widetilde{s}: C \to E(X)_{\cM}$ is compatible with the meromorphic $t$-connections (i.e., it is a morphism of affine $C$-schemes with meromorphic $t$-connections).
\end{lemma}
\begin{proof}
   The data of $(E, \nabla)$ comes from the parabolic bundle with connection $(E_P, \nabla_P)$. If we view $X$ as a $P \times \mathbb{G}_m$-scheme, the corresponding meromorphic connection on $E(X)_{\cM} = E_P(X)_{\cM}$ is induced by $\nabla_P \times \nabla_{\cM}$. Let $p \in G/P(k)$ denote the identity point. For any lift $\widetilde{p}$ of $p$ to the affine cone $X$, the $\mathbb{G}_m$-orbit $\mathbb{G}_m \cdot \widetilde{p}$ can be completed to a $P \times \mathbb{G}_m$-equivariant line $L_{\widetilde{p}} \cong \mathbb{A}^1_k \subset X$. Here $P$ acts linearly via the character $\chi$, and $\mathbb{G}_m$ acts linearly with weight $1$. Since $\cM = E_P(-\chi)$, it follows that the associated closed subscheme $(E_P\times_C \cM^{\times})(L_{\widetilde{p}}) \subset E(X)_{\cM}$ is isomorphic to $\mathbb{A}^1_C$ equipped with the trivial meromorphic $t$-connection. By the construction of $\widetilde{s}$, we have a factorization
   \[ \widetilde{s}: C \xrightarrow{1_C} \mathbb{A}^1_C \cong (E_P\times_C \cM^{\times})(L_{\widetilde{p}}) \hookrightarrow E(X)_{\cM}\]
   where the first morphism is given by the unit section $1_C$. This inclusion $1_C$ of trivial schemes is compatible with the meromorphic $t$-connections (cf. \Cref{example: trivial connection on affine scheme}), and the inclusion of associated $P \times \mathbb{G}_m$-schemes $(E_P\times_C \cM^{\times})(L_{\widetilde{p}}) \hookrightarrow E(X)_{\cM}$ is also compatible with the corresponding meromorphic $t$-connections (see \Cref{example: inclusion of associated affine schemes is compatible with connections}). It follows that $\widetilde{s}$ is compatible with the connections.
\end{proof}

The action morphism of $G \times \mathbb{G}_m$ on $X$ induces a $G\times \mathbb{G}_m$ equivariant morphism $E \times_C \cM^{\times} \times_C E(X)_{\cM} \to X_C$, where on the left $G \times \mathbb{G}_m$ acts on the torsor $E \times_C \cM^{\times}$. In the discussion above, both $E$ and $\cM$ are equipped with meromorphic $t$-connections, and in that case $E \times_C \cM^{\times} \times_C E(X)_{\cM} \to X_C$ is a morphism of schemes with meromorphic $t$-connections (here $X_C$ is equipped with the trivial connection). Composing with the section $\widetilde{s}: C \to E(X)_{\cM}$, we obtain a $G \times \mathbb{G}_m$-equivariant morphisms of schemes with meromorphic $t$-connections
\[ v: E \times_C \cM^{\times} \xrightarrow{\text{id} \times \widetilde{s}} E \times_C \cM^{\times} \times_C E(X)_{\cM} \to X_C\]
\end{section}

\begin{section}{A criterion for properness of equivariant morphisms}
In this section we prove a criterion for properness between schemes equipped with a $\mathbb{G}_m$-action. We work over a fixed locally Noetherian base scheme $S$. For any ring $R$, we denote by $1_R: \Spec(R) \to \mathbb{A}^1_R$ the closed immersion corresponding to setting $t=1$.

\begin{defn} \label{defn: zero limits and contracting action}
    Let $Y$ be a separated scheme locally of finite type over $S$. We say that an action of $\mathbb{G}_{m,S}$ on $Y$ \underline{has zero limits} if for all algebraically closed field $K$ over $S$, every $S$-morphism $f:\Spec(K) \to Y$ can be extended to a $\mathbb{G}_{m,K}$-equivariant $K$-morphism $\widetilde{f}: \mathbb{A}^1_K \to Y_K$ such that $\widetilde{f} \circ 1_K = f$.

We say that the action is \underline{contracting} if for any discrete valuation ring $R$ over $S$, every $S$-morphism $f:\Spec(R) \to Y$ can be extended to a $\mathbb{G}_{m,R}$-equivariant $R$-morphism $\widetilde{f}: \mathbb{A}^1_R \to Y_R$ such that $\widetilde{f} \circ 1_R =f$.
\end{defn}
If $Y$ has a contracting action of $\mathbb{G}_{m,S}$, then it has zero limits. This can be seen by considering the power series ring $R = K \bseries{x}$.

\begin{example}
The scheme $\mathbb{A}^1_S$ admits a standard linear contracting action of $\mathbb{G}_{m,S}$. We equip $\mathbb{A}^1_S$ with this action unless otherwise stated.
\end{example}

\begin{example}
For any field $k$, the standard nontrivial action of $\mathbb{G}_{m,k}$ on $\mathbb{P}^1_k$ has zero limits, but it is not contracting.
\end{example}

\begin{notn}
For any separated scheme $Y$ locally of finite type over $S$, we denote by $Y^{\mathbb{G}_m} \hookrightarrow Y$ the closed subscheme of $\mathbb{G}_{m,S}$-fixed points (cf. \cite[Prop. 1.4.1]{halpernleistner2018structure}).
\end{notn}

\begin{prop}
\label{prop: criterion for properness}
    Let $X$ and $Y$ be schemes locally of finite type over $S$ equipped with $\mathbb{G}_{m,S}$-actions. Let $f: X \to Y$ be a $\mathbb{G}_{m,S}$-equivariant separated morphism of finite type. Suppose that
    \begin{enumerate}[(1)]
        \item The action on $X$ has zero limits, and the action on $Y$ is contracting. 
        \item The base change $X \times_{Y} Y^{\mathbb{G}_m} \to Y^{\mathbb{G}_m}$ is proper.
    \end{enumerate}
    Then the morphism $f$ is proper.
\end{prop}
\begin{proof}
    It suffices to show that $f$ is universally closed. Let $R$ be a discrete valuation ring equipped with a morphism $\Spec(R) \to Y$. We denote by $s$ (resp. $\eta$) the special (resp. generic) point of $\Spec(R)$. For any lift $g: \eta \to X\times_{Y} \Spec(R)$ of the inclusion $\eta \hookrightarrow \Spec(R) \to Y$, we want to show that the closure $\overline{g(\eta)} \subset X \times_{Y} \Spec(R)$ intersects nontrivially the special fiber $X\times_{Y} s$. By the assumption that the action on $Y$ is contracting, we can extend $\Spec(R) \to Y$ to a $\mathbb{G}_{m,R}$-equivariant morphism $\mathbb{A}^1_R \to Y_R$. Since the base change $X \times_Y \mathbb{A}^1_R \to \mathbb{A}^1_R$ equipped with the diagonal $\mathbb{G}_{m,R}$ action satisfies the hypotheses (1) and (2), we can assume without loss of generality that $S = \Spec(R)$ and $Y = \mathbb{A}^1_R$. We have a morphism $1_R: \Spec(R) \to \mathbb{A}^1_R$ and a lift $g: \eta \to X$ of $1_{\eta}$, and we want to show that the closure $\overline{g(\eta)} \subset X$ intersects the fiber $X_{1_s}$. 

Let $Or(1_{\eta}) = \mathbb{A}^1_{\eta} \setminus 0_{\eta} \to \mathbb{A}^1_{R}$ be the $\mathbb{G}_{m,\eta}$-orbit of $1_{\eta}$ in $\mathbb{A}^1_{\eta}$. We use similar notation for the orbits $Or(1_s) = \mathbb{A}^1_s \setminus 0_s \to \mathbb{A}^1_R$ and $Or(g(\eta)) \subset X_{\mathbb{A}^1_{\eta}}$. 

    \noindent \textbf{Claim:} The closure $\overline{Or(g(\eta))} \subset X$ intersects $X_{\mathbb{A}^1_{s} \setminus 0_s}$ nontrivially.

    Replacing $X$ with $\overline{Or(g(\eta))}$, we can assume that $X$ is integral and $f: X \to \mathbb{A}^1_R$ is an isomorphism over $Or(1_{\eta})$. 
    We want to show that the image of $f$ intersects the orbit $Or(1_s) = \mathbb{A}^1_s \setminus 0_s \subset \mathbb{A}^1_R$ of $1_s$. 
    Let $R\bseries{t}$ denote the completion of the ring $R[t] = \mathcal{O}_{\mathbb{A}^1_R}$ at the ideal $I = (t)$ of the closed subscheme $0_R = (\mathbb{A}^1_R)^{\mathbb{G}_m} \subset \mathbb{A}^1_R$. 
    The completion morphism $c: \Spec(R\bseries{t}) \to \mathbb{A}^1_R$ is flat. 
    Since the set-theoretic image of $c$ contains the special point $0_s$, it also contains the generic point of the special orbit $Or(1_s)$, which is a generalization of $0_s$. 
    In particular, in order to see that the image of $f: X \to \mathbb{A}^1_R$ intersects $Or(1_s)$, it suffices to show that the base change $X':= X \times_{\mathbb{A}^1_R} \Spec(R\bseries{t}) \to \Spec(R\bseries{t})$ is surjective. 
    Taking the scheme-theoretic closure of a quasi-compact morphism commutes with flat base change, and so $X'$ remains the scheme-theoretic closure of the base change of the generic point of $\mathbb{A}^1_R$, which is just the field of fractions of $R\bseries{t}$. In particular, the base change $X'$ remains integral. The pair $\Spec(R\bseries{t})$ and $0_R \subset \Spec(R\bseries{t})$ is Henselian. 
    By the hypothesis (2), we know that the base change $X'_{0_R} \to 0_R$ is proper. Moreover, by the hypothesis (1) that the action on $X$ zero limits, the closure $\overline{Or(g(\eta)})$ contains a point in $X_{0_{\eta}}$, and so $X'_{0_R}$ is nonempty. 
    By \cite[\href{https://stacks.math.columbia.edu/tag/0CT9}{0CT9}]{stacks-project}, there is a nonempty open and closed subscheme $U \subset X'$ that is proper over $\Spec(R\bseries{t})$. Since $X'$ is integral, we must have $U = X'$, so $X'$ is proper over $\Spec(R\bseries{t})$. 
    Since we already know that the image of $X' \to \Spec(R\bseries{t})$ contains the generic point of $\Spec(R\bseries{t})$, it follows that $X' \to \Spec(R\bseries{t})$ is surjective, thus concluding the proof of the \textbf{Claim}.
    
     The fiber $F := \overline{Or(g(\eta))} \cap X_{\mathbb{A}^1_R \setminus 0_R}$ is a $\mathbb{G}_{m,R}$-stable integral closed subscheme of $X_{\mathbb{A}^1_R \setminus 0_R} = X_{\mathbb{G}_{m,R}}$ with generic fiber $Or(g(\eta))$ over $Or(1_{\eta})$. Notice that $\mathbb{G}_{m,R}$ acts freely on $X_{\mathbb{G}_{m,R}}$. The morphism $f: X_{\mathbb{G}_{m,R}} \to \mathbb{G}_{m,R}$, induces a $\mathbb{G}_{m,R}$-invariant morphism $X_{\mathbb{G}_{m,R}} \to X_{1_R}$ given, at the level of scheme-valued points, by $x \mapsto f(x)^{-1} \cdot x$. This induces a canonical identification $X_{\mathbb{G}_{m,R}}/\mathbb{G}_{m,R} \cong X_{1_R}$. Taking the quotient of the inclusion $F \subset X_{\mathbb{G}_{m,R}}$ we get a closed integral subscheme $F/\bG_{m,R}\subset X_{\bG_{m,R}}/\bG_{m,R} \cong X_{1_R}$ with generic fiber $g(\eta) \subset X_{1_\eta}$. Hence $F/\mathbb{G}_{m,R}$ must be the closure $\overline{Or(g(\eta))} \subset X_{1_{R}}$. By the \textbf{Claim} above, $F/\mathbb{G}_{m,R}$ intersects the special fiber $X_{1_s}$ nontrivially, as desired.
\end{proof}

\end{section}

\footnotesize{\bibliography{hodge_moduli_poles.bib}}
\bibliographystyle{alpha}

  \textsc{Department of Mathematics, University of Pennsylvania,
209 South 33rd Street,
Philadelphia, PA 19104, USA}\par\nopagebreak
  \textit{E-mail address}: \texttt{andresfh@sas.upenn.edu}

    \textsc{School of Mathematics, Institute for Advanced Study, Princeton, NJ 08540,
USA}\par\nopagebreak
  \textit{E-mail address}: \texttt{szhang@ias.edu}
\end{document}